\documentclass[a4paper,12pt,reqno]{amsart}

\usepackage{amsmath, amssymb, amsthm, fullpage}
\usepackage{bbm}
\usepackage{enumerate, enumitem}
\usepackage{hyperref}
\usepackage{verbatim}
\usepackage{color}

\newtheorem{thm}{Theorem}[section]
\newtheorem{lemma}[thm]{Lemma}
\newtheorem{fact}[thm]{Fact}
\newtheorem{cor}[thm]{Corollary}
\newtheorem{prop}[thm]{Proposition}
\newtheorem{conj}[thm]{Conjecture}
\newtheorem{obs}[thm]{Observation}

\theoremstyle{definition}
\newtheorem{defn}[thm]{Definition}
\newtheorem{quest}[thm]{Question}
\newtheorem*{claim}{Claim}
\newtheorem{claimGnp}{Claim}

\newcommand{\Aut}{\mathrm{Aut}}
\newcommand{\Find}{\mathcal{F}^{\mathrm{ind}}}
\newcommand{\Fn}{\Find_n}
\newcommand{\Fnm}{\Find_{n,m}}
\newcommand{\FnmCf}{\Fnm(C_4)}
\newcommand{\Gnpind}{G _{n,p}^{\mathrm{ind}}}

\newcommand{\ellnm}{\ell_{n,m}}
\newcommand{\Nnm}{N_{n,m}}

\newcommand{\Snm}{\mathcal{S}_{n,m}}

\newcommand{\Ex}{\mathbb{E}}
\renewcommand{\Pr}{\mathbb{P}}

\newcommand{\scolon}{:}

\newcommand\cP{\mathcal{P}}
\newcommand{\Pnm}{\cP_{n,m}}

\def\Q{\mathcal{Q}}

\def\Bin{\textup{Bin}}

\newcommand{\indicator}{\mathbbm{1}}

\newcommand{\eps}{\varepsilon}

\newcommand{\ex}{\mathrm{ex}}

\newcommand{\cA}{\mathcal{A}}

\newcommand{\CC}{\mathcal{C}}
\newcommand{\FF}{\mathcal{F}}
\newcommand{\GG}{\mathcal{G}}
\newcommand{\HH}{\mathcal{H}}
\newcommand{\II}{\mathcal{I}}
\newcommand{\N}{\mathbb{N}}
\newcommand{\cS}{\mathcal{S}}
\newcommand{\T}{\mathcal{T}}
\newcommand{\UU}{\mathcal{U}}

\newcommand{\GGn}{\GG_*}

\newcommand{\NO}{\texttt{NO}}
\newcommand{\YES}{\texttt{YES}}
\newcommand{\STOP}{\texttt{STOP}}

\newcommand{\Delf}[4]{\Delta_{(\ell_0{#1}, \ell_1{#2})}^{(i_0{#3},i_1{#4})}}
\newcommand{\Delp}{\Delta_{(\ell_0',\ell_1')}^{(i_0',i_1')}}
\newcommand{\Delnp}{\Delta_{(\ell_0,\ell_1)}^{(i_0',i_1')}}
\newcommand{\Del}{{\Delf{}{}{}{}}}
\newcommand{\MDel}{M^{(i_0,i_1)}_{(\ell_0,\ell_1)}}
\newcommand{\MDelp}{M^{(i_0',i_1')}_{(\ell_0,\ell_1)}}

\renewcommand{\le}{\leqslant}
\renewcommand{\ge}{\geqslant}

\title{An asymmetric container lemma and \\ the structure of graphs with no induced $4$-cycle}

\author{Robert Morris}
\address{IMPA, Estrada Dona Castorina 110, Jardim Bot\^anico, Rio de Janeiro, RJ, Brasil}
\email{rob@impa.br}

\author{Wojciech Samotij}
\address{School of Mathematical Sciences, Tel Aviv University, Tel Aviv 6997801, Israel}
\email{samotij@post.tau.ac.il}

\author{David Saxton}
\address{DeepMind, London, UK}
\email{saxton@google.com}

\thanks{Research supported in part by CNPq (Proc.~303275/2013-8) and FAPERJ (Proc.~201.598/2014) (RM) and by ISF grant 1147/14 (WS)}

\begin{document}

\begin{abstract}
  The method of hypergraph containers, introduced recently by Balogh, Morris, and Samotij, and independently by Saxton and Thomason, has proved to be an extremely useful tool in the study of various monotone graph properties. In particular, a fairly straightforward application of this technique allows one to locate, for each non-bipartite graph $H$, the threshold at which the distribution of edges in a typical $H$-free graph with a given number of edges undergoes a transition from `random-like' to `structured'. On the other hand, for non-monotone hereditary graph properties the standard version of this method does not allow one to establish even the existence of such a threshold.

  In this paper we introduce a refinement of the container method that takes into account the asymmetry between edges and non-edges in a sparse member of a hereditary graph property. As an application, we determine the approximate structure of a typical graph with $n$ vertices, $m$ edges, and no induced copy of the $4$-cycle, for each function $m = m(n)$ satisfying $n^{4/3} (\log n)^4 \le m \ll n^2$. We show that almost all such graphs $G$  have the following property: the vertex set of $G$ can be partitioned into an `almost-independent' set (a set with $o(m)$ edges) and an `almost-clique' (a set inducing a subgraph with density $1-o(1)$). The lower bound on $m$ is optimal up to a polylogarithmic factor, as standard arguments show that if $n \ll m \ll n^{4/3}$, then almost all such graphs are `random-like'. As a further consequence, we deduce that the random graph $G(n,p)$ conditioned to contain no induced $4$-cycles undergoes phase transitions at $p = n^{-2/3 + o(1)}$ and $p = n^{-1/3 + o(1)}$. 
\end{abstract}

\maketitle

\section{Introduction}
\label{sec:intro}

Two of the central objects of study in combinatorics are the family of \emph{$H$-free} graphs, that is, the collection of graphs that do not contain $H$ as a subgraph, and the family of \emph{induced-$H$-free} graphs, that is, graphs without an induced subgraph isomorphic to $H$. An extremely well-studied problem (see,~e.g.,~\cite{FuSi13} and references therein) is to determine the largest number of edges in an $H$-free graph with a given number of vertices. This line of research dates back to the seminal works of Tur\'an~\cite{Tu41} and of Erd\H{o}s and Stone~\cite{ErSt46}, which are considered to be the cornerstones of the field of extremal graph theory. 

Another natural and well-studied problem, which also makes sense in the setting of induced-$H$-free graphs, can be informally phrased as follows:\smallskip
\begin{center}
  What does a typical $H$-free (or induced-$H$-free) graph look like?\smallskip
\end{center}
The first to address this problem were Erd\H{o}s, Kleitman, and Rothschild~\cite{ErKlRo76}, who proved that almost all triangle-free graphs are bipartite. That is, the proportion of triangle-free graphs on a given set of $n$ vertices that are bipartite (among all triangle-free graphs) tends to one as $n$ tends to infinity. This result was generalised by Kolaitis, Pr\"omel, and Rothschild~\cite{KoPrRo87}, who showed that for every $r \ge 2$, almost all $K_{r+1}$-free graphs are $r$-partite, and later by Pr\"omel and Steger~\cite{PrSt92}, who showed that the same remains true if one replaces $K_{r+1}$ with any $(r+1)$-colourable edge-critical\footnote{A graph $H$ is \emph{edge-critical} if it contains an edge $e$ such that $\chi(H \setminus e) < \chi(H)$.} graph. Further results in this direction were obtained by Hundack, Pr\"omel, and Steger~\cite{HuPrSt93} and by Balogh, Bollob\'as, and Simonovits~\cite{BaBoSi04,BaBoSi09,BaBoSi11}. Since the problem of describing the typical structure of an $H$-free graph is essentially a counting problem in disguise, we should also mention here the closely-related work of Erd\H{o}s, Frankl, and R\"odl~\cite{ErFrRo86}, who estimated the number of $H$-free graphs for every non-bipartite $H$, observing a close connection between this counting problem and the extremal question mentioned above.

The problem of understanding the typical structure of induced-$H$-free graphs seems to be significantly harder and, as a result, much less is known. The pioneers of this line of research were Pr\"omel and Steger, who described the typical structure of induced-$C_4$-free graphs~\cite{PrSt91} and induced-$C_5$-free graphs~\cite{PrStBerge}. They also proved an analogue of the Erd\H{o}s--Frankl--R\"odl theorem for induced-$H$-free graphs~\cite{PrStIII} after finding the correct generalisation of the extremal question in this setting~\cite{PrStII}, which involves the notion of a \emph{colouring number}. (This notion was later extended to the more general context of hereditary graph properties by Alekseev~\cite{Al92} and by Bollob\'as and Thomason~\cite{BoTh97}.) Much later, Alon, Balogh, Bollob\'as, and Morris~\cite{AlBaBoMo11} gave a rough structural description of a typical induced-$H$-free graph for an arbitrary $H$ (in fact, their result applies to all hereditary properties of graphs). Soon afterwards, Balogh and Butterfield~\cite{BaBu11} gave a precise structural description of a typical induced-$H$-free graph for all $H$ that are critical\footnote{The definition of criticality in the context of induced-$H$-free graphs is rather complicated, so we will only note here that it is a natural analogue of the notion of edge-criticality and refer the interested reader to~\cite{BaBu11} for the details.}. Finally, let us mention two recent works of Kim, K\"uhn, Osthus, and Townsend~\cite{KiKuOsTo} and of Keevash and Lochet~\cite{KeLo} on the typical structure of induced-$C_{2\ell}$-free graphs and induced-$(K_{a+b} \setminus K_a)$-free graphs, respectively.

Even though many of the theorems above describe very precisely the structure of a typical $H$-free  (or induced-$H$-free) graph, they say nothing about sparse $H$-free graphs. This is because the number of $H$-free graphs with $n$ vertices is in each case much greater than the number of all $n$-vertex graphs with $o(n^2)$ edges; for example, there are more than $2^{n^2/4}$ bipartite (and hence $H$-free for any non-bipartite $H$) graphs with $n$ vertices. This fact naturally leads one to consider the following refined question: 

\begin{quest} \label{quest:refined}
  Given a graph $H$ and a function $m = m(n)$, what does a typical $H$-free (or induced-$H$-free) graph with $n$ vertices and $m$ edges look like?
\end{quest}

The first to address this question were Pr\"omel and Steger~\cite{PrSt96}, who proved that almost every triangle-free graph with $n$ vertices and $m$ edges is bipartite whenever $m \ge Cn^{7/4}\log n$ and it is not bipartite if $n \ll m \ll n^{3/2}$. A few years later, {\L}uczak~\cite{Lu00} showed that this latter bound is (in some sense) sharp, by proving that if $m \gg n^{3/2}$, then almost every triangle-free graph with $n$ vertices and $m$ edges can be made bipartite by removing from it only $o(m)$ edges. More generally, it is not very hard to verify that if $n \ll m \ll n^{2-1/m_2(H)}$, where 
\[
  m_2(H) = \max\left\{ \frac{e(F)-1}{v(F)-2} \scolon F \subseteq H \text{ and } e(F) \ge 2 \right\}
\]
is the so-called \emph{$2$-density} of $H$, then almost all $H$-free graphs with $n$ vertices and $m$ edges are \emph{quasirandom}, in the sense that all sets of vertices of size $\Omega(n)$  induce subgraphs of (asymptotically) the same density. {\L}uczak~\cite{Lu00} proved that, for every non-bipartite $H$, if a certain probabilistic version of the embedding lemma for regular partitions of sparse graphs (conjectured a few years earlier by Kohayakawa, {\L}uczak, and R\"odl~\cite{KoLuRo97}) holds, then above the threshold, if $m \gg n^{2-1/m_2(H)}$, almost every $H$-free graph with $n$ vertices and $m$ edges can be made $(\chi(H)-1)$-partite by removing only $o(m)$ edges.
 
The existence of this phase transition was confirmed several years ago by Balogh, Morris, and Samotij~\cite{BaMoSa15} and by Saxton and Thomason~\cite{SaTh15}, using (what is now known as) the \emph{method of hypergraph containers}. This method (see Section~\ref{sec:asymmetric-container-lemma} or the recent survey~\cite{BaMoSa-survey}) allows one to prove the conjecture of Kohayakawa, {\L}uczak, and R\"odl mentioned above, but also provides a more direct way of determining the rough structural description of a typical $H$-free graph above the $2$-density threshold $n^{2-1/m_2(H)}$. We should also mention here the earlier works of Conlon and Gowers~\cite{CoGo16} and Schacht~\cite{Sc16} on the closely related problem of determining the size and structure of the largest $H$-free subgraph of a random graph, since these breakthroughs had a significant impact on~\cite{BaMoSa15,SaTh15}.

The exact analogue of the Erd\H{o}s--Kleitman--Rothschild theorem in the setting of sparse graphs was obtained by Osthus, Pr\"omel, and Taraz~\cite{OsPrTa03}, who proved that in fact $m = \frac{\sqrt{3}}{4} n^{3/2}(\log n)^{1/2}$ is a sharp threshold at which a typical triangle-free graph with $n$ vertices and $m$ edges becomes bipartite. A generalization of this result from triangle-free to $K_{r+1}$-free graphs, the sparse analogue of the Kolaitis--Pr\"omel--Rothschild theorem, was obtained recently by Balogh, Morris, Samotij, and Warnke~\cite{BaMoSaWa16}. The exact analogue of Tur\'an's theorem in $G(n,p)$, which sharpens the results of Conlon--Gowers and Schacht in the case $H = K_{r+1}$, was obtained by DeMarco and Kahn~\cite{DeKa, DeKa15}.

Despite the significant developments described above on the problem of determining the typical structure of a sparse $H$-free graph, there has been (as far as we are aware) essentially no progress on the corresponding problem for induced-$H$-free graphs. One reason for this is that, in contrast to the case of $H$-free graphs, the hypergraph container method does not (in general) provide the correct threshold for the appearance of structure in a typical induced-$H$-free graph. From the point of view of the container theorems, an induced-$H$-free graph is a two-edge-coloured graph that does not contain a (two-edge-coloured) clique with the same number of vertices. Since the container method does not take into account the asymmetry between the two colours, it cannot distinguish between an induced copy of $H$ and a clique. 

In this paper we introduce a new `asymmetric' version of the method of hypergraph containers that can distinguish between these two settings and provides the correct threshold for the emergence of structure in typical induced-$H$-free graphs (at least for non-bipartite graphs $H$, see Theorem~\ref{thm:KaSa}). As an illustrative example, we use it to determine the structure of a~typical induced-$C_4$-free graph with $n$ vertices and $m$ edges whenever $n^{4/3} (\log n)^4 \le m \ll n^2$. The lower bound on $m$ is best possible up to a polylogarithmic factor, as we shall also show that if $n \ll m \ll n^{4/3} (\log n)^{1/3}$, then a typical such graph does not exhibit a similar structure, and if $n \ll m \ll n^{4/3}$, then it is actually quasirandom (in the precise sense described above). We expect that the ideas contained in this work will allow analogous thresholds to be determined for families of graphs containing no induced copy of an arbitrary graph $H$, see Section~\ref{sec:open:problems}. 

\subsection{The structure of graphs with no induced $4$-cycle}

Given a graph $H$ and $n \in \N$, let $\Fn(H)$ denote the family of all graphs with vertex set $\{1, \dotsc, n\}$ that contain no induced copy of $H$ and let $\Fnm(H)$ denote the family of graphs in $\Fn(H)$ with precisely $m$ edges. A \emph{split graph} is a graph whose vertex set can be partitioned into a clique and an independent set. It is easy to check that a split graph cannot contain an induced copy of $C_4$; indeed, the property of being a split graph is hereditary and $C_4$ itself is not a split graph. Conversely, as mentioned above, it was proved by Pr\"omel and Steger~\cite{PrSt91} over 25 years ago that almost all graphs in $\Fn(C_4)$ are split graphs. However, since almost all $n$-vertex split graphs admit a partition into a clique and an independent set of roughly equal sizes and have approximately $n^2/4$ edges, this result says nothing about a typical member of $\FnmCf$ when $m$ is not approximately $n^2/4$. It is worth mentioning that Gishboliner and Shapira~\cite{GiSh} recently described the structure of \emph{all} induced-$C_4$-free graphs; their description is much coarser, however.

We will prove that if $n^{4/3}(\log n)^4 \le m \ll n^2$, then a typical member of $\FnmCf$ is `almost' a split graph. We will write a.a.s.\ (shorthand for asymptotically almost surely) as an abbreviation of ``with probability tending to $1$ as $n \to \infty$'' and say that a graph $G$ with $n$ vertices and $p\binom{n}{2}$ edges is \emph{$\eps$-quasirandom} if every subset of more than $\eps n$ vertices of $G$ induces a subgraph with density between $(1-\eps)p$ and $(1+\eps)p$. We will say that a graph $G$ is \emph{$\eps$-close to a split graph} if there exists a partition $V(G) = A \cup B$ such that $e_G(A) \ge (1-\eps)\binom{|A|}{2}$ and $e_G(B) \le \eps e(G)$. Our first main result is the following structural description of a typical graph in $\FnmCf$.

\begin{thm}
  \label{thm:asymC4free}
  For every $\eps > 0$, there exists $\delta > 0$ such that the following holds. Let $G$ be a uniformly chosen random graph in $\FnmCf$. 
  \begin{enumerate}[label={(\alph*)}]
  \item
    \label{item:C4-free-quasirandom}
    If $n \ll m \le \delta n^{4/3}$, then a.a.s.\ $G$ is $\eps$-quasirandom.\smallskip
  \item
    \label{item:C4-free-non-split}
    If $n \ll m \le \delta n^{4/3} (\log n)^{1/3}$, then a.a.s.\ $G$ is not $1/4$-close to a split graph.\smallskip
  \item
    \label{item:C4-free-almost-split}
    If $n^{4/3} (\log n)^4 \le m \le \delta n^2$, then a.a.s.\ $G$ is $\eps$-close to a split graph.
  \end{enumerate}
\end{thm}

The following result is a relatively straightforward consequence of Theorem~\ref{thm:asymC4free}. It determines the number of edges in (and therefore, by Theorem~\ref{thm:asymC4free}, the typical structure) of the random graph $G(n,p)$ conditioned on not containing an induced copy of $C_4$. We write $\Gnpind(C_4)$ to denote the random graph chosen according to this conditional distribution. 

\begin{cor}\label{cor:GnpC4free}
The following bounds hold asymptotically almost surely as $n \to \infty$:
\[
e \big( \Gnpind(C_4) \big) =
    \begin{cases}
      \big( 1 + o(1) \big) p {n \choose 2} & \text{if $n^{-1} \ll p \ll n^{-2/3}$}, \\[+0.1ex]
      n^{4/3} (\log n)^{O(1)} & \text{if $n^{-2/3} \le p \le n^{-1/3} (\log n)^4$}, \\[+0.1ex]
      \Theta\big( p^2 n^2 / \log(1/p) \big) & \text{if $p \ge n^{-1/3} (\log n)^4$}.
    \end{cases}
  \]
\end{cor}

Note that it follows immediately from Theorem~\ref{thm:asymC4free} that $\Gnpind(C_4)$ is a.a.s.\ $\eps$-quasirandom if $n^{-1} \ll p \ll n^{-2/3}$ and a.a.s.\ $\eps$-close to a split graph if $p \ge n^{-1/3} (\log n)^4$. We remark that we have not attempted to optimize the exponents of $\log n$, since (we believe that) our technique cannot give the correct power.

We would like to draw the reader's attention to the (somewhat surprising) fact that in the middle range $n^{-2/3+o(1)} \le p \le n^{-1/3+o(1)}$, the typical value of $e\big(\Gnpind(C_4)\big)$ stays essentially constant. This is because the proportion of $n$-vertex graphs with $m$ edges that are induced-$C_4$-free drops very sharply from $e^{-o(m)}$ to $e^{-\Omega(m \log n)}$ as $m$ crosses a very narrow interval around $n^{4/3}$, as shown by Theorem~\ref{thm:asymC4free}. A similar phenomenon has been observed in several random Tur\'an problems for forbidden bipartite graphs (even cycles~\cite{KoKrSt98, MoSa16} and complete bipartite graphs~\cite{MoSa16}) as well as Tur\'an-type problems in additive combinatorics~\cite{DeKoLeRoSa-B3, DeKoLeRoSa-Bh}. It would be very interesting to determine whether a similar `long flat segment' appears in the graphs of $p \mapsto e\big(\Gnpind(H)\big)$ and $p \mapsto \ex\big( G(n,p), H \big)$ for every bipartite $H$.

Our proof of Theorem~\ref{thm:asymC4free} relies on two new results: (i) an asymmetric container lemma, which generalises the main results of~\cite{BaMoSa15, SaTh15}, and (ii) a new robust stability theorem for induced copies of $C_4$ in `pregraphs' (see below). We discuss these two ingredients in the remainder of this section.

\subsection{The asymmetric container lemma}

\label{sec:asymmetric-container-lemma}

The hypergraph container theorems, proved independently by Balogh, Morris, and Samotij~\cite{BaMoSa15} and by Saxton and Thomason~\cite{SaTh15}, state (roughly speaking) that the family of independent sets of a uniform hypergraph whose edges are distributed somewhat evenly can be covered with a small number of sets, called \emph{containers}, each of which is `almost independent' in the sense that it contains only few edges of the hypergraph. This fact has proved to be a very convenient and useful tool in the study of the families of $H$-free graphs, as well as other monotone properties of graphs, hypergraphs, sets of integers, etc. There are several reasons for this. First, there is a natural correspondence between $H$-free graphs with a given number $n$ of vertices and independent sets in the $e_H$-uniform hypergraph $\HH$ whose vertex set is $E(K_n)$, the edge set of the complete graph with $n$ vertices, and whose edges are the edge sets of all copies of $H$ found in $K_n$. Second, classical results in extremal graph theory provide very precise and explicit descriptions of graphs with few copies of $H$, which correspond to the containers for independent sets of $\HH$. Third, the bounds for the number of containers given by~\cite{BaMoSa15,SaTh15} are essentially optimal, which allows one to deduce many best-possible estimates on the number of $H$-free graphs with given numbers of vertices and edges and describe their typical structure.

The container theorems can also be used to enumerate graphs with no \emph{induced} copy of $H$. In fact, this was already done by Saxton and Thomason in their original paper~\cite{SaTh15}, where they obtained (implicitly) upper bounds on $|\Find_n(H)|$ for all $H$.  One way to phrase this problem in the language of independent sets is to consider the hypergraph $\HH$ whose vertex set is $E(K_n) \times \{0,1\}$ and whose edges are
\begin{enumerate}[label={(\textit{\roman*})}]
\item
  all the $\binom{v_H}{2}$-element sets of the form $E \times \{1\} \cup (\binom{W}{2} \setminus E) \times \{0\}$, where $W$ ranges over all $v_H$-element sets of vertices of $K_n$ and $E$ is the subset of $\binom{W}{2}$ covered by $E(H)$ in one of the $v_H!/|\Aut(H)|$ non-isomorphic embeddings of $H$ into $W$ and
\item
  \label{item:edges-ii}
  all the $\binom{n}{2}$ pairs $\{(e,0), (e,1)\}$, where $e$ ranges over all edges of $K_n$.
\end{enumerate}
One can see that $n$-vertex graphs with no induced copy of $H$ are in a natural one-to-one correspondence with the independent sets of $\HH$ with $\binom{n}{2}$ elements. Even though the container theorems may be applied only to uniform hypergraphs, since one is usually interested in upper bounds, one may disregard the $2$-uniform edges of type~\ref{item:edges-ii} and construct containers for independent sets of the resulting smaller $\binom{v_H}{2}$-uniform hypergraph, which clearly include all independent sets of the original hypergraph.

One soon realises that the above approach is somewhat flawed when one is interested in the family $\Find_{n,m}(H)$ whenever $m$ is either very small or very close to $\binom{n}{2}$ and $H$ is neither complete nor empty. This is because the original container theorems completely disregard the obvious asymmetry between the edges and the non-edges of $H$ in each of the $\binom{v_H}{2}$-uniform edges of $\HH$. As a result, one cannot expect to deduce optimal bounds on $|\Find_{n,m}(H)|$ for all $m$ using this approach. Our main motivation for this work is to address this issue.

Departing somewhat from the language of independent sets, we shall regard a graph $G \subseteq K_n$ as the characteristic function $h_G \colon E(K_n) \to \{0,1\}$ of its edge set; that is, $h_G(e) = 1$ if $e \in E(G)$ and $h_G(e) = 0$ otherwise. The family $\Fn(H)$, viewed as a set of functions $h \colon E(K_n) \to \{0,1\}$, may be described by a set of \emph{constraints} of the form
\[
  \neg\left( h|_E \equiv 1 \wedge h|_{\binom{W}{2} \setminus E} \equiv 0 \right).
\]
In other words, a function $h \in \Fn(H)$ cannot simultaneously map all elements of $E$ to $1$ and all elements of $\binom{W}{2} \setminus E$ to $0$, for any $W \subset V(K_n)$ with $|W| = v_H$ and any $E$ that is the edge set of an embedding of $H$ into $W$.


There is nothing special here about the family $\Fn(H)$ or the set $E(K_n)$. Therefore, for the remainder of this discussion, we shall replace $E(K_n)$ with an arbitrary finite set $V$, let $\HH$ be an arbitrary family of pairs of disjoint subsets of $V$, and let
\[
  \FF(\HH) = \left\{ h \in \{0,1\}^V \scolon \neg( h|_{A_0} \equiv 0 \wedge h|_{A_1} \equiv 1) \text{ for all $(A_0, A_1) \in \HH$} \right\}.
\]
In other words, one obtains the family $\FF(\HH)$ from $\{0,1\}^V$ by discarding all $h \colon V \to \{0,1\}$ that map each element of $A_0$ to $0$ and each element of $A_1$ to $1$ for some pair $(A_0, A_1) \in \HH$. We shall informally refer to these pairs of sets as \emph{constraints} and say that $h$ \emph{violates} (resp.~\emph{satisfies}) a constraint $(A_0, A_1)$ if $h$ maps (resp.~does not map) each element of $A_0$ to $0$ and each element of $A_1$ to $1$. Finally, let us note here for future reference that according to the above definition, $\FF(\HH)$ is empty whenever $\HH$ contains the pair $(\emptyset, \emptyset)$; in other words, every function violates the `empty' constraint $(\emptyset, \emptyset)$.

The container theorems imply that if such a family $\HH$ contains only pairs $(A_0, A_1)$ with a given value of $|A_0|+|A_1|$ and the sets $A_0 \cup A_1$ are distributed somewhat uniformly, then there is a small family $\CC$ of partitions $V = V_0 \cup V_1 \cup V_*$ such that
\[
  \FF(\HH) \subseteq \bigcup_{(V_0,V_1,V_*) \in \CC} \{0\}^{V_0} \times \{1\}^{V_1} \times \{0,1\}^{V_*}
\]
and, importantly, every function in each of the cylinders $\{0\}^{V_0} \times \{1\}^{V_1} \times \{0,1\}^{V_*}$ violates only few constraints in $\HH$. In particular, one does not allow a trivial covering of $\FF(\HH)$ with $\{0,1\}^V$, which corresponds to $V_* = V$. Roughly speaking, we might say that $\FF(\HH)$ may be `tightly' covered by a small family of cylinders. 

In this work, we take a refined approach to this covering problem. We shall build families of containers that are tailored to the subfamily of all $h \in \FF(\HH)$ that attain the values $0$ and $1$ given numbers of times, unlike in previous works. More precisely, for each integer $m$ with $0 \le m \le |V|$, we shall consider the subfamily $\FF_m(\HH) \subseteq \FF(\HH)$ defined by
\[
  \FF_m(\HH) = \big\{h \in \FF(\HH) \scolon |h^{-1}(1)| = m\big\}
\]
and build a family of containers for the elements of $\FF_m(\HH)$ only.

We shall focus our attention on families $\FF(\HH)$ determined by collections $\HH$ of constraints that are uniform in the sense that each $(A_0, A_1) \in \HH$ satisfies $|A_0| = k_0$ and $|A_1| = k_1$ for some fixed integers $k_0$ and $k_1$. We shall refer to such collections $\HH$ as \emph{$(k_0,k_1)$-uniform hypergraphs}. In standard applications of the container method this should not be a huge restriction, provided that we are only interested in constraints of bounded size, that is, pairs $(A_0, A_1)$ where $|A_0| + |A_1|$ is bounded from above by a constant. Indeed, given a non-uniform family of constraints of bounded size, we may restrict our attention to the `densest' $(k_0, k_1)$-uniform hypergraph that is contained in the family, losing only some constant factors. In fact, this is precisely what we are going to do in our proof of Theorem~\ref{thm:asymC4free}. 

For a $(k_0, k_1)$-uniform hypergraph $\HH$ and two disjoint sets $T_0$ and $T_1$, we define
\[
  \deg_{\HH}(T_0, T_1) = |\{ (A_0, A_1) \in \HH \scolon T_0 \subseteq A_0 \text{ and } T_1 \subseteq A_1 \}|.
\]
Furthermore, for each pair of integers $(\ell_0, \ell_1)$, we let
\[
  \Delta_{(\ell_0,\ell_1)}(\HH) = \max \left\{ \deg_{\HH} (T_0, T_1) \scolon \text{$T_0, T_1 \subseteq V$ with $|T_0| = \ell_0$ and $|T_1| = \ell_1$} \right\}.
\]
Abusing notation somewhat, we shall identify a partition $V = V_0 \cup V_1 \cup V_*$ with the cylinder $\{0\}^{V_0} \times \{1\}^{V_1} \times \{0,1\}^{V_*}$ and the function $a \colon V(\HH) \to \{0,1,*\}$ defined by $a^{-1}(x) = V_x$ for each $x \in \{0,1,*\}$. In particular, a function $h \colon V(\HH) \to \{0, 1\}$ belongs to the cylinder $a \colon V(\HH) \to \{0, 1, *\}$ if $h(v) = a(v)$ for all $v \in V(\HH)$ such that $a(v) \neq *$. In other words, $h(v)$ is forced to equal $a(v)$ unless $a(v) = *$, in which case $h(v)$ can be either $0$ or $1$.

We are now ready to state the main result of this section, an asymmetric container theorem. In the statement of the theorem, $\FF_{\le m}(\HH)$ is a shorthand for $\bigcup_{m'=0}^m \FF_{m'}(\HH)$.

\begin{thm}
  \label{thm:container}
  For all integers $k_0, k_1 \ge 0$, not both zero, and each $K > 0$, the following holds. Suppose that $\HH$ is a non-empty $(k_0, k_1)$-uniform hypergraph and $b$, $m$, and $r$ are integers satisfying
  \begin{equation}
    \label{eq:container-Del}
    \Delta_{(\ell_0, \ell_1)}(\HH) \le K \cdot \frac{b^{\ell_0+\ell_1-1}}{m^{\ell_0} \cdot v(\HH)^{\ell_1}} \cdot e(\HH) \cdot \left(\frac{m}{r}\right)^{\indicator[\ell_0 > 0]}
  \end{equation}
  for every pair $(\ell_0,\ell_1) \in \{0, \ldots, k_0\} \times \{0, \ldots, k_1\}$ with $(\ell_0,\ell_1) \neq (0,0)$. Then there exist a~family $\cS \subseteq \binom{V(\HH)}{\le k_0b} \times \binom{V(\HH)}{\le k_1b}$ and functions $f \colon \cS \to \{0,1,*\}^{V(\HH)}$ and $g \colon \FF_{\le m}(\HH) \to \cS$ such that, letting $\delta = 2^{-(k_0+k_1)(k_0+k_1+1)} K^{-1}$:
  \begin{enumerate}[label={(\alph*)}]
  \item
    \label{item:container-1}
    Every $h \in \FF_{\le m}(\HH)$ belongs to the cylinder $f(g(h))$.\smallskip
  \item
    \label{item:container-2}
    Either $|f(S)^{-1}(0)| \ge \delta v(\HH)$ or $|f(S)^{-1}(1)| \ge \delta r$ for every $S \in \cS$; moreover, the former can hold only if $k_1 > 0$ and the latter can hold only if $k_0 > 0$.\smallskip
     \item
       \label{item:container-3} 
       If $g(h) = (S_0, S_1)$ for some $h \in \FF_{\le m}(\HH)$, then $S_0 \subseteq h^{-1}(0)$ and $S_1 \subseteq h^{-1}(1)$.
  \end{enumerate}
\end{thm}

Yet another rephrasing of condition~\ref{item:container-1} is that whenever $g(h) = S$, then $h$ is forced to take the value $0$ on $f(S)^{-1}(0)$ and it is forced to take the value $1$ on $f(S)^{-1}(1)$. Note the asymmetry between the guaranteed lower bounds on the cardinalities of the sets $f(S)^{-1}(0)$ and $f(S)^{-1}(1)$ in~\ref{item:container-2}. Roughly speaking, we are equally satisfied with \ref{item:cube-0}~containers forcing our function $h$ to take the value $0$ on a positive proportion of $V(\HH)$ and \ref{item:cube-1}~containers forcing our function to take the value $1$ only on some $\delta r$ elements of $V(\HH)$. Condition~\ref{item:container-3} states that for every $h \in \FF_{\le m}(\HH)$, the value of $g(h)$ is `consistent' with $h$. This additional property of the function $g$ will not be used in our application of the theorem to enumerating $\FnmCf$. However, we state it here as the analogous property in the original container theorems was crucial in avoiding superfluous logarithmic factors in many applications of the container method. Finally, let us point out here that we shall be allowing all of our hypergraphs to contain edges with multiplicities greater than one. In particular, both $e(\cdot)$ and $\deg_{\HH}(\cdot, \cdot)$ count edges with their multiplicities.


A reader who is familiar with the container method might notice that by setting $r = m = v(\HH)$ in Theorem~\ref{thm:container}, one recovers the statement of the original container theorem~\cite[Proposition~3.1]{BaMoSa15} in the somewhat 
more general context of $(k_0, k_1)$-uniform hypergraphs. To illustrate the `asymmetry' in Theorem~\ref{thm:container}, we need to assume that $m \ll v(\HH)$. For brevity, let $N = v(\HH)$ and consider two cylinders, described by the following two partitions of~$V(\HH)$:
\begin{enumerate}[label={(\textit{\roman*})}]
\item 
  \label{item:cube-0}
  $V(\HH) = V_0 \cup V_1 \cup V_*$, where $|V_0| = \delta N$ and $V_1 = \emptyset$,
\item
  \label{item:cube-1}
  $V(\HH) = V_0' \cup V_1' \cup V_*'$, where $V_0' = \emptyset$ and $|V_1'| = \delta r$.
\end{enumerate}
Observe that the cylinder described in~\ref{item:cube-0} contains at most $\binom{(1-\delta)N}{m}$ functions from $\FF_m(\HH)$, whereas the cylinder described in~\ref{item:cube-1} contains at most $\binom{N-\delta r}{m - \delta r}$ functions from $\FF_m(\HH)$. Assume that $r \ll m \ll N$. Since
\[
  \binom{(1-\delta)N}{m} \approx (1-\delta)^m \cdot \binom{N}{m} \qquad \text{and} \qquad \binom{N-\delta r}{m - \delta r} \approx \bigg( \frac{m}{N} \bigg)^{\delta r} \binom{N}{m},
\]
then both cylinders will have equal \emph{volume}\footnote{By the volume of a cylinder $\{0\}^{V_0} \times \{1\}^{V_1} \times \{0,1\}^{V_*}$, we will mean the number of functions $h \colon V(\HH) \to \{0,1\}$, with $|h^{-1}(1)| = m$, that are contained in the cylinder, that is, $\binom{|V_*|}{m-|V_1|}$.} when $r \approx m / \log(N/m) \ll m$. On the other hand, when $r \ll m \ll N$, then the assumptions on the maximum degrees of $\HH$ stated in~\eqref{eq:container-Del} are weaker by a factor of $\left(\frac{N}{m}\right)^{\ell_0} \cdot \left(\frac{m}{r}\right)^{\indicator[\ell_0 > 0]}$ when compared to the original container theorems, see~\cite[Proposition~3.1]{BaMoSa15}. This allows one to choose a smaller $b$, which results in a smaller family of containers.

As we believe that having a trade-off between the upper bound on the size of containers for independent sets in a hypergraph $\HH$ and the upper bounds on maximum degrees $\Delta_\ell(\HH)$ can be useful in other applications of the container method, we conclude this section with a sharpening of the original container theorems, \cite[Proposition~3.1]{BaMoSa15} and \cite[Theorem~3.4]{SaTh15}, that follows easily from Theorem~\ref{thm:container}. We write $\II(\HH)$ for the family of independent sets of $\HH$ and $\Delta_\ell(\HH)$ for the largest number of edges of $\HH$ that contain a particular $\ell$-element subset of $V(\HH)$.

\begin{thm}
  \label{thm:container-mono}
  Suppose that positive integers $b$, $k$, and $r$ and a non-empty $k$-uniform hypergraph $\HH$ satisfy
  \begin{equation}
    \label{eq:container-mono-Del}
    \Delta_\ell(\HH) \le \left(\frac{b}{v(\HH)}\right)^{\ell-1} \frac{e(\HH)}{r}
  \end{equation}
  for every $\ell \in \{1, \ldots, k\}$. Then there exist a family $\cS \subseteq \binom{V(\HH)}{\le kb}$ and functions $f \colon \cS \to \cP(V(\HH))$ and $g \colon \II(\HH) \to \cS$ such that for every $I \in \II(\HH)$,
  \[
    g(I) \subseteq I \subseteq f(g(I)) \qquad \text{and} \qquad |f(g(I))| \le v(\HH) - \delta r,
  \]
  where $\delta = 2^{-k(k+1)}$.
\end{thm}

To obtain Theorem~\ref{thm:container-mono}, we simply apply Theorem~\ref{thm:container} to the $(0, k)$-uniform hypergraph with the same vertex set as $\HH$ whose edges are all pairs $(\emptyset, A)$ such that $A$ is an edge of~$\HH$. We shall spell out a few more details at the end of Section~\ref{sec:asym:containers}.

\subsection{Robust balanced stability for induced $C_4$s}

\label{sec:robust-stability-intro}

In order to determine the structure of a typical graph in $\FnmCf$ using the container method, we ought to characterise all containers whose volume is (close to) the largest possible. Our containers for $\FnmCf$ will be cylinders in $\{0,1\}^{E(K_n)}$ that correspond to partitions $E(K_n) = E_0 \cup E_1 \cup E_*$ with the following property: There are only few $4$-vertex subsets $\{v_1, v_2, v_3, v_4\}$ such that $v_1v_2, v_3v_4 \in E_0 \cup E_*$ and $v_1v_3, v_1v_4, v_2v_3, v_2v_4 \in E_1 \cup E_*$. Each such set $\{v_1, v_2, v_3, v_4\}$ induces a copy of $C_4$ in some graph described by the partition $E(K_n) = E_0 \cup E_1 \cup E_*$.\footnote{These are all $G$ such that $E_1 \subseteq E(G) \subseteq E_1 \cup E_*$ and $E_0 \subseteq E(K_n) \setminus E(G) \subseteq E_0 \cup E_*$.} Since we are interested only in graphs with exactly $m$ edges, the volume of a container is simply the number of graphs with $m$ edges that this cylinder contains, that is, $\binom{|E_*|}{m-|E_1|}$. The precise statements of our results are rather technical, but roughly speaking we show that each container whose volume is close to largest possible has the following structure: the graph $E_1$ contains an `almost-complete' graph with vertex set $W$, and most edges in $E_*$ have an endpoint in $W$.

To avoid excessive use of indices, we shall view partitions of $E(K_n)$ of the above type as partial two-colourings of the edges of $K_n$ that we shall call pregraphs. More precisely, by a \emph{pregraph} $\cP$ of $E(K_n)$ we will mean a pair $(M,E)$ of disjoint subsets of $E(K_n)$. We shall refer to the elements of the set $E$ as \emph{edges} and the elements of the set $M$ as \emph{mixed edges}.\footnote{The sets $E$ and $M$ correspond to the sets $E_1$ and $E_*$ above, respectively.} A~\emph{good copy} of $C_4$ in $\cP$ is a copy of $C_4$ in $M$ whose vertex set is independent in $E$. Note that each good copy of $C_4$ corresponds to a set $\{v_1, v_2, v_3, v_4\}$ described in the previous paragraph (but not vice-versa). This means, in particular, that the pregraph corresponding to each container contains only few good copies of $C_4$. We shall therefore restrict our attention to characterising pregraphs with few good copies of $C_4$. As we will later see, a sufficiently precise and useful characterisation of containers can be derived from a robust stability theorem for pregraphs, which we state here in an abbreviated form; for the full statement, we refer the reader to Section~\ref{sec:robust-stability}. We will say that a graph $G$ is \emph{$\eps$-close to $K_\ell$} if one can transform $G$ into $K_\ell$ by adding or deleting at most $\eps \binom{\ell}{2}$ edges.

\begin{thm}
  \label{thm:robust-stability}
  For every $\eps > 0$ there exist positive constants $C$, $\delta$, and $\beta$ such that the following holds for all integers $\ell$ and $n$ with $\ell \ge C \sqrt{n}$. Let $\cP = (M,E)$ be a pregraph on $n$ vertices with 
  \[
    |E| \le \binom{\ell}{2} \qquad \textup{and} \qquad |M| \ge (1 - \delta)\ell n.
  \]
  Then either $E$ is $\eps$-close to $K_\ell$ or $\cP$ contains at least $\beta \ell^4$ good copies of $C_4$.
\end{thm}

Observe that Theorem~\ref{thm:robust-stability} provides a structural characterisation of all those pregraphs $(M,E)$ on $n$ vertices with $|E| \le \binom{\ell}{2}$ and $|M| \ge (1-o(1))\ell n$ for some $\ell \gg \sqrt{n}$ that contain only $o(\ell^4)$ good copies of $C_4$. For each such pregraph $(M, E)$, there is a set $U$ of $\ell$ vertices on which $E$ is almost complete. Moreover, all but $o(\ell n)$ mixed edges have an endpoint in~$U$. Indeed, if some $\Omega(\ell n)$ mixed edges did not have an endpoint in $U$, then Theorem~\ref{thm:robust-stability} applied to the pregraph induced by the complement of $U$ would produce $\Omega(\ell^4)$ good copies of $C_4$.

\subsection{Organisation of the paper}
\label{sec:organisation-paper}

The rest of the paper is organised as follows. In Section~\ref{sec:asym:containers} we prove the asymmetric container lemma, in Section~\ref{sec:robust-stability} we prove Theorem~\ref{thm:robust-stability},  in Section~\ref{sec:non-structured-regime} we prove the lower bounds in Theorem~\ref{thm:asymC4free}, and in Section~\ref{sec:C4s:proof} we complete the proof of Theorem~\ref{thm:asymC4free} and Corollary~\ref{cor:GnpC4free}. Finally, in Section~\ref{sec:open:problems} we discuss some open questions and further applications of the asymmetric container lemma. 


\section{The proof of the asymmetric container lemma}\label{sec:asym:containers}

\subsection{Proof outline}
\label{sec:proof-outline}

Our proof of Theorem~\ref{thm:container} follows the general strategy of~\cite{BaMoSa15}. Namely, we construct a function $f^* \colon \FF(\HH) \to \{0,1,*\}^{V(\HH)}$ that satisfies the following two conditions for every $h \in \FF(\HH)$. Writing $f^*_h$ as a shorthand for $f^*(h)$, the two conditions are:
\begin{enumerate}[label={(\textit{\alph*})}]
\item
  \label{item:outline-f*h-1}
  $h$ belongs to the cylinder $f^*_h$,
\item
  \label{item:outline-f*h-2}
  $|(f^*_h)^{-1}(0)| \ge \delta v(\HH)$ or $|(f^*_h)^{-1}(1)| \ge \delta r$,
\end{enumerate}
cf.~\ref{item:container-1} and \ref{item:container-2} in the statement of Theorem~\ref{thm:container}. Crucially, the function $f^*$ takes only at most $\binom{v(\HH)}{\le k_0b} \cdot \binom{v(\HH)}{\le k_1b}$ different values. This last property is a simple consequence of the fact that the algorithmic construction of $f^*$ can be encoded as a sequence of decisions that naturally correspond to a pair of subsets of $V(\HH)$ containing at most $k_0b$ and $k_1b$ elements, respectively. In particular, we shall obtain an implicit decomposition $f^* = f \circ g$ promised in Theorem~\ref{thm:container}.

The function $f^*$ is constructed by an algorithm that operates in a sequence of at most $k_0+k_1-1$ rounds. At the beginning of each round, we are given an $(i_0, i_1)$-uniform hypergraph $\GG$ with the same vertex set as $\HH$ and such that $h \in \FF(\GG)$; at the beginning of the first round, $(i_0, i_1) = (k_0, k_1)$ and $\GG = \HH$. We let $(i_0', i_1') = (i_0, i_1-1)$ if $i_1 > 0$ and let $(i_0', i_1') = (i_0-1, i_1) = (i_0-1, 0)$ otherwise. By the end of the round, we will have either (i)~defined a function $f^*_h \colon V(\HH) \to \{0,1,*\}$ satisfying both~\ref{item:outline-f*h-1} and~\ref{item:outline-f*h-2} above, or (ii)~constructed an $(i_0', i_1')$-uniform hypergraph $\GG^*$ with $V(\GG^*) = V(\HH)$ and such that $h \in \FF(\GG^*)$ whose maximum degrees satisfy conditions akin to the conditions on the maximum degrees of $\HH$ given by~\eqref{eq:container-Del}. This is achieved in the following way.

We start with $\GG^*$ empty and $f^*_h \equiv *$. We set $c = 1$ if $i_1 > 0$ and $c = 0$ otherwise, so $i_c' = i_c - 1$. Our algorithm considers a sequence of questions of the form ``Is $h(v) = c$?''~for some carefully chosen (sequence of) vertices $v \in V(\HH)$. If the answer is \YES, then we set $f^*_h(v) = c$ and, more importantly, we add new $(i_0', i_1')$-uniform constraints to $\GG^*$ in the following way. As $h(v) = c$, if $h$ satisfies a constraint\footnote{Recall from Section~\ref{sec:asymmetric-container-lemma} that $h$ satisfies the constraint $(A_0, A_1)$ if an only if $f$ does \emph{not} simultaneously take only the value $0$ on $A_0$ and only the value $1$ on $A_1$; equivalently, $f$ either takes the value $1$ on some element of $A_0$ or the value $0$ on some element of $A_1$.} $(A_0, A_1)$ with $v \in A_c$, then it also satisfies the constraint $(A_0', A_1')$ defined by $A_c' = A_c \setminus \{v\}$ and $A_{1-c}' = A_{1-c}$. In view of this, for each $(A_0, A_1) \in \GG$ with $v \in A_c$, we add to $\GG^*$ the corresponding $(A_0', A_1')$. If the answer is \NO, then we only set $f^*_h(v) = 1-c$. (We thus choose to ignore all the constraints $(A_0, A_1) \in \GG$ such that $v \in A_{1-c}$.) The round ends when either the number of \YES\ answers reaches $b$ or if no constraints remain involving only the vertices that we have not yet asked about. Our assumptions on the maximum degrees of the hypergraph $\GG$ imply that in the latter case, the number of \NO\ answers will be sufficiently large to deduce that $|(f^*_h)^{-1}(1-c)|$ is sufficiently large (that is, at least $\delta v(\HH)$ if $c=1$ and at least $\delta r$ if $c = 0$). If this does not happen (and hence the number of \YES\ answers reaches $b$), then we shall be able to show that the hypergraph $\GG^*$, which we have created based on the \YES\ answers, contains a subhypergraph with sufficiently many edges, whose maximum degrees satisfy the required conditions. In this case, we let $\GG \leftarrow \GG^*$ and $(i_0, i_1) \leftarrow (i_0', i_1')$ and proceed to the next round.

Since, as noted before, no function satisfies the empty constraint $(\emptyset, \emptyset)$, it follows that in the round when $i_0 + i_1 = 1$, no \YES\ answers can be given. (Otherwise, a non-empty $(0,0)$-uniform hypergraph $\GG^*$ with $h \in \FF(\GG^*)$ would be constructed.) In particular, the function $f^*_h$ will have to be defined in this round, provided that the algorithm reaches it. 

Even though the sequence of values of $c$ that we choose (i.e., we let $c  = 1$ as long as $i_1$ is not yet zero) may seem somewhat arbitrary, it has a very important consequence. Namely, if $\GG$ is an $(i_0, 0)$-uniform hypergraph with $V(\GG) = V(\HH)$ and $h \in \FF_{\le m}(\GG)$, then there must be a vertex $v \in V(\GG)$ such that $\deg_{\GG}(\{v\}, \emptyset) \ge e(\GG) / m$. Indeed, the set $h^{-1}(1)$ has at most $m$ elements and it has to intersect $A_0$ for each $(A_0, \emptyset) \in \GG$. Note that if $m \ll v(\HH)$, then $e(\GG) / m$ is much larger than the average degree of $\GG$. This simple observation is the reason why restricting to the family $\FF_{\le m}(\HH)$ allows us to create a smaller family of containers.

Finally, since each of the questions asked by the algorithm is a \YES/\NO\ question, we may encode the execution of the algorithm, and thus also the function $f^*$, as a set of at most $(k_0+k_1-1) \cdot b$ vertices for which the answer was \YES.

We conclude this outline with an important technical remark. Throughout this section we allow all of our hypergraphs to contain edges with multiplicities greater than one. Moreover, when computing various degrees $\deg(\cdot,\cdot)$ or cardinalities $e(\cdot)$ of the edge sets of various hypergraphs, we shall always count edges with multiplicities. As first discovered by Saxton and Thomason in~\cite{SaTh15} and later reiterated in~\cite{BaMoSa15}, this seemingly insignificant detail has far-reaching consequences in both the statement and the proof of the container theorems.

\subsection{Setup}

Let $k_0$ and $k_1$ be nonnegative integers and let $K$ be a positive real. Let $b$, $m$, and $r$ be positive integers and suppose that $\HH$ is a $(k_0, k_1)$-uniform hypergraph satisfying~\eqref{eq:container-Del} for every pair $(\ell_0, \ell_1)$ as in the statement of Theorem~\ref{thm:container}. We claim that without loss of generality we may assume that $b \le m \le v(\HH)$. Indeed, if $m > v(\HH)$, then we may replace $m$ with $v(\HH)$ as $\FF_{\le m} \subseteq \FF(\HH) = \FF_{\le v(\HH)}(\HH)$ and the right-hand side of~\eqref{eq:container-Del} is a non-increasing function of $m$. If $b > v(\HH) \ge m$, then we may replace $b$ with $v(\HH)$. This is because $\binom{V(\HH)}{\le k_i b} = \binom{V(\HH)}{\le k_i v(\HH)}$ and the assumed upper bounds on the maximum degrees of $\HH$ remain true even after we replace $b$ with $v(\HH)$. Indeed, if $\ell_0 > 0$, then for every $\ell_1 \in \{0, \dotsc, k_1\}$,
\[
  \Delta_{(\ell_0, \ell_1)}(\HH) \le \Delta_{(1, 0)}(\HH) \le K \cdot \frac{e(\HH)}{r} \le K \cdot \frac{v(\HH)^{\ell_0+\ell_1-1}}{r \cdot m^{\ell_0-1} \cdot v(\HH)^{\ell_1}} \cdot e(\HH), 
\]
as $v(\HH) \ge m$, and if $\ell_0 = 0$, then for every $\ell_1 \in \{1, \dotsc, k_1\}$,
\[
  \Delta_{(0, \ell_1)}(\HH) \le \Delta_{(0,1)}(\HH) \le K \cdot \frac{e(\HH)}{v(\HH)} = K \cdot \frac{v(\HH)^{\ell_1-1}}{v(\HH)^{\ell_1}} \cdot e(\HH).
\]
Finally, if $v(\HH) \ge b > m$, then we may replace $m$ with $b$, since $\FF_{\le m}(\HH) \subseteq \FF_{\le b}(\HH)$, the bound on $\Delta_{(0, \ell_1)}(\HH)$ in~\eqref{eq:container-Del} does not depend on $m$, and if $\ell_0 > 0$, then
\[
  \Delta_{(\ell_0, \ell_1)}(\HH) \le \Delta_{(1,\ell_1)}(\HH) \le K \cdot \frac{b^{\ell_1}}{r \cdot v(\HH)^{\ell_1}} \cdot e(\HH) = K \cdot \frac{b^{\ell_0+\ell_1-1}}{r \cdot b^{\ell_0-1} \cdot v(\HH)^{\ell_1}} \cdot e(\HH)
\]
for every $\ell_1 \in \{0, \dotsc, k_1\}$.

We shall be working only with hypergraphs whose uniformities come from the set 
\[
  \UU := \big\{ (1,0), (2,0), \ldots, (k_0,0), (k_0,1), \ldots, (k_0,k_1) \big\}.
\]
We now define a collection of numbers that will be upper bounds on the maximum degrees of the hypergraphs constructed by our algorithm. To be more precise, for each $(i_0, i_1) \in \UU$ and all $(\ell_0, \ell_1)$, we shall force the maximum $(\ell_0, \ell_1)$-degree of the $(i_0, i_1)$-uniform hypergraph not to exceed the quantity $\Del$, defined as follows.

\begin{defn}
  \label{dfn:Delta}
  For every $(i_0, i_1) \in \UU$ and every $(\ell_0, \ell_1) \in \{0, \ldots, i_0\} \times \{0, \ldots, i_1\}$ with $(\ell_0, \ell_1) \neq (0,0)$, we define the number $\Del$ using the following recursion:
  \begin{enumerate}[label=(\arabic*)]
  \item 
    Set $\Delta_{(\ell_0, \ell_1)}^{(k_0,k_1)} := \Delta_{(\ell_0, \ell_1)}(\HH)$ for all $(\ell_0, \ell_1) \in \{0, \ldots, k_0\} \times \{0, \ldots, k_1\} \setminus \{(0,0)\}$.\smallskip
  \item
    If $i_0 = k_0$ and $0 \le i_1 < k_1$, then
    \[
      \Del := \max \left\{ 2 \cdot \Delf{}{+1}{}{+1}, \, \frac{b}{v(\HH)} \cdot \Delf{}{}{}{+1} \right\}.
    \]
  \item
    If $0 < i_0 < k_0$ and $i_1 = 0$, then
    \[
      \Del := \max \left\{ 2 \cdot \Delf{+1}{}{+1}{}, \, \frac{b}{m} \cdot \Delf{}{}{+1}{} \right\}.
    \]
  \end{enumerate}
\end{defn}

The above recursive definition will be convenient in some parts of our analysis. In other parts, we shall require the following explicit formula for $\Del$, which one easily derives from Definition~\ref{dfn:Delta} using a straightforward induction on $k_0 + k_1 - i_0 - i_1$.

\begin{obs}
  \label{obs:Delta}
  For all $i_0$, $i_1$, $\ell_0$, and $\ell_1$ as in Definition~\ref{dfn:Delta},
  \[
    \Del = \max \left\{ 2^{d_0+d_1} \left(\frac{b}{v(\HH)}\right)^{k_1-i_1-d_1} \left(\frac{b}{m}\right)^{k_0-i_0-d_0} \Delta_{(\ell_0+d_0, \ell_1+d_1)}(\HH) \scolon 0 \le d_j \le k_j-i_j \right\}.
  \]
\end{obs}

For future reference, we note the following two simple corollaries of Observation~\ref{obs:Delta} and our assumptions on the maximum degrees of $\HH$, see~\eqref{eq:container-Del}. Suppose that $(i_0, i_1) \in \UU$. If $i_1 > 0$, then necessarily $i_0 = k_0$ and hence,
\begin{equation}
  \label{eq:Delta01}
  \begin{split}
    \Delta_{(0,1)}^{(i_0,i_1)} & \le \max \left\{ 2^{d_1} \left(\frac{b}{v(\HH)}\right)^{k_1-i_1-d_1} K \cdot \frac{b^{d_1}}{v(\HH)^{d_1+1}} \cdot e(\HH) \scolon 0 \le d_1 \le k_1 - i_1\right\} \\
    & \le 2^{k_1}K \left(\frac{b}{v(\HH)}\right)^{k_1-i_1} \frac{e(\HH)}{v(\HH)} = 2^{k_1}K \left(\frac{b}{v(\HH)}\right)^{k_1-i_1} \left(\frac{b}{m}\right)^{k_0-i_0} \frac{e(\HH)}{v(\HH)}.
  \end{split}  
\end{equation}
Moreover, if $i_0 > 0$ then 
\begin{equation}
  \label{eq:Delta10}
  \begin{split}
    \Delta_{(1,0)}^{(i_0,i_1)} & \le \max \left\{ 2^{d_0+d_1} \left(\frac{b}{v(\HH)}\right)^{k_1-i_1-d_1}\left(\frac{b}{m}\right)^{k_0-i_0-d_0} K \cdot \frac{b^{d_0+d_1}}{m^{d_0} \cdot v(\HH)^{d_1}} \cdot \frac{e(\HH)}{r} \right\}\\
    & \le 2^{k_0+k_1}K \left(\frac{b}{v(\HH)}\right)^{k_1-i_1} \left(\frac{b}{m}\right)^{k_0-i_0} \frac{e(\HH)}{r},
  \end{split}
\end{equation}
where the maximum is over all pairs $(d_0, d_1)$ of integers satisfying $0 \le d_j \le k_j - i_j$.

\begin{defn}
  \label{dfn:MDel}
  Given $(i_0, i_1) \in \UU$, $(\ell_0, \ell_1) \in \{0, \ldots, i_0\} \times \{0, \ldots, i_1\}$ with $(\ell_0, \ell_1) \neq (0,0)$, and an $(i_0, i_1)$-uniform hypergraph $\GG$, we define
  \[
    \MDel(\GG) = \left\{ (T_0, T_1) \in \binom{V(\GG)}{\ell_0} \times \binom{V(\GG)}{\ell_1} \scolon \deg_\GG(T_0,T_1) \ge \frac{1}{2} \cdot \Del \right\}.
  \]
\end{defn}

Finally, let us say that $c \in \{0,1\}$ is \emph{compatible} with $(i_0, i_1) \in \UU$ if the unique pair $(i_0', i_1')  \in \UU \cup \{(0,0)\}$ with $i_0' + i_1' = i_0 + i_1 - 1$ satisfies $i_c' = i_c - 1$ (and $i_{1-c}' = i_{1-c}$). By the definition of $\UU$, it follows that $1$ is compatible with $(i_0, i_1) \in \UU$ if and only if $i_1 > 0$.

\subsection{The algorithm}
\label{sec:algorithm}

We shall now define precisely a single round of the algorithm that we described informally in Section~\ref{sec:proof-outline}. To this end, fix some $(i_0, i_1) \in \UU$ and a compatible $c \in \{0,1\}$ and (as in the definition of a compatible $c$) set
\begin{equation}
  \label{eq:i-prime}
  i_c' = i_c - 1 \qquad \text{and} \qquad i_{1-c}' = i_{1-c}.
\end{equation}
Suppose that $\GG$ is an $(i_0, i_1)$-uniform hypergraph with $V(\GG) = V(\HH)$. A single round of the algorithm takes as input an arbitrary $h \in \FF(\GG)$ and outputs an $(i_0', i_1')$-uniform hypergraph $\GG^*$ satisfying $V(\GG^*) = V(\GG)$ and $h \in \FF(\GG^*)$ as well as a~set of vertices of $\GG$ on which $h$ takes the value $c$ at most $b$ times. Crucially, the number of possible outputs of the algorithm (over all possible input functions $h \in \FF(\GG)$) is at most $\binom{v(\HH)}{\le b}$.

Assume that there is an implicit linear order $\preccurlyeq$ on $V(\GG)$. The \emph{$c$-maximum vertex} of a hypergraph $\cA$ with $V(\cA) = V(\GG)$ is the $\preccurlyeq$-smallest vertex among those $v$ that maximise $|\{(A_0, A_1) \in \cA \scolon v \in A_c\}|$.

\medskip

\noindent
\textbf{The algorithm.}
Set $\cA^{(0)} := \GG$, let $S$ be the empty set, and let $\GGn^{(0)}$ be the empty $(i_0',i_1')$-uniform hypergraph on $V(\GG)$. Do the following for each integer $j \ge 0$ in turn:
\begin{enumerate}[label={(S\arabic*)}]
\item
  \label{item:alg-stop}
  If $|S| = b$ or $\cA^{(j)}$ is empty, then set $J := j$ and \STOP.
\item
  \label{item:alg-c-maximum}
  Let $v_j$ be the $c$-maximum vertex of $\cA^{(j)}$.
\item
  \label{item:alg-main-step}
  If $h(v_j) = c$, then add $j$ to the set $S$ and let
  \[
    \GGn^{(j+1)} := \GGn^{(j)} \cup \Big\{ \big( A_0 \setminus \{v_j\},  A_1 \setminus \{v_j\}\big) \scolon (A_0, A_1) \in \cA^{(j)} \text{ and } v_j \in A_c \Big\}.
  \]
\item
  \label{item:alg-cleanup}
  Let $\cA^{(j+1)}$ be the hypergraph obtained from $\cA^{(j)}$ by removing from it all pairs $(A_0, A_1)$ such that either of the following hold:
  \begin{enumerate}[label={(\alph*)}]
  \item
    \label{item:cleanup-1}
    $v_j \in A_c$;
  \item
    \label{item:cleanup-2}
    there exist $T_0 \subseteq A_0$ and $T_1 \subseteq A_1$, not both empty, such that 
    \[
      (T_0, T_1) \in \MDelp\big( \GGn^{(j+1)} \big)
    \]
 for some $\ell_0 \in \{ 0, \ldots, i_0'\}$ and $\ell_1 \in \{0, \ldots, i_1'\}$.
  \end{enumerate}
\end{enumerate}
Finally, set $\cA := \cA^{(J)}$ and $\GGn := \GGn^{(J)}$. Moreover, set
\[
  W := \big\{ 0, \dotsc, J-1 \big\} \setminus S = \Big\{ j \in \big\{ 0, \dotsc, J-1 \big\} \scolon h(v_j) \neq c\Big\}.
\]

\smallskip

Observe that the algorithm always stops after at most $v(\GG)$ iterations of the main loop. Indeed, since all constraints $(A_0, A_1)$ with $v_j \in A_c$ are removed from $\cA^{(j+1)}$ in part~\ref{item:cleanup-1} of step~\ref{item:alg-cleanup}, the vertex $v_j$ cannot be the $c$-maximum vertex of any $\cA^{(j')}$ with $j' > j$ and hence the map $\{0, \dotsc, J-1\} \ni j \mapsto v_j \in V(\GG)$ is injective. 

\subsection{The analysis}

We shall now establish some basic properties of the algorithm described in the previous subsection. To this end, let us fix some $(i_0, i_1) \in \UU$ and a compatible $c \in \{0,1\}$ and let $i_0'$ and $i_1'$ be the numbers defined in~\eqref{eq:i-prime}. Moreover, suppose that $\GG$ is an $(i_0, i_1)$-uniform hypergraph and that we have run the algorithm with input $h \in \FF(\GG)$ and obtained the $(i_0', i_1')$-uniform hypergraph $\GGn$, the integer $J$, the injective map $\{0, \dotsc, J-1\} \ni j \mapsto v_j \in V(\GG)$, and the partition of $\{0, \dotsc, J-1\}$ into $S$ and $W$ such that $h(v_j) = c$ if and only if $j \in S$. We first state two straightforward, but fundamental, properties of the algorithm.

\begin{obs}
  \label{obs:h-in-FF-Gn}
  If $h \in \FF(\GG)$, then $h \in \FF(\GGn)$.
\end{obs}

\begin{proof}
  Observe that $\GGn$ contains only constraints of the form:
  \begin{enumerate}[label={(\textit{\roman*})}]
  \item
    \label{item:constraint-i}
    $(A_0 \setminus \{v\}, A_1)$, where $v \in A_0$ and $h(v) = 0$, or
  \item
    \label{item:constraint-ii}
    $(A_0, A_1 \setminus \{v\})$, where $v \in A_1$ and $h(v) = 1$,
  \end{enumerate}
  where $(A_0, A_1) \in \GG$, see~\ref{item:alg-main-step}. Hence, if $h$ violated a constraint of type~\ref{item:constraint-i} (resp.~\ref{item:constraint-ii}) then $h$ would also violate the constraint $(A_0, A_1)$, as $h(v) = 0$ (resp.\ $h(v) = 1$).
\end{proof}

The next observation says that if the algorithm applied to two functions $h$ and $h'$ outputs the same set $\{v_j \scolon j \in S\}$, then the rest of the output is also the same. 

\begin{obs}
  \label{obs:number-of-containers}
Suppose that the algorithm applied to $h' \in \FF(\GG)$ outputs a hypergraph~$\GGn'$, an integer $J'$, a map $j \mapsto v_j'$, and a partition of $\{0, \dotsc, J'-1\}$ into $S'$ and $W'$. If $\{v_j \scolon j \in S\} = \{v_j' \scolon j \in S'\}$, then $\GGn = \GGn'$, $J = J'$, $v_j = v_j'$ for all $j$, and $W = W'$.
\end{obs}

\begin{proof}
  The only step of the algorithm that depends on the input function $h$ is~\ref{item:alg-main-step}. There, an index $j$ is added to the set $S$ if and only if $h(v_j) = c$. Therefore, the execution of the algorithm depends solely on the set $\{v_j \scolon j \in S\}$.
\end{proof}

The next two lemmas will allow us to maintain suitable upper and lower bounds on the degrees and densities of the hypergraphs obtained by applying the algorithm iteratively. The first lemma, which is the easier of the two, states that if all the maximum degrees of $\GG$ are appropriately bounded, then all the maximum degrees of $\GGn$ are also appropriately bounded.

\begin{lemma}\label{lemma:alg-analysis-degrees}
Given $(\ell_0, \ell_1) \in \{0, \ldots, i_0\} \times \{0, \ldots, i_1\}$ with $\ell_0 + \ell_1 \ge 2$ and $\ell_c > 0$, set $\ell_c' = \ell_c-1$ and $\ell_{1-c}' = \ell_{1-c}$. If $\Delta_{(\ell_0,\ell_1)}(\GG) \le \Del$, then $\Delta_{(\ell_0',\ell_1')}(\GGn) \le \Delp$.
\end{lemma}

\begin{proof}
Suppose (for a contradiction) that there exist sets $T_0'$ and $T_1'$, with $|T_0'| = \ell_0'$ and $|T_1'| = \ell_1'$, such that $\deg_{\GGn}(T_0', T_1') > \Delp$. Let $j$ be the smallest integer satisfying
  \[
    \deg_{\GGn^{(j+1)}}(T_0', T_1') > \frac{1}{2} \cdot \Delp
  \]
  and note that $j \ge 0$, since $\GGn^{(0)}$ is empty. We claim first that
   \begin{equation}
    \label{eq:deg-j-is-deg-final}
    \deg_{\GGn} (T_0', T_1') = \deg_{\GGn^{(j+1)}}(T_0', T_1').
  \end{equation}
  Indeed, observe that $(T_0', T_1') \in M^{(i_0',i_1')}_{(\ell_0',\ell_1')} \big( \GGn^{(j+1)} \big)$, and therefore the algorithm removes from $\cA^{(j)}$ (when forming $\cA^{(j+1)}$ in step~\ref{item:alg-cleanup}) all pairs $(A_0, A_1)$ such that $T_0' \subseteq A_0$ and $T_1' \subseteq A_1$. As a consequence, no further pairs $(A_0', A_1')$ with $T_0' \subseteq A_0'$ and $T_1' \subseteq A_1'$ are added to $\GGn$ in step~\ref{item:alg-main-step}.

  We next claim that
  \begin{equation}
    \label{eq:one-step-deg-change}
    \deg_{\GGn^{(j+1)}}(T_0', T_1') - \deg_{\GGn^{(j)}}(T_0', T_1') \le \Del.
  \end{equation}
  To see this, recall that when we extend $\GGn^{(j)}$ to $\GGn^{(j+1)}$ in step~\ref{item:alg-main-step}, we only add pairs $\big( A_0 \setminus \{v_j\}, A_1 \setminus \{v_j\} \big)$ such that $(A_0, A_1) \in \cA^{(j)} \subseteq \GG$ and $v_j \in A_c$. Therefore, setting $T_c = T_c' \cup \{v_j\}$ and $T_{1-c} = T_{1-c}'$, we have
  \begin{equation*}
    \deg_{\GGn^{(j+1)}}(T_0', T_1') - \deg_{\GGn^{(j)}}(T_0', T_1') \le \deg_{\GG}(T_0, T_1) \le \Delta_{(\ell_0, \ell_1)}(\GG) \le \Del,   
  \end{equation*}
  where the last inequality is by our assumption, as claimed.

  Combining~\eqref{eq:deg-j-is-deg-final} and~\eqref{eq:one-step-deg-change}, it follows immediately that
  \[
    \deg_{\GGn}(T_0',T_1') \le \frac{1}{2} \cdot \Delp + \Del \le \Delp,
  \]
  where the final inequality holds by Definition~\ref{dfn:Delta}. This contradicts our choice of $(T_0', T_1')$ and therefore the lemma follows.
\end{proof}

We are now ready for our final lemma, which is really the heart of the matter. We will show that if $\GG$ has sufficiently many edges and all of the maximum degrees of $\GG$ are appropriately bounded, then either the output hypergraph~$\GGn$ has sufficiently many edges or the value of $h(v)$ will be determined for sufficiently many vertices~$v$. We remark that here we shall use the assumption that $h$ takes the value $1$ at most $m$ times.

\begin{lemma}
  \label{lemma:alg-analysis-progress}
  Suppose that $|h^{-1}(1)| \le m$ and let $\alpha > 0$. If 
  \begin{enumerate}[label=(A\arabic*)]
  \item
    \label{item:assumption-edges}
    $e(\GG) \ge \alpha \cdot \big( \frac{b}{v(\HH)} \big)^{k_1-i_1} \big( \frac{b}{m} \big)^{k_0-i_0} e(\HH)$ and\smallskip
     \item
    \label{item:assumption-Delta}
    $\Delta_{(\ell_0, \ell_1)}(\GG) \le \Del$ for every $(0,0) \neq (\ell_0, \ell_1) \in \{0, \ldots, i_0\} \times \{0, \ldots, i_1\}$,\smallskip
  \end{enumerate}
  then at least one of the following statements is true:
  \begin{enumerate}[label=(P\arabic*)]
  \item
    \label{item:reduce-uniformity}
    $e(\GGn) \ge 2^{-i_0-i_1-1} \alpha \cdot \big( \frac{b}{v(\HH)} \big)^{k_1-i_1'} \big( \frac{b}{m} \big)^{k_0-i_0'} e(\HH)$.\smallskip
  \item
    \label{item:determine-many-0}
    $c = 1$ and $|W| \ge 2^{-k_1-1} K^{-1} \alpha \cdot v(\HH)$.\smallskip
  \item
    \label{item:determine-many-1}
    $c = 0$ and $|W| \ge 2^{-k_0-k_1-1} K^{-1} \alpha  \cdot r$.
  \end{enumerate}
\end{lemma}

\begin{proof}
  Suppose first that $c = 0$ and observe that\footnote{Recall that $\GGn$ (and $\GGn^{(j)}$ etc.) are multi-hypergraphs and that edges are counted with multiplicity.}
  \begin{equation}
    \label{eq:eG-sum}
    e(\GGn) = \sum_{j \in S} \left( e(\GGn^{(j+1)}) - e(\GGn^{(j)}) \right) = \sum_{j \in S} \deg_{\cA^{(j)}}(\{v_j\}, \emptyset),
  \end{equation}
  since $e(\GGn^{(j+1)}) - e(\GGn^{(j)}) = \deg_{\cA^{(j)}}(\{v_j\}, \emptyset)$ for each $j \in S$ and $\GGn^{(j+1)} = \GGn^{(j)}$ for each $j \not\in S$. To bound the right-hand side of~\eqref{eq:eG-sum}, we count the edges removed from $\cA^{(j)}$ in \ref{item:cleanup-1} and \ref{item:cleanup-2} of step \ref{item:alg-cleanup}, which gives
  \[
    e(\cA^{(j)}) - e(\cA^{(j+1)}) \le \deg_{\cA^{(j)}}(\{v_j\}, \emptyset) + \sum_{(\ell_0, \ell_1)} \big| \MDelp(\GGn^{(j+1)}) \setminus \MDelp(\GGn^{(j)}) \big| \cdot \Delta_{(\ell_0, \ell_1)}(\GG).
  \]
  Summing over $j \in \{0, \ldots, J-1\}$, it follows (using~\eqref{eq:eG-sum}) that
  \[
    e(\GG) - e(\cA) \le e(\GGn) + |W| \cdot \Delta_{(1,0)}(\GG) + \sum_{(\ell_0, \ell_1)} \big| \MDelp(\GGn) \big| \cdot \Del,
  \]
  since $\cA = \cA^{(J)} \subseteq \ldots \subseteq \cA^{(0)} =\GG$ and $\Delta_{(\ell_0, \ell_1)}(\GG) \le \Del$ by~\ref{item:assumption-Delta}. Observe also that if $c = 1$, then we obtain an identical bound, with $\Delta_{(1,0)}(\GG)$ replaced by $\Delta_{(0,1)}(\GG)$.

  In order to discuss both cases simultaneously, we set $\chi(0) = (1,0)$ and $\chi(1) = (0,1)$. Observe that
  \begin{equation}
    \label{eq:e-G0-increment-estimate}
    \Delta_{\chi(c)}(\cA) \le \Delta_{\chi(c)}(\cA^{(j)}) \le \Delta_{\chi(c)}(\GG) \le \Delta_{\chi(c)}^{(i_0,i_1)},
  \end{equation}
  since $\cA \subseteq \cA^{(j)} \subseteq \GG$ and $\GG$ satisfies~\ref{item:assumption-Delta}. It follows that, for both $c \in \{0, 1\}$,
  \begin{equation}
    \label{eq:Gn-final-0}
    e(\GG) - e(\cA) \le e(\GGn) + |W| \cdot \Delta_{\chi(c)}^{(i_0,i_1)} + \sum_{(\ell_0, \ell_1)} \big|\MDelp(\GGn) \big| \cdot \Del.
  \end{equation}
  Now, recall that $v_j$ is the $c$-maximum vertex of $\cA^{(j)}$ and observe that therefore, by~\eqref{eq:eG-sum} and~\eqref{eq:e-G0-increment-estimate},  
  \begin{equation}
    \label{eq:e-Gn-S}
    e(\GGn) = \sum_{j \in S} \Delta_{\chi(c)}\big(\cA^{(j)}\big) \ge |S| \cdot \Delta_{\chi(c)}(\cA) = b \cdot \Delta_{\chi(c)}(\cA),
  \end{equation}
  where the equality is due to the fact that $|S| \neq b$ only when $\cA$ is empty, see step~\ref{item:alg-stop}. 
  
 Next, to bound the sum in~\eqref{eq:Gn-final-0}, observe that, by Definition~\ref{dfn:MDel}, we have
 \[
   \big| \MDelp(\GGn) \big| \cdot \frac{1}{2} \cdot \Delnp \le \sum_{(T_0,T_1) \in \binom{V(\GG)}{\ell_0} \times \binom{V(\GG)}{\ell_1}} \deg_{\GGn}(T_0, T_1) = \binom{i_0'}{\ell_0} \binom{i_1'}{\ell_1} \cdot e(\GGn)
 \]
 for each $(\ell_0, \ell_1)$ and therefore
 \begin{equation}
   \label{eq:e-Gn-MDel-final}
   \begin{split}
     \sum_{(\ell_0, \ell_1)} \big| \MDelp(\GGn) \big| \cdot \Del & \le 2 \cdot \sum_{(\ell_0, \ell_1)} \binom{i_0'}{\ell_0} \binom{i_1'}{\ell_1} \cdot e(\GGn) \cdot \left(\Del / \Delnp\right)  \\
     & \le 2 \cdot \big( 2^{i_0'+i_1'} - 1 \big) \cdot e(\GGn) \cdot \max_{(\ell_0, \ell_1)} \left\{\Del / \Delnp\right\}.
   \end{split}
 \end{equation}
 We claim that $\Del / \Delnp \le m / b$ if $c = 0$ and $\Del / \Delnp \le v(\HH) / b$ if $c = 1$. Indeed, both inequalities following directly from Definition~\ref{dfn:Delta}, since if $c = 0$, then $(i_0', i_1') = (i_0-1, i_1)$, and if $c = 1$, then $(i_0', i_1') = (i_0, i_1-1)$. We split the remainder of the proof into two cases, depending on the value of $c$.

  Suppose first that $c = 1$ and observe that substituting~\eqref{eq:e-Gn-MDel-final} into~\eqref{eq:Gn-final-0} yields, using the bound $\Del / \Delnp \le v(\HH) / b$,
  \begin{equation}
    \label{eq:Gn-final-case1}
    e(\GG) - e(\cA) \le e(\GGn) + |W| \cdot \Delta_{(0,1)}^{(i_0,i_1)} + 2 \cdot \big( 2^{i_0'+i_1'} - 1 \big) \cdot e(\GGn) \cdot \frac{v(\HH)}{b}.
  \end{equation}
  Moreover, by~\eqref{eq:e-Gn-S}, and since $i_1 \ge 1$ when $c = 1$, we have
  \begin{equation}
    \label{eq:Delta-estimates-1}
    \frac{e(\GGn)}{b} \ge \Delta_{(0,1)}(\cA) \ge \frac{i_1 \cdot e(\cA)}{v(\cA)} \ge \frac{e(\cA)}{v(\HH)},
  \end{equation}
  since the maximum degree of a hypergraph is at least as large as its average degree. Combining~\eqref{eq:Gn-final-case1} and~\eqref{eq:Delta-estimates-1}, we obtain
  \begin{equation}
    \label{eq:eGG-eGGn-W-1}
    \begin{split}
      e(\GG) & \le e(\GGn) \cdot \frac{v(\HH)}{b} \cdot \left(\frac{b}{v(\HH)} + 1 + 2^{i_0'+i_1'+1} - 2\right) + |W| \cdot \Delta_{(0,1)}^{(i_0,i_1)} \\
      & \le e(\GGn) \cdot \frac{v(\HH)}{b} \cdot 2^{i_0+i_1} + |W| \cdot \Delta_{(0,1)}^{(i_0,i_1)},
    \end{split}
  \end{equation}
  since $b \le v(\HH)$. Now, if the first summand on the right-hand side of~\eqref{eq:eGG-eGGn-W-1} exceeds $e(\GG) / 2$, then~\ref{item:assumption-edges} implies~\ref{item:reduce-uniformity}, since $(i_0', i_1') = (i_0, i_1-1)$. Otherwise, the second summand is at least $e(\GG)/2$ and by~\ref{item:assumption-edges} and~\eqref{eq:Delta01},
  \[
    |W| \ge \frac{e(\GG)}{2 \cdot \Delta_{(0,1)}^{(i_0,i_1)}} \ge \frac{\alpha}{2^{k_1+1}K} \cdot v(\HH),
  \]
  which is~\ref{item:determine-many-0}.

  The case $c = 0$ is slightly more delicate; in particular, we will finally use our assumption that $|h^{-1}(1)| \le m$. Observe first that if $c = 0$, then substituting~\eqref{eq:e-Gn-MDel-final} into~\eqref{eq:Gn-final-0} yields, using the bound $\Del / \Delnp \le m / b$,
 \begin{equation}
    \label{eq:Gn-final-case2}
    e(\GG) - e(\cA) \le e(\GGn) + |W| \cdot \Delta_{(1,0)}^{(i_0,i_1)} + \big( 2^{i_0+i_1} - 2 \big) \cdot e(\GGn) \cdot \frac{m}{b},
  \end{equation}
  cf.~\eqref{eq:Gn-final-case1}. We claim that
  \begin{equation}
    \label{eq:Delta-estimates-0}
    \frac{e(\GGn)}{b} \ge \Delta_{(1,0)}(\cA) \ge \frac{e(\cA)}{m}.
  \end{equation}
 The first inequality follows from~\eqref{eq:e-Gn-S}, so we only need to prove the second inequality. To do so, observe that $\GG$ is an $(i_0, 0)$-uniform hypergraph (since $c = 0$) and therefore each function in $\FF(\GG)$ must take the value $1$ on at least one element of each set $A_0$ such that $(A_0, \emptyset) \in \GG$. Now, recall that $h \in \FF(\GG)$, that $\cA \subseteq \GG$, and that $h$ takes the value $1$ at most $m$ times. It follows that $e(\cA) \le m \cdot \Delta_{(1,0)}(\cA)$, as claimed.
  
 Combining~\eqref{eq:Gn-final-case2} and~\eqref{eq:Delta-estimates-0}, we obtain (cf.~\eqref{eq:eGG-eGGn-W-1})
  \begin{equation}
    \label{eq:eGG-eGGn-W-0}
    \begin{split}
      e(\GG) & \le e(\GGn) \cdot \frac{m}{b} \cdot \left(\frac{b}{m} + 1 + 2^{i_0 + i_1} - 2\right) + |W| \cdot \Delta_{(1,0)}^{(i_0,i_1)} \\
      & \le e(\GGn) \cdot \frac{m}{b} \cdot 2^{i_0+i_1} + |W| \cdot \Delta_{(1,0)}^{(i_0,i_1)},
    \end{split}
  \end{equation}
  since $b \le m$. Now, if the first summand on the right-hand side of~\eqref{eq:eGG-eGGn-W-1} exceeds $e(\GG) / 2$, then~\ref{item:assumption-edges} implies~\ref{item:reduce-uniformity}, since $(i_0', i_1') = (i_0-1, i_1)$. Otherwise, the second summand is at least $e(\GG)/2$ and by~\ref{item:assumption-edges} and~\eqref{eq:Delta10},
  \[
    |W| \ge \frac{e(\GG)}{2 \cdot \Delta_{(1,0)}^{(i_0,i_1)}} \ge \frac{\alpha}{2^{k_0+k_1+1}K} \cdot r,
  \]
  which is~\ref{item:determine-many-1}.
\end{proof}

\subsection{Construction of the container}
\label{sec:constr-container}

In this section, we present the construction of containers for functions in $\FF_{\le m}(\HH)$ and analyse their properties, thus proving Theorem~\ref{thm:container}. For each $s \in \{0, \ldots, k_0+k_1\}$, define
\[
  \alpha_s = 2^{-s(k_0+k_1+1)} \qquad \text{and} \qquad \beta_s = \alpha_s \cdot \left(\frac{b}{v(\HH)}\right)^{\min\{{k_1,s\}}}\left(\frac{b}{m}\right)^{\max\{0,s-k_1\}}.
\]
Given an $h \in \FF_{\le m}(\HH)$, we construct the container $f^*_h$ for $h$ using the following procedure.

\medskip
\noindent
\textbf{Construction of the container.}
Let $\HH^{(k_0,k_1)} = \HH$, let $S_0 = S_1 = \emptyset$, and let $(i_0, i_1) = (k_0, k_1)$. Do the following for $s = 0, \ldots, k_0+k_1-1$:
\begin{enumerate}[label=(C\arabic*)]
\item
  Let $c \in \{0,1\}$ be the number that is compatible with $(i_0, i_1)$ and let $(i_0', i_1')$ be the pair defined by $i_c' = i_c-1$ and $i_{1-c}' = i_{1-c}$.
\item
  \label{item:constr-run-algorithm}  
  Run the algorithm with $\GG \leftarrow \HH^{(i_0, i_1)}$ to obtain the $(i_0', i_1')$-uniform hypergraph $\GGn$, the sequence $v_0, \ldots, v_{J-1} \in V(\HH)$, and the partition $\{0, 1, \ldots, J-1\} = S \cup W$.
\item
  \label{item:constr-update-signature}
  Let $S_c \leftarrow S_c \cup \{v_j \scolon j \in S\}$.
\item
  \label{item:constr-determine-many}
  If $e(\GGn) < \beta_{s+1} \cdot e(\HH)$, then define $f^*_h \colon V(\HH) \to \{0, 1, *\}$, the container for $h$, by
  \[
    f^*_h(v) = \left\{
    \begin{array}{cl}
      1-c & \text{if $v = v_j$ for some $j \in W$}, \\
      * & \text{otherwise},
    \end{array} 
    \right.
  \]
  and \STOP.
\item
  \label{item:constr-reduce-uniformity}
  Otherwise, let $\HH^{(i_0', i_1')} \leftarrow \GGn$ and $(i_0, i_1) \leftarrow (i_0', i_1')$ and \texttt{CONTINUE}.
\end{enumerate}

We will show that the above procedure indeed constructs containers for $\FF_{\le m}(\HH)$ that have the desired properties. To this end, we first claim that for each pair $(i_0, i_1) \in \UU \cup \{(0,0)\}$, the hypergraph $\HH^{(i_0,i_1)}$, if it was defined, satisfies:
\begin{enumerate}[label=(\textit{\roman*})]
\item
  \label{item:container-propty-1}
  $h \in \FF(\HH^{(i_0,i_1)})$ and
\item
  \label{item:container-propty-2}
  $\Delta_{(\ell_0,\ell_1)}(\HH^{(i_0,i_1)}) \le \Del$ for every $(0,0) \ne (\ell_0, \ell_1) \in \{0, \ldots, i_0\} \times \{0, \ldots, i_1\}$.
\end{enumerate}
Indeed, one may easily prove~\ref{item:container-propty-1} and~\ref{item:container-propty-2} by induction on $(k_0+k_1) - (i_0+i_1)$. The basis of the induction is trivial as $\HH^{(k_0,k_1)} = \HH$, see Definition~\ref{dfn:Delta}. The inductive step follows immediately from Observation~\ref{obs:h-in-FF-Gn} and Lemma~\ref{lemma:alg-analysis-degrees}.

Second, we claim that for each input $h \in \FF_{\le m}(\HH)$, step~\ref{item:constr-determine-many} is called for some $s$ and hence the function $f^*_h \colon V(\HH) \to \{0,1,*\}$  is defined. If this were not true, the condition in step~\ref{item:constr-reduce-uniformity} would be met $k_0+k_1$ times and, consequently, we would finish with a non-empty $(0,0)$-uniform hypergraph $\HH^{(0,0)}$, i.e., we would have $(\emptyset, \emptyset) \in \HH^{(0,0)}$. But this contradicts~\ref{item:container-propty-1}, since no function satisfies the empty constraint and thus $h \not\in \FF(\HH^{(0,0)})$. 

Suppose, therefore, that step~\ref{item:constr-determine-many} is executed when $\GG = \HH^{(i_0,i_1)}$ for some $(i_0,i_1) \in \UU$, and note that $s = (k_0+k_1) - (i_0+i_1)$. We claim that $e(\HH^{(i_0,i_1)}) \ge \beta_s e(\HH)$. Indeed, this is trivial if $s = 0$, whereas if $s > 0$ and this were not true, then we would have executed step~\ref{item:constr-determine-many} at the previous step. We therefore have
\[
  e(\GG) = e(\HH^{(i_0,i_1)}) \ge \beta_s \cdot e(\HH) \qquad \text{and} \qquad e(\GGn) < \beta_{s+1} \cdot e(\HH),
\]
which, by Lemma~\ref{lemma:alg-analysis-progress} and~\ref{item:container-propty-2}, implies that either~\ref{item:determine-many-0} or~\ref{item:determine-many-1} of Lemma~\ref{lemma:alg-analysis-progress} holds. Note that if $c = 1$, then $k_1 \ge i_1 > 0$ and we have
\[
  |(f^*_h)^{-1}(0)| \ge 2^{-k_1-1}K^{-1} \alpha_s \cdot v(\HH) \ge \alpha_{k_0+k_1}K^{-1} v(\HH) = \delta v(\HH),
\]
where $\delta = 2^{-(k_0+k_1)(k_0+k_1+1)} K^{-1}$. On the other hand, if $c = 0$, then $k_0 \ge i_0 > 0$ and  
\[
  |(f^*_h)^{-1}(1)| \ge 2^{-k_0-k_1-1}K^{-1} \alpha_s \cdot r \ge \alpha_{k_0+k_1}K^{-1} r = \delta r.
\]
This verifies that $f^*_h$ satisfies property~\ref{item:container-2} from the statement of Theorem~\ref{thm:container}.

To complete the proof, we need to show that $f^*$ decomposes as $f^* = f \circ g$ for some $g \colon \FF_{\le m}(\HH) \to \binom{V(\HH)}{\le k_0 b} \times \binom{V(\HH)}{\le k_1 b}$ and to verify that properties~\ref{item:container-1} and~\ref{item:container-3} from the statement of the theorem hold. We claim that one may take $g(h) = (S_0, S_1)$, where $S_0$ and $S_1$ are the sets constructed by the above procedure, see~\ref{item:constr-update-signature}. To this end, it suffices to show that if for some $h, h' \in \FF(\HH)$ the above procedure produces the same pair $(S_0, S_1)$, then $f^*_h = f^*_{h'}$. To see this, observe first that the set $S$ defined in step~\ref{item:constr-run-algorithm} is precisely the set of all indices $j \in \{0, \ldots, J-1\}$ that satisfy $v_j \in S_c$. Indeed, the former set is contained in the latter by construction, see~\ref{item:constr-update-signature}. The reverse inclusion holds because 
\[
  S = \big\{ j \in \{ 0, \dotsc, J-1 \} \scolon h(v_j) = c \big\}
\]
and $h(v) = c$ for every $v \in S_c$. By Observation~\ref{obs:number-of-containers}, it follows that the output of the algorithm depends only on the pair $(S_0,S_1)$ and hence $f^*_h = f^*_{h'}$, as claimed.

Finally, observe that $S_0 \subseteq h^{-1}(0)$ and $S_1 \subseteq h^{-1}(1)$, by construction, and that $h$ belongs to the cylinder $f(g(h)) = f^*_h$, since $h(v) = 1 - c$ for every $v = v_j$ with $j \in W$, by the definition of $W = \big\{ j \in \{ 0, \dotsc, J-1 \} \scolon h(v_j) \neq c \big\}$. This verifies properties~\ref{item:container-1} and~\ref{item:container-3} and hence completes the proof of Theorem~\ref{thm:container}. \qed

\subsection{{Derivation of Theorem~\ref{thm:container-mono}}}
\label{sec:derivation-theorem}

We conclude this part of the paper with the easy derivation of Theorem~\ref{thm:container-mono} from Theorem~\ref{thm:container}. Given a $k$-uniform hypergraph $\HH$ satisfying the assumptions of Theorem~\ref{thm:container-mono} for some $b$ and $r$, one may invoke Theorem~\ref{thm:container} with $\HH_{\ref{thm:container}}$ being the $(0, k)$-uniform hypergraph with the same vertex set as $\HH$ whose edges are all pairs $(\emptyset, A)$ such that $A$ is an edge of $\HH$. Since $k_0 = 0$ and $\Delta_{(0, \ell)}(\HH_{\ref{thm:container}}) = \Delta_\ell(\HH)$ for every $\ell \in \{1, \dotsc, k\}$, one can easily check that $\HH_{\ref{thm:container}}$ satisfies the assumptions of Theorem~\ref{thm:container} with the same $b$, $m_{\ref{thm:container}} \leftarrow v(\HH)$, and $K_{\ref{thm:container}} \leftarrow v(\HH)/r$.

Now, observe that the family $\FF(\HH_{\ref{thm:container}})$ comprises precisely the characteristic functions of all independent sets of $\HH$ and that $\FF_{\le m}(\HH_{\ref{thm:container}}) = \FF_{\le v(\HH)}(\HH_{\ref{thm:container}}) = \FF(\HH_{\ref{thm:container}})$. Given an independent set $I \in \II(\HH)$, let $h \in \FF(\HH_{\ref{thm:container}})$ be its characteristic function. Let $S := (\emptyset, S_1) = g_{\ref{thm:container}}(h)$ and $X := f_{\ref{thm:container}}(S)^{-1}(0)$ and set $g(I) := S_1$ and $f(g(I)) := V(\HH) \setminus X$. Recalling again that $k_0 = 0$, it is straightforward to verify that properties~\ref{item:container-1},~\ref{item:container-2}, and~\ref{item:container-3} from the statement of Theorem~\ref{thm:container} imply the assertion of Theorem~\ref{thm:container-mono}. \qed

\section{Robust balanced stability for induced $C_4$s}

\label{sec:robust-stability}

Recall from Section~\ref{sec:robust-stability-intro} that a pregraph is a pair $(M, E)$ of disjoint subsets of $E(K_n)$. The elements of $E$ are called edges whereas the elements of $M$ are called mixed edges. A~good copy of $C_4$ in a pregraph $(M, E)$ is a copy of $C_4$ in $M$ whose vertex set is independent in $E$. In particular, the vertex set of each good copy of $C_4$ induces four, five, or six edges of $M$, four of which play the roles of edges of $C_4$.\footnote{If the vertex set of a good copy of $C_4$ induces six mixed edges, then there are three choices for these four edges, each corresponding to a different embedding of $C_4$ into $K_4$.}

Given a pregraph $\cP = (M,E)$, we define three hypergraphs with vertex set $M$, denoted $\HH_0^\cP$, $\HH_1^\cP$, and $\HH_2^\cP$. The $(i,4)$-uniform hypergraph $\HH_i^\cP$ comprises all pairs $(A,B)$ such that $B$ is a good copy of $C_4$ and $A$ is the set of the remaining $i$ mixed edges induced by the vertex set of this copy (which induces exactly $4+i$ edges of $M$). Recall that we say that a graph $G$ is $\eps$-close to $K_\ell$ if one can transform $G$ into $K_\ell$ by adding or deleting at most $\eps {\ell \choose 2}$ edges. The following theorem, a robust stability statement for good copies of the $4$-cycle in a pregraph, is the main result of this section.

\begin{thm}
  \label{thm:robust-stability-refined}
  For every $\eps > 0$, there exist positive constants $\beta$, $\delta$, $\lambda$, and $C$ such that the following holds for all $\ell$ and $n$ satisfying $\ell \ge C \sqrt{n}$. Suppose that $\cP = (M,E)$ is a pregraph on $n$ vertices with $e(E) \le \binom{\ell}{2}$ and either
  \begin{enumerate}[label=(M\arabic*)]
  \item
    \label{item:supersaturation-cond}
    $e(M) \ge 4\ell n$, or
  \item
    \label{item:stability-cond}
    $e(M) \ge (1 - \delta)\ell n$, $E$ is not $\eps$-close to $K_\ell$, and $\ell \le \lambda n$, or
  \item
    \label{item:non-split-cond}
   there exists $U \subseteq V(K_n)$ with $|U| = \ell$, $e_E(U) \ge (1-\eps)\binom{\ell}{2}$, and $e_M(U^c) \ge 7\sqrt{\eps}\ell n$.
  \end{enumerate}
  Then there exist $i \in \{0,1,2\}$ and $\HH_i \subseteq \HH_i^\cP$ such that
  \[
    e(\HH_i) \ge \beta \ell^4, \qquad \Delta_{(0,1)}(\HH_i) \le \frac{\ell^3}{n}, \qquad \text{and} \qquad \Delta_{(0,2)}(\HH_i) \le \ell
  \]
  and, if $i > 0$, then also $\Delta_{(1,0)}(\HH_i) \le \ell^2$.
\end{thm}

Let us say that an $(i,4)$-uniform hypergraph $\HH_i$ is \emph{permissible} if it satisfies both (all three, if $i > 0$) maximum degree conditions stated in Theorem~\ref{thm:robust-stability-refined}. We shall thus be looking for a permissible subhypergraph $\HH_i \subseteq \HH_i^\cP$, for some $i \in \{0,1,2\}$, that has $\Omega(\ell^4)$ edges. We shall build the $\HH_0$, $\HH_1$, and $\HH_2$ by adding to them one edge at a time, making sure that we stay within the class of permissible hypergraphs, until one of them has sufficiently many edges. (Trivially, an empty hypergraph is permissible.) 

It will be convenient to use the following nomenclature. A pair $(S, T)$ of disjoint sets of edges of $K_n$ is \emph{saturated} in a hypergraph $\HH$ if $\deg_\HH(S, T)$ attains or exceeds its maximum permitted value. That is, if
\begin{enumerate}[label=(\textit{\roman*})]
\item
  \label{item:saturated-01}
  $(|S|,|T|) = (0,1)$ and $\deg_\HH(S, T) \ge \lfloor \ell^3/n \rfloor$, or
\item
  \label{item:saturated-10}
  $(|S|,|T|) = (1,0)$ and $\deg_\HH(S, T) \ge \ell^2$, or
\item
  \label{item:saturated-02}
  $(|S|,|T|) = (0,2)$ and $\deg_\HH(S, T) \ge \ell$.
\end{enumerate}
Thus, in the setting of Theorem~\ref{thm:robust-stability-refined}, we shall be looking for an $i \in \{0,1,2\}$ and an edge of $\HH_i^\cP \setminus \HH_i$ which does not contain any saturated pair. We first show how to deduce Theorem~\ref{thm:robust-stability-refined} from the following, seemingly weaker, statement by performing an appropriate preprocessing of the pregraph $\cP$. This preprocessing of $\cP$ will `disable' all saturated pairs of types~\ref{item:saturated-01} and~\ref{item:saturated-10}, so that we will only have to worry about pairs of type~\ref{item:saturated-02}.

\begin{thm}
  \label{thm:robust-stability-increment}
  For every $0 < \eps \le 1/2$, there exist positive constants $\beta$, $\delta$, $\lambda$, and $C$ such that the following holds for all $\ell$ and $n$ satisfying $\ell \ge C \sqrt{n}$.  Suppose that $\cP = (M,E)$ is a pregraph on $n$ vertices with $e(E) \le \binom{\ell}{2}$ and either
  \begin{enumerate}[label=(M\arabic**)]
  \item
    \label{item:supsersaturation-cond-incr}
    $e(M) \ge 3\ell n$ or
  \item
    \label{item:stability-cond-incr}
    $e(M) \ge (1 - \delta)\ell n$, $E$ is not $\eps$-close to $K_\ell$, and $\ell \le \lambda n$.
  \end{enumerate}
  Then for any collection $\CC$ of at most $12 \beta \ell^3$ pairs of elements of $M$, there exist at least $3\beta\ell^4$ good copies of $C_4$ in $\cP$ that contain no pair from $\CC$.
\end{thm}

\begin{proof}[Derivation of Theorem~\ref{thm:robust-stability-refined} from Theorem~\ref{thm:robust-stability-increment}]
  Given $0 < \eps \le 2$,\footnote{Note that the result for $\eps > 2$ is implied by the statement for $\eps = 2$, since condition (M3) is then stronger than condition (M1), and every graph with at most ${\ell \choose 2}$ edges is $2$-close to $K_\ell$.} let $\beta_{\ref{thm:robust-stability-increment}}$, $\delta_{\ref{thm:robust-stability-increment}}$, $\lambda_{\ref{thm:robust-stability-increment}}$, and $C_{\ref{thm:robust-stability-increment}}$ be the constants whose existence is asserted by Theorem~\ref{thm:robust-stability-increment} with $\eps_{\ref{thm:robust-stability-increment}} \leftarrow \eps/4$ and let
  \[
    \delta = \min\left\{ \frac{\delta_{\ref{thm:robust-stability-increment}}}{3}, \, \frac{\eps}{10} \right\}, \quad \beta = \min\left\{ \frac{\eps^2 \beta_{\ref{thm:robust-stability-increment}}}{4}, \, \frac{\delta}{20}\right\}, \quad \lambda = \frac{\lambda_{\ref{thm:robust-stability-increment}}}{2}, \quad \text{and} \quad C = \frac{C_{\ref{thm:robust-stability-increment}}}{\sqrt{\eps}}.
  \]
  Suppose that a pregraph $\cP = (M,E)$ satisfies the assumptions of Theorem~\ref{thm:robust-stability-refined}. We shall build the (initially empty) hypergraphs $\HH_0$, $\HH_1$, and $\HH_2$ edge by edge, making sure that we stay within the class of permissible hypergraphs, until one of them has sufficiently many edges. To this end, suppose that we have succeeded in constructing some permissible $\HH_0$, $\HH_1$, and $\HH_2$, but each of them has fewer than $\beta \ell^4$ edges. We shall modify the pregraph $\cP$ by removing from $M$ all mixed edges $f$ for which there exists $i \in \{0,1,2\}$ such that either $(\emptyset, \{f\})$ or $(\{f\},\emptyset)$ (or both) is saturated in $\HH_i$. This will ensure that every good copy of $C_4$ that we will later find in this modified colouring will not contain any saturated pair $(S,T)$ of type~\ref{item:saturated-01} or~\ref{item:saturated-10}. To achieve this, we first move all mixed edges $f$ for which $(\{f\}, \emptyset)$ is saturated in either $\HH_1$ or $\HH_2$ from $M$ to $E$ and then move all $f$ for which $(\emptyset, \{f\})$ is saturated in any of the $\HH_i$ from $M$ to an initially empty set $N$. Denote the modified pregraph by $\cP' = (M', E')$. Observe, crucially, that each good copy of $C_4$ in $\cP'$ is also good in $\cP$, as $E' \supseteq E$ and $M' \subseteq M$. Moreover, each such copy yields an edge of one of the $\HH_i^\cP$ with no saturated pair of type~\ref{item:saturated-01} or~\ref{item:saturated-10}, where $4+i$ is the number of edges of $M' \cup N$ induced by the vertex set of this $4$-cycle.\footnote{The four edges forming a good copy of $C_4$ in $\cP'$ belong to $M'$, but the remaining two edges induced by the vertex set of this cycle could belong to $N \subseteq M \setminus M'$.}

  Let $\ell' = \lfloor (1+\delta)\ell \rfloor$. As each of the $\HH_i$ has fewer than $\beta \ell^4$ edges, then
 $$   e(E' \setminus E) \le \sum_{i = 1}^2 \frac{ie(\HH_i)}{\ell^2} < 3\beta\ell^2 \le \frac{\delta\ell^2}{2}$$
 and
$$    e(M \setminus M') \le e(E' \setminus E) + \sum_{i=0}^2 \frac{4e(\HH_i)}{\lfloor \ell^3/n \rfloor} < 3\beta\ell^2 + 13\beta\ell n < 20\beta\ell n \le \delta \ell n.$$
  In particular,
  \[
    e(E') \le \binom{\ell}{2} + \frac{\delta\ell^2}{2} \le \binom{\ell'}{2}.
  \]
  Moreover, if $e(M) \ge 4\ell n$, then $e(M') \ge 3\ell'n$, and if $e(M) \ge (1-\delta) \ell n$, then $e(M') \ge (1-3\delta)\ell'n \ge (1-\delta_{\ref{thm:robust-stability-increment}})\ell'n$. Finally, if $E'$ is $(\eps/4)$-close to $K_{\ell'}$, then $E$ is $\eps$-close to $K_\ell$, as $e(K_{\ell'}) - e(K_\ell) \le 2\delta\ell^2$, and $\delta \le \eps/10$. Therefore, if $\cP$ satisfies the assumptions of Theorem~\ref{thm:robust-stability-refined} with either~\ref{item:supersaturation-cond} or~\ref{item:stability-cond}, then $\cP'$ satisfies the assumptions of Theorem~\ref{thm:robust-stability-increment} with $\eps_{\ref{thm:robust-stability-increment}} \leftarrow \eps/4$ and $\ell_{\ref{thm:robust-stability-increment}} \leftarrow \ell'$, see~\ref{item:supsersaturation-cond-incr} and~\ref{item:stability-cond-incr}.

  Now, let $\CC$ be the collection of all $T$ such that $(\emptyset, T)$ is a saturated pair of type~\ref{item:saturated-02} in one of the $\HH_i$ and observe that
  \[
    |\CC| \le \sum_{i=0}^2 \frac{4e(\HH_i)}{\ell} < 12\beta\ell^3 \le 12\beta_{\ref{thm:robust-stability-increment}} (\ell')^3,
  \]
  as each edge of $\HH_i$ contains at most four such saturated pairs (if $f_1, f_2 \in M$ do not share a vertex, then $\deg_{\HH_i^\cP}(\emptyset, \{f_1, f_2\}) \le 2$). Therefore, if $\cP$ satisfies the assumptions of Theorem~\ref{thm:robust-stability-refined} with either~\ref{item:supersaturation-cond} or~\ref{item:stability-cond}, then we may invoke Theorem~\ref{thm:robust-stability-increment} to find at least $3\beta_{\ref{thm:robust-stability-increment}} (\ell')^4 \ge 3 \beta \ell^4$ good copies of $C_4$ in $\cP'$, none of which contains a pair from $\CC$.
  
  On the other hand, if $\cP$ satisfies the assumptions of Theorem~\ref{thm:robust-stability-refined} with~\ref{item:non-split-cond}, then $\cP'$ restricted to the set $U^c$ satisfies the assumptions of Theorem~\ref{thm:robust-stability-increment} with $\ell_{\ref{thm:robust-stability-increment}} \leftarrow 2 \sqrt{\eps} \ell$, as
  \[
    e_{E'}(U^c) \le e_E(U^c) + e(E' \setminus E) \le \eps\binom{\ell}{2} + \frac{\delta\ell^2}{2} \le \eps\ell^2 \le \binom{2\sqrt{\eps}\ell}{2}
  \]
  and
  \[
    e_{M'}(U^c) \ge e_M(U^c) - e(M \setminus M') \ge 7\sqrt{\eps}\ell n - \delta\ell n \ge 6\sqrt{\eps}\ell n \ge 3(2\sqrt{\eps}\ell)|U^c|,
  \]
  see~\ref{item:supsersaturation-cond-incr}. Since
  \[
    2\sqrt{\eps}\ell \ge 2\sqrt{\eps} \cdot C\sqrt{n} \ge C_{\ref{thm:robust-stability-increment}}\sqrt{|U^c|} \quad \text{and} \quad |\CC| < 12\beta \ell^3 \le 12\beta_{\ref{thm:robust-stability-increment}} \cdot \left(2\sqrt{\eps}\ell\right)^3,
  \]
  we may again invoke Theorem~\ref{thm:robust-stability-increment} to find at least $3\beta_{\ref{thm:robust-stability-increment}} (2\sqrt{\eps}\ell)^4 \ge 3\beta \ell^4$ good copies of $C_4$ in $\cP'$, none of which contains a pair from $\CC$.

  Finally, it follows from our construction that each good copy of $C_4$ in $\cP'$ corresponds to an edge of $\HH_i^\cP$ for some $i \in \{0,1,2\}$ that additionally does not contain any saturated pairs of type~\ref{item:saturated-01} or~\ref{item:saturated-10}. Moreover, by our definition of $\CC$, none of the at least $3\beta\ell^4$ copies we have found above contains a saturated pair of type~\ref{item:saturated-02} either. Recalling that $e(\HH_0) + e(\HH_1) + e(\HH_2) < 3\beta\ell^4$, it follows that one of these good $C_4$s yields a pair $(A, B) \in \HH_i^\cP \setminus \HH_i$ such that $\HH_i \cup \{(A, B)\}$ is permissible. Iterating this process, we must eventually arrive at a permissible hypergraph $\HH_i$ (for some $i \in \{0,1,2\}$) with at least $\beta \ell^4$ edges, as required. 
\end{proof}

The remainder of the this section is dedicated to the proof of Theorem~\ref{thm:robust-stability-increment}. We begin by proving the following proposition, which proves Theorem~\ref{thm:robust-stability-increment} when the condition~\ref{item:supsersaturation-cond-incr} holds and will moreover serve as a helpful warm-up for the proof of the theorem. It will also be a step in the proof of the theorem under the assumption~\ref{item:stability-cond-incr}.


\begin{prop}
  \label{prop:supersaturation-step}
  Suppose that integers $\ell$ and $n$ satisfy $\ell \ge \sqrt{n}$ and that $\cP = (M,E)$ is a pregraph on $n$ vertices with $e(E) \le \binom{\ell}{2}$ and $e(M) \ge 3 \ell n$. Then for any collection $\CC$ of at most $\ell^3/40$ pairs of elements of $M$, there exist at least $\ell^4/40$ good copies of $C_4$ in $\cP$ that contain no pair from $\CC$.
\end{prop}

Our proofs will use the following two auxiliary statements. The first is a well-known result of Caro~\cite{Ca79} and Wei~\cite{We81}. We remark that, in this section, if $G$ is a graph (such as $M$ or $E$), we will write $d_G(v)$ and $d_G(v,S)$ to denote the number of neighbours of $v$ and the number of neighbours of $v$ in $S$, respectively.

\begin{lemma}
  \label{lem:CW}
  For every graph $G$, 
  \[
    \alpha(G) \ge \sum_{v \in V(G)} \frac{1}{1+ d_G(v)}.
  \]
\end{lemma}

The second is an easy consequence of Jensen's inequality applied to the convex function $[0,\infty) \ni x \mapsto 1/(1+x)$. Given a nonnegative integer $d$ and a real number $q \in [0,1]$, we shall denote by $\Bin(d, q)$ the binomial random variable with parameters $d$ and $q$.

\begin{fact}
  \label{fact:Jensen}
  For every $d \in \N$ and $q \in [0,1]$,
  \[
    \Ex\left[ \frac{1}{1 + \Bin(d,q)} \right] \ge \frac{1}{1 + qd}.
  \]
\end{fact}

\begin{proof}[Proof of Proposition~\ref{prop:supersaturation-step}]
  Fix a pregraph $\cP = (M, E)$ on $n$ vertices and a collection $\CC$ satisfying the assumptions of the proposition. We first remove all vertices whose degree in $M$ is less than $2\ell$. As this way we lose at most $2\ell n$ edges of $M$, we arrive at an $m$-vertex subset $W \subseteq V(K_n)$, for some $2\ell \le m \le n$, such that $\delta(M[W]) \ge 2\ell$. Clearly, it is sufficient to find $\ell^4/40$ good copies of $C_4$ in $\cP$ restricted to $W$, none of which contains a pair from $\CC$. Therefore, shall replace the original $M$, $E$, and $\cP$ with their restrictions to the set $W$.

  Set $q = m / \ell^2 \le n / \ell^2 \le 1$ and form a random subset $R \subseteq W$ by retaining each element of $W$ independently with probability $q$. We apply Lemma~\ref{lem:CW} to the graph $E[R]$ to find an independent set $I \subseteq R$ with 
  \[
    |I| \ge \sum_{v \in R} \frac{1}{1 + d_{E[R]}(v)} = \sum_{v \in W} \frac{\indicator[v \in R]}{1 + d_E(v, R)}.
  \]
  By Fact~\ref{fact:Jensen}, we have
  \[
    \Ex[|I|] \ge \Ex\left[ \sum_{v \in W} \frac{\indicator[v \in R]}{1 + d_E(v,R)} \right] = \sum_{v \in W} q \cdot \Ex\left[ \frac{1}{1 + \Bin(d_E(v),q)} \right] \ge \sum_{v \in W} \frac{q}{1 + q d_E(v)}.
  \]
  As the function $[0, \infty) \ni x \mapsto q/(1+qx)$ is convex, the sum in the right-hand side above is minimised when $d_E(v) = 2e(E)/m$ for every $v \in W$. As $e(E) \le \ell^2/2$, then
  \begin{equation}
    \label{eq:supersaturation-I-LB}
    \Ex[|I|] \ge \frac{qm}{1+2qe(E)/m} \ge \frac{qm}{1 + q \cdot \ell^2 / m} = \frac{m^2}{2\ell^2}.
  \end{equation}

  Next, let us choose, for each vertex $v \in W$, an arbitrary set $M_v$ of $2\ell$ edges of $M$ that are incident to $v$. We shall say that a copy of $K_{1,2}$ is \emph{good} if its centre $v$ lies in $I$, both of its edges are in $M_v$, and the pair comprising its two non-centre vertices does not belong to $E$. The number $X_g$ of such good $K_{1,2}$s satisfies
  \begin{equation}
    \label{eq:Xg-lower}
    X_g \ge \sum_{v \in I} \left( \binom{2\ell}{2} - e(E) \right) \ge \left(\binom{2\ell}{2} - \binom{\ell}{2}\right)\cdot |I| \ge \frac{4\ell^2}{3} \cdot |I|.
  \end{equation}
  We shall say that a copy of $K_{1,2}$ in $M$ is \emph{saturated} if (the set consisting of) its two edges belong to $\CC$. Let $X_s$ be the number of saturated $K_{1,2}$s in $M$ whose centre vertex belongs to the (random) set $I \subseteq R$. Writing $X$ for the number of good $K_{1,2}$s that are not saturated, we have $X \ge X_g - X_s$ and hence, recalling that $|\CC| \le \ell^3 / 40$, 
\begin{equation}\label{eq:step:X:lower}
\Ex[X] \ge \Ex[X_g] - \Ex[X_s] \ge \frac{4\ell^2}{3} \cdot \Ex[|I|] - q \cdot |\CC| \ge \frac{2m^2}{3} - \frac{\ell m}{40} \ge \frac{3m^2}{5},
\end{equation}
  where we have used~\eqref{eq:supersaturation-I-LB}, \eqref{eq:Xg-lower}, and the inequality $m \ge 2\ell$.
  
Since $I$ is an independent set in $E$, it follows that any pair of good $K_{1,2}$s with the same non-centre vertices form a good $C_4$ and therefore we have at least $X - \binom{m}{2}$ such $C_4$s. However, we must disregard those $C_4$s that contain a saturated $K_{1,2}$ whose two non-centre vertices lie in $I$, since the two edges of such a saturated $K_{1,2}$ could come from two different good non-saturated $K_{1,2}$s whose centre vertices lie in $I$. The expected number of saturated $K_{1,2}$s of this type is at most $q^2 \cdot |\CC|$ and each of them lies in at most $2\ell$ of our good $C_4$s, since the edges of our good $C_4$s came only from the sets $M_v$. We must therefore discard (in expectation) at most $2\ell q^2 |\CC|$ of the (at least) $X - \binom{m}{2}$ good $C_4$s found using pairs of good $K_{1,2}$s.
  
To summarise, let $Z$ be the number of good $C_4$s that contain no saturated $K_{1,2}$ and at least two vertices of $I$. By~\eqref{eq:step:X:lower} and the argument above, we have
  \[
    \Ex[Z] \ge \Ex[X] - \binom{m}{2} - 2\ell q^2 |\CC| \ge \left(\frac{3}{5} - \frac{1}{2} - \frac{2|\CC|}{\ell^3}\right)m^2 \ge \frac{m^2}{20}.
  \]
Finally, observe that each good copy of $C_4$ containing no saturated $K_{1,2}$ has probability at most $2q^2$ of being counted by $Z$. It therefore follows that the total number of such copies of $C_4$ must be at least $m^2/(40q^2) = \ell^4/40$, as required.
\end{proof}

We next consider pregraphs $\cP = (M,E)$ for which one can find a small set $A$ of vertices of $K_n$ that contains only a tiny proportion of the edges of $E$, but still a large proportion of mixed edges have an endpoint in $A$. The following proposition will be invoked in the proof of Theorem~\ref{thm:robust-stability-increment}.

\begin{prop}
  \label{prop:supersaturation-A}
  Suppose that integers $\ell$ and $n$ satisfy $\ell \ge 4\sqrt{n}$ and set $\alpha := 1/640$. Let $\cP = (M,E)$ be a pregraph on $n$ vertices with $e(E) \le \binom{\ell}{2}$ and suppose that there exists a set $A \subseteq V(K_n)$, with $|A| \le \alpha n$ and $e_E(A) \le \alpha\ell^2$, such that
  \[
    \sum_{w \in A} d_M(w) \ge \frac{\ell n}{2}.
  \]
  Then for any collection $\CC$ of at most $\alpha\ell^3$ pairs of elements of $M$, there exist at least $\alpha\ell^4$ good copies of $C_4$ in $\cP$ that contain no pair from $\CC$.
\end{prop}

\begin{proof}
  The proof of Proposition~\ref{prop:supersaturation-A} follows the general strategy of the proof of Proposition~\ref{prop:supersaturation-step}, but there are some key differences. In particular, we will find the independent set $I$ inside the set $A$ alone and we shall select vertices of $R$ with different probabilities. Rather than invoking Lemma~\ref{lem:CW} and Fact~\ref{fact:Jensen}, we shall give a somewhat finer argument to produce a large independent set $I \subseteq R$ and use it to construct good copies of $C_4$.

  We start by iteratively removing from $A$ all vertices $v$ that do not satisfy
  \begin{equation}
    \label{eq:degE-degM}
    d_M(v) \ge \max\left\{2\ell, \, \frac{n}{16\alpha \ell} \cdot d_E(v,A)\right\}.
  \end{equation}
  Observe that the set $A'$ of vertices remaining after this deletion satisfies
  \begin{equation}
    \label{eq:sum-deg-Ap}
    \begin{split}
      \sum_{v \in A'} d_M(v) & \ge \sum_{v \in A} d_M(v) - 2\cdot\left(|A| \cdot 2\ell + e_E(A) \cdot \frac{n}{16 \alpha \ell}\right) \\
      & \ge \frac{\ell n}{2} - 4\alpha\ell n - \alpha\ell^2 \cdot \frac{n}{8 \alpha \ell} \ge \frac{\ell n}{3}.
  \end{split}
\end{equation}
  Let $a = |A'|$ and order the elements of $A'$ as $v_1, \dotsc, v_a$ so that $d_M(v_i) \le d_M(v_j)$ whenever $1\le i \le j \le a$. For each $i \in [a]$, let
  \[
    q_i = \frac{8n}{\ell \cdot d_M(v_i)} \le \frac{4n}{\ell^2} \le \frac{1}{4}
  \]
  and form a random set $R \subseteq A'$ by keeping each $v_i$ independently with probability $q_i$. Define
  \[
    I = \big\{ v_i \in R \scolon \text{$v_j \not\in R$ for every $j > i$ such that $v_iv_j \in E$} \big\}
  \]
  and observe that $I$ is an independent set in the graph $E$.\footnote{The idea of forming a large independent set this way is taken from the proof of Lemma~\ref{lem:CW} given in~\cite{AlSp16}.}

Similarly to before, we shall say that a copy of $K_{1,2}$ in $M$ is good if its centre lies in $I$ and the pair comprising its two non-centre vertices does not belong to $E$. Observe that the number $X_g$ of good $K_{1,2}$s satisfies
  \[
    X_g \ge \sum_{v \in I}  \left( \binom{d_M(v)}{2} - e(E) \right) \ge  \sum_{v \in I}  \left( \binom{d_M(v)}{2} - \binom{\ell}{2} \right) \ge  \sum_{v \in I}  \frac{d_M(v)^2}{3},
  \]
  as $d_M(v) \ge 2\ell$ for each $v \in I$. We shall now estimate the probability that a given vertex $v \in A'$ belongs to the random set $I$. To this end, suppose that $v = v_i$ for some $i \in [a]$ and note that, by~\eqref{eq:degE-degM}, there are at most $d_M(v_i) \cdot 16\alpha \ell/n$ indices $j$ such that $v_iv_j \in E$. Moreover, by our choice of the ordering, $q_j \le q_i$ whenever $j > i$. Letting $d = d_M(v) = d_M(v_i)$, and recalling that $8n/(\ell d) \le 1/4$, it follows that
  \[
    \Pr\big( v \in I \big) \ge q_i \cdot (1-q_i)^{16 \alpha \ell d/n} = \frac{8n}{\ell d} \cdot \left(1 - \frac{8n}{\ell d}\right)^{16 \alpha \ell d / n} \ge  e^{-160\alpha} \cdot \frac{8n}{\ell d} \ge\frac{6n}{\ell d},
  \]
 where we used the bounds $1-x \ge e^{-5x/4}$ when $0 \le x \le 1/4$ and $e^{-1/4} > 3/4$.

  We will need to disregard the saturated $K_{1,2}$s, that is, all those whose pair of edges belongs to $\CC$. Let $X_s$ be the number of those saturated $K_{1,2}$s whose centre vertex belongs to the set $I$. Writing $X$ for the number of good $K_{1,2}$s that are not saturated, we have $X \ge X_g - X_s$, and hence
  \[
    \begin{split}
      \Ex[X] & \ge \Ex[X_g] - \Ex[X_s] \ge \sum_{v \in A'} \Pr(v \in I) \cdot \frac{d_M(v)^2}{3} - \max_i q_i \cdot |\CC| \\
      & \ge \sum_{v \in A'} \frac{2n d_M(v)}{\ell} - \frac{4n}{\ell^2} \cdot |\CC| \ge\frac{2n^2}{3} - \frac{2|\CC|n^2}{\ell^3} \ge \frac{3n^2}{5},
    \end{split}
  \]
  where we have used~\eqref{eq:sum-deg-Ap} and the inequality $n \ge 2\ell$ (which holds since $A'$ is non-empty).
  
  Since $I$ is an independent set in $E$, it follows that any pair of good $K_{1,2}$s with the same non-centre vertices forms a good $C_4$. Thus we have at least $X - \binom{n}{2}$ such $C_4$s. However, we must still disregard those $C_4$s that contain a saturated $K_{1,2}$ with two non-centre vertices in $I$. Fix some $K_{1,2}$ from $\CC$ and suppose that its non-centre vertices are $v_i$ and $v_j$. Observe that it can lie in at most $d_M(v_i)$ of our good copies of $C_4$. Therefore, the expected number of good $C_4$s that we are forced to disregard because of this single $K_{1,2}$ is at most
  \[
    q_i \cdot q_j \cdot d_M(v_i) = \frac{64 n^2}{\ell^2 d_M(v_j)} \le \frac{32 n^2}{\ell^3}.
  \]
  Consequently, the expected number of good copies of $C_4$ that we have to disregard because of one of the saturated $K_{1,2}$s from $\CC$ is at most $32 |\CC| n^2 / \ell^3$.

  To summarise, let $Z$ be the number of good $C_4$s that contain no saturated $K_{1,2}$ and at least two vertices of $I$. We have shown that
  \[
    \Ex[Z] \ge \Ex[X] - \binom{n}{2} - \frac{32|\CC| \cdot n^2}{\ell^3} \ge \left(\frac{3}{5} - \frac{1}{2} - \frac{32|\CC|}{\ell^3}\right)n^2 \ge \frac{n^2}{20}.
  \]
  But as each good copy of $C_4$ containing no saturated $K_{1,2}$ has chance at most $2q_1^2$ to be counted by $Z$, the number of them is at least $n^2/(40q_1^2) \ge \ell^4/640$. This completes the proof of the proposition.
\end{proof}

\begin{proof}[Proof of Theorem~\ref{thm:robust-stability-increment}]
We begin by defining the constants whose existence is claimed in the statement of the theorem. Given $0 < \eps \le 1/2$, set $\alpha = 2^{-16}$ and define
  \[
 C = \frac{4}{\alpha^3}, \qquad \delta = \min\left\{ \frac{\alpha^3}{2^7}, \frac{\eps}{16} \right\}, \qquad \beta = \frac{\delta^4}{2^{100}},  \qquad \text{and} \qquad \lambda = \frac{\delta^7}{2^{10}}.
  \]
Suppose that $\ell \ge C\sqrt{n}$ and let $\cP = (M,E)$ be a pregraph on $n$ vertices with $e(E) \le \binom{\ell}{2}$. If $\cP$ satisfies~\ref{item:supsersaturation-cond-incr}, then we may immediately invoke Proposition~\ref{prop:supersaturation-step}, noting that $|\CC| \le 12\beta\ell^3 \le \ell^3/40$, to find find $\ell^4/40$ good copies of $C_4$ that contain no pair from $\CC$. 
 
We may therefore assume from now on that $\cP$ satisfies~\ref{item:stability-cond-incr}, that is, 
\[
  e(M) \ge (1-\delta)\ell n, \qquad E \textup{ is not $\eps$-close to $K_\ell$}, \qquad \textup{and} \qquad \ell \le \lambda n.
\]
We begin by iteratively removing all vertices $v$ whose degree in $M$ is smaller than $(1-2\delta)\ell$. As this way we can remove at most $(1-2\delta)\ell n$ edges of $M$, we will eventually arrive at a set $W \subseteq V(K_n)$ with $\delta(M[W]) \ge (1-2\delta)\ell$. Set $m = |W|$, and note that, since we removed at most $(1-2\delta)\ell (n-m)$ edges of $M$, we have $e_M(W) \ge \max\big\{ (1 - \delta) \ell m, \delta \ell n \big\}$, and therefore
 \begin{equation}  \label{eq:m-LB}
    m  \ge  \sqrt{\delta \ell n}  \ge  \sqrt{\frac{\delta}{\lambda}} \cdot \ell  \ge  \frac{32}{\delta^3} \cdot \ell.
  \end{equation}
  Observe that the subgraph of $E$ induced by $W$ is also not $(\eps/2)$-close to $K_\ell$. Indeed, otherwise there would be an $\ell$-element set $U \subseteq W$ with $e_E(U) \ge (1-\eps/2) \binom{\ell}{2}$, which would imply that $E$ itself is $\eps$-close to $K_\ell$, as $e(E) \le \binom{\ell}{2}$. We may thus work with the restrictions of $M$, $E$, and $\cP$ to the set $W$. We shall surpress $W$ from the notation and write $M$, $E$, and $\cP$ in place of $M[W]$, $E[W]$ and $(M[W], E[W])$. In particular,
  \[
    e(M) \ge (1-\delta)\ell m, \qquad e(E) \le \binom{\ell}{2}, \qquad \text{$E$ is not $(\eps/2)$-close to $K_\ell$,}
  \]
 and moreover $\delta(M) \ge (1-2\delta)\ell$.
   
  We split the proof into two cases, depending on the shape of the degree sequence of $E$.

  \bigskip
  \noindent
  \textbf{Case 1.} There is a set $L \subseteq W$ of $\alpha m$ vertices $v$ satisfying $d_E(v) \le (1 - \alpha) \ell^2 / m$.\medskip

  Set $q = C m / \ell^2 \le 1/C$ and form a random subset $R \subseteq W$ by keeping each element of $W$ independently with probability $q$. We apply Lemma~\ref{lem:CW} to the graph $E[R]$ (cf.\ the proof of Proposition~\ref{prop:supersaturation-step}) to find an independent set $I \subseteq R$ with
  \[
    |I| \ge \sum_{v \in R} \frac{1}{1 + d_{E[R]}(v)} = \sum_{v \in W} \frac{\indicator[v \in R]}{1 + d_E(v, R)}.
  \]
  By Fact~\ref{fact:Jensen}, we have
  \[
    \Ex[|I|] \ge \Ex\left[ \sum_{v \in W} \frac{\indicator[v \in R]}{1 + d_E(v,R)} \right] = \sum_{v \in W} q \cdot \Ex\left[ \frac{1}{1 + \Bin(d_E(v),q)} \right] \ge \sum_{v \in W} \frac{q}{1 + q d_E(v)}.
  \]
  As the function $[0, \infty) \ni x \mapsto q/(1+qx)$ is convex, the sum in the right-hand side above is minimised when $d_E(v) = 2e(E)/m$ for every $v \in W$.  However, we assumed that $d_E(v) \le (1-\alpha)\ell^2/m$ for every $v \in L$, so a slightly stronger bound holds. Indeed, since
  \[
    2e(E) \le \ell^2 = \alpha m \cdot (1 - \alpha) \frac{\ell^2}{m} + (1 - \alpha)m  \cdot \left(\frac{1}{1- \alpha} - \alpha\right) \frac{\ell^2}{m},
  \]
  then it follows that
  \begin{equation}
    \label{eq:I-LB-uneven-degrees}
    \Ex[|I|] \ge \frac{\alpha m \cdot q}{1 + q \cdot (1 - \alpha) \ell^2 / m} + \frac{(1 - \alpha) m \cdot q}{1 + q \cdot (1/(1 - \alpha) - \alpha) \ell^2 / m} \ge \left(1+\frac{\alpha^3}{2}\right) \frac{m^2}{\ell^2}.
  \end{equation}
  One may verify the the last inequality in~\eqref{eq:I-LB-uneven-degrees} by multiplying the numerators and the denominators in the left-hand side by $m/(\ell^2q) = 1/C = \alpha^3/4$ and observing that
  \[
    \frac{\alpha}{1 - \alpha} + \frac{1 - \alpha}{1/(1 - \alpha) - \alpha}  = 1 + \frac{\alpha^3}{(1-\alpha)(1-\alpha+\alpha^2)} \ge 1 + \alpha^3.
  \]

  Set $d = \big\lfloor (1-2\delta)\ell \big\rfloor$ and choose, for each vertex $v \in W$, an arbitrary set $M_v$ of $d$ edges of $M$ that are incident to $v$. As before, we shall say that a copy of $K_{1,2}$ is good if its centre $v$ lies in $I$, both of its edges are in $M_v$, and the pair of its non-centre vertices is not in $E$. As $E$ is not $(\eps/2)$-close to $K_\ell$, then for every $v \in W$, the set $\widehat{M}_v$ of the $d$ other endpoints of the edges in $M_v$ contains at least $\binom{d}{2} - (1-\eps/2) \binom{\ell}{2}$ pairs that do not belong to $E$. In particular, as $\delta \le \eps/16$, each vertex of $I$ is the centre of at least $\eps \ell^2/8$ good $K_{1,2}$s. Unfortunately, this lower bound is not sufficiently strong for the naive argument given in the proof of Proposition~\ref{prop:supersaturation-step} to work, as $\Ex[|I|]$ is too small. Instead, we shall exploit the rough structure of $E$.
  
  To this end, we partition the set $W$ into sets $W_L$ and $W_H$ of low and high degree vertices, which are defined as follows:
  \[
    W_H := \big\{ v \in W \scolon d_E(v) \ge \delta \ell / 2 \big\} \qquad \text{and} \qquad W_L := W \setminus W_H.
  \]
  Given an independent set $I$, we split it into $I_L$ and $I_H$, which are defined as follows:
  \[
    I_H := \big\{ v \in I \scolon |\widehat{M}_v \cap W_H| \ge \delta \ell \big\} \qquad \text{and} \qquad I_L := I \setminus I_H.
  \]
  Observe that if $v \in I_L$, then $\widehat{M}_v$ contains at least $\binom{d-\delta\ell}{2} - \delta\ell^2/2 \ge (1-7\delta)\ell^2/2$ pairs that do not belong to $E$. We shall argue differently for different $I$, depending on the relative sizes of the sets $I_L$ and $I_H$.

  In both cases, we will find a (random) collection of at least $\delta^2m^2/16$ good $C_4$s (in expectation) each of which is the union of two $K_{1,2}$s centred at some $v, w \in I$ and such that neither of (the pairs of edges of) these $K_{1,2}$s belongs to $\CC$. We first argue that this is sufficient. Indeed, even though we will still have to disregard those copies of $C_4$ that contain a $K_{1,2}$ with two non-centre vertices in $I$ whose edges belong to $\CC$, the expected number of such saturated $K_{1,2}$s is at most $q^2 \cdot |\CC|$ and each of them lies in at most $d \le \ell$ of our good copies of $C_4$, as the edges of these good $C_4$s came only from the sets $M_v$. Hence, letting $Z$ be the (random) number of good $C_4$s that contain at least two vertices of $I$ and no $K_{1,2}$ whose edges belong to $\CC$, we will have
  \[
    \Ex[Z] \ge \frac{\delta^2m^2}{16} - q^2 \cdot |\CC| \cdot \ell \ge \frac{\delta^2m^2}{16} - 12C^2\beta m^2 \ge \frac{\delta^2m^2}{32}.
  \]
  But as each good copy of $C_4$ containing no saturated $K_{1,2}$ has chance at most $2q^2$ to be counted by $Z$, the number of them is at least
  \[
    \frac{\delta^2m^2}{64q^2} = \frac{\delta^2\ell^4}{64C^2} \ge 3\beta \ell^4.
  \]
  Therefore, in order to complete the proof of the theorem in Case 1, it suffices to prove the existence of (a random collection of) $\delta^2m^2/16$ good copies of $C_4$ (in expectation) of the less restrictive type described above.

  \medskip
  \noindent
  \textbf{Subcase 1A.} $|I_H| \le \delta |I|$.
  \smallskip

  Recall that if $v \in I_L$, then $\widehat{M}_v$ contains at least $(1-7\delta)\ell^2/2$ pairs that do not belong to $E$. It follows that the number $X_g$ of good $K_{1,2}$s satisfies
  \[
    X_g \ge (1 - 7\delta)\frac{\ell^2}{2} \cdot |I_L| \ge (1-8\delta)\frac{\ell^2}{2} \cdot |I|.
  \]
  Writing again $X_s$ for the number of saturated $K_{1,2}s$ (those whose edges belong to $\CC$) whose centre vertex belongs to $I$ and $X$ for the number of good $K_{1,2}$s that are not saturated, we have $X \ge X_g - X_s$ and consequently,
  \[
    \begin{split}
      \Ex[X] & \ge \Ex[X_g] - \Ex[X_s] \ge  (1 - 8\delta)\frac{\ell^2}{2} \cdot \Ex[|I|] - q \cdot |\CC| \\
      & \ge (1-8\delta)\left(1+\frac{\alpha^3}{2}\right) \frac{m^2}{2} - 12C\beta \ell m \ge \left(1+\frac{\alpha^3}{4}\right)\frac{m^2}{2},
    \end{split}
  \]
  where we used~\eqref{eq:I-LB-uneven-degrees}, the facts that $\delta < \alpha^3/2^7$ and $\beta < \alpha^3/(8\cdot24C)$, and the trivial inequality $m \ge (1-2\delta)\ell \ge \ell/2$. Since $I$ is an independent set in $E$, any pair of good $K_{1,2}$s with the same non-centre vertices forms a good $C_4$. Thus we have at least $X - \binom{m}{2}$ such $C_4$ and
  \[
    \Ex[X] - \binom{m}{2} \ge \frac{\alpha^3m^2}{8} \ge \frac{\delta^2m^2}{16},
  \]
  as required.

  \medskip
  \noindent
  \textbf{Subcase 1B.} $|I_H| > \delta |I|$.
  \medskip

  Let us write $X_g$ for the number of good $K_{1,2}$s with at least one non-centre vertex in $W_H$. We will show in this case that 
  \[
    X_g \ge \frac{\delta^2 \ell^2}{4} \cdot |I| \qquad \textup{and} \qquad |W_H| \le \frac{\delta^2m}{16},
  \]
  from which it will be straightforward (as in Subcase~1A) to deduce the existence of the required collection of good $C_4$s.

  To prove the lower bound on $X_g$, recall first that each vertex $v \in I_H$ is the centre of at least $\eps \ell^2 / 8$ good $K_{1,2}$s; we claim that at least $\delta \ell^2 / 4$ of these have at least one non-centre vertex in $W_H$. To prove this, set $w = |\widehat{M}_v \cap W_L|$ and suppose first that $w \le \eps \ell/2$. Then at most $\eps^2 \ell^2 / 8$ good $K_{1,2}$s centred at $v$ have both non-centre vertices in $W_L$ and since $\eps/8 - \eps^2/8 \ge \eps/16 \ge \delta/4$, the claim follows in this case. On the other hand, if $w > \eps \ell/2$, then there are at least 
  \[
    \min\bigg\{ w\bigg( d - w  - \frac{\delta \ell}{2} \bigg) \scolon \frac{\eps \ell}{2} < w \le d - \delta \ell \bigg\} \ge \frac{\delta\ell^2}{4}
  \]
  good $K_{1,2}$s centred at $v$ with at least one non-centre vertex in $W_H$. Indeed, since $|\widehat{M}_v| = d$ and each $u \in \widehat{M}_v \cap W_L$ has degree at most $\delta \ell / 2$ in $E$, there are at least $d - w - \delta \ell / 2$ good $K_{1,2}$s centred at $v$ that contain $u$ and a third vertex from $W_H$. Thus
  \[
    X_g \ge \frac{\delta \ell^2}{4} \cdot |I_H| \ge \frac{\delta^2 \ell^2}{4} \cdot |I|,
  \]
  as claimed. To prove the claimed upper bound on $|W_H|$, observe that
  \[
    \ell^2 \ge 2e(E) \ge \sum_{v \in W_H} d_E(v) \ge |W_H| \cdot \frac{\delta \ell}{2}
  \]
  which implies, by~\eqref{eq:m-LB}, that
  \[
    |W_H| \le \frac{2\ell}{\delta} \le \frac{\delta^2m}{16},
  \]
  as required. Now, writing $X$ for the number of good $K_{1,2}$s with a non-centre vertex in $W_H$ that are moreover not saturated and $X_s$ for the number of saturated $K_{1,2}$s (that is, $K_{1,2}$s whose pair of edges belongs to $\CC$) whose centre vertex belongs to $I$, we have $X \ge X_g - X_s$ and hence,
  \[
    \Ex[X] \ge \Ex[X_g] - \Ex[X_s]  \ge  \frac{\delta^2 \ell^2}{4} \cdot \Ex[|I|] - q \cdot |\CC|  \ge   \frac{\delta^2 m^2}{4} - 12C\beta \ell m  \ge  \frac{\delta^2m^2}{8}, 
  \]
  where we again used the bounds $\beta < \delta^2 / (8 \cdot 24 C)$ and $m \ge (1-2\delta)\ell \ge \ell/2$. 
  
  Finally, since $I$ is an independent set in $E$, it follows that there are at least $X - |W_H|m$ good $C_4$s formed by pairs of $K_{1,2}$s that are counted by $X$ and hence the expected number of good copies of $C_4$ that are formed by two $K_{1,2}$s centred at vertices in $I$, neither of which belongs to $\CC$, is at least 
  \[
    \Ex[X] - |W_H|m \ge \frac{\delta^2m^2}{16},
  \]
  as required. This completes the proof in Case 1.

  \medskip
  \noindent
  \textbf{Case 2.} There are fewer than $\alpha m$ vertices $v$ satisfying $d_E(v) \le (1 - \alpha) \ell^2 / m$.
  
  \smallskip

  In this case, we shall find our good $C_4$s in various ways, depending on the distribution of degrees (in both the graphs $M$ and $E$) on the set $A$ of vertices whose degree in $E$ is somewhat larger than average. To be precise, set $\gamma = 1/32$ and define
  \begin{equation}\label{def:AB}
    A := \big\{ v \in W \scolon d_E(v) \ge (1+\gamma)\ell^2/m \big\} \qquad \text{and} \qquad B := W \setminus A.
  \end{equation}
  We claim that $e_E(A) \le \gamma^2 \ell^2$. To prove this, observe first that
  \[
    \begin{split}
      2e(E) & = \sum_{v \in W} d_E(v) \ge |A| \cdot (1 + \gamma) \frac{\ell^2}{m} + \big( (1 - \alpha)m - |A| \big) \big( 1 - \alpha \big) \frac{\ell^2}{m} \\
      & = \bigg( (1-\alpha)^2 + (\gamma+\alpha) \cdot \frac{|A|}{m} \bigg) \ell^2.
    \end{split}
  \]
  Noting that $\alpha = \gamma^3/2$, and recalling that $e(E) \le \binom{\ell}{2}$, it follows that
  \[
    \frac{|A|}{m} \le \frac{\alpha(2-\alpha)}{\gamma+\alpha} \le \gamma^2,
  \]
  and therefore
  \[
    \begin{split}
      \sum_{v \in A} d_E(v) &  =  2e(E) - \sum_{v \in B} d_E(v)  \le  \ell^2 -  \big( (1-\alpha)m - |A| \big) \big( 1 - \alpha \big) \frac{\ell^2}{m} \\
      &  \le  \big( 2\alpha + \gamma^2 \big) \ell^2  \le  2\gamma^2\ell^2,
    \end{split}
  \]
  so in particular $e_E(A) \le \gamma^2 \ell^2$, as claimed. 
  
  We next use Propositions~\ref{prop:supersaturation-step} and~\ref{prop:supersaturation-A} to show that we may assume that $e_M(A) <  9 \gamma^3 \ell m$ and $e_M(A, B) < \ell m / 2$. Indeed, if $e_M(A) \ge 9 \gamma^3 \ell m$, then let us fix an arbitrary superset $A'$ of $A$ with exactly $\gamma^2 m$ elements and apply Proposition~\ref{prop:supersaturation-step} to the pregraph $\cP$ restricted to the set $A'$, with $\ell_{\ref{prop:supersaturation-step}} \leftarrow 3\gamma\ell$ and $n_{\ref{prop:supersaturation-step}} \leftarrow \gamma^2m$. To see that the conditions of the proposition are satisfied, note that 
  \[
    e_E(A') \le \gamma^2\ell^2 + \gamma^2m \cdot (1 + \gamma) \frac{\ell^2}{m} \le \binom{3\gamma \ell}{2} \quad \text{and} \quad e_M(A') \ge 9 \gamma^3 \ell m = 3(3\gamma \ell) (\gamma^2 m)
  \]
  and that $(3\gamma\ell)^2 \ge 9\gamma^2n\ge\gamma^2m$ and $|\CC| \le 12\beta\ell^3 \le (3\gamma\ell)^3/40$. The proposition provides $(3\gamma\ell)^4/40$ good copies of $C_4$ that contain no pair from $\CC$, and so in this case we are done. Similarly, if $e_M(A,B) \ge \ell m / 2$, then, noting that
  \[
    e_E(A) \le \gamma^2 \ell^2 < \frac{\ell^2}{640}, \qquad |A| \le \gamma^2m < \frac{m}{640} \qquad \textup{and} \qquad |\CC| \le 12\beta\ell^3 < \frac{\ell^3}{640},
  \]
  we may invoke Proposition~\ref{prop:supersaturation-A} to find $\ell^4/640$ good copies of $C_4$, none of which contains a pair from $\CC$, and so in this case we are also done. We may therefore assume from now on that $e_M(A) <  9 \gamma^3 \ell m$ and $e_M(A, B) < \ell m / 2$, and hence that
  \begin{equation}
    \label{eq:sum-degM-B}
    \sum_{v \in B} d_M(v) = 2e(M) - 2e_M(A) - e_M(A, B) \ge \left(2(1-\delta) - 18\gamma^3 - \frac{1}{2} \right)\ell m \ge \frac{4\ell m}{3}.
  \end{equation}

  For the rest of the proof, we will search for good $C_4$s formed by two $K_{1,2}$s whose centre vertices belong to $B$. Let us say that a copy of $K_{1,2}$ in $M$ is good if its centre lies in $B$ and the pair of its non-centre vertices does not belong to $E$. Observe that for each $v \in B$, letting $N_M(v)$ denote the $M$-neighbourhood of $v$, we have
  \begin{equation}    
    \begin{split}
      e_E\big( N_M(v) \big) &  \le  e_E(A) + \sum_{w \in N_M(v) \cap B} d_E(w) \\
      &  \le  \gamma^2\ell^2 + d_M(v) \cdot (1+\gamma) \frac{\ell^2}{m}  \le  \frac{\gamma}{2} \cdot d_M(v)^2,\label{eq:eE-NMv-UB}
    \end{split}
  \end{equation}
  since $d_M(v) \ge \delta(M) \ge \ell/2$ and $\ell/m \le \gamma/16$ by~\eqref{eq:m-LB}. We therefore have at least $(1/2 - \gamma) d_M(v)^2$ good $K_{1,2}$s centred at $v$, for each $v \in B$. 
  
  It only remains to bound the number of good $C_4$s composed of two good $K_{1,2}$s, and remove those that contain a pair from $\CC$. Our strategy will be similar to that used above, but there are two additional problems to overcome in this case: the set $B$ is not an independent set and we do not have an upper bound on the degrees $d_M(v)$. To deal with the first problem, we will use our upper bound on $d_E(v)$ for $v \in B$, together with a slightly more careful application of convexity than was needed earlier in the proof. To deal with the second issue, we will partition $B$ according to the approximate size of $d_M(v)$ and restrict our search to one of the parts.
  
  We first partition $B$ into two parts, depending (roughly speaking) on whether or not $d_M(v) = O(\ell)$. Define
  \[
    B_L := \left\{ v \in B \scolon d_M(v) \le 2^{20} \ell \right\} \qquad \text{and} \qquad B_H := B \setminus B_L.
  \]
  We first consider the case in which sufficiently many of the mixed edges incident to $B$ have an endpoint in $B_L$. 

  \medskip
  \noindent
  \textbf{Subcase 2A.}  
   \begin{equation} \label{eq:subcase-BL}
    \sum_{v \in B_L} d_M(v) \ge \frac{5\ell m}{4}.
  \end{equation}
  \smallskip

Let $X$ denote the number of good $K_{1,2}$s whose centre vertex lies in $B_L$ and whose pair of edges does not belong to the family $\CC$. By~\eqref{eq:eE-NMv-UB}, we have
  \[
    \begin{split}
      X &  \ge  \sum_{v \in B_L} \left(\binom{d_M(v)}{2} - \frac{\gamma}{2} d_M(v)^2\right) - |\CC| \\
      &  \ge  \left( \frac{1}{2} - \gamma \right)\sum_{v\in B_L} d_M(v)^2 - 12 \beta \ell^3  \ge  \frac{1}{3} \sum_{v \in B_L} d_M(v)^2,
    \end{split}
  \]
  since by the Cauchy--Schwarz inequality and~\eqref{eq:subcase-BL},
  \begin{equation}
    \label{eq:sum-dM2-BL}
    \sum_{v \in B_L} d_M(v)^2  \ge  \frac{1}{|B_L|} \bigg(\sum_{v \in B_L} d_M(v) \bigg)^2  \ge  \frac{25}{16} \cdot \ell^2m.
  \end{equation}
Let $Y$ denote the number of (ordered) pairs of $K_{1,2}$s that are counted by $X$ and have the same non-centre vertices. By the convexity of the function $x \mapsto x(x-1)$ and by~\eqref{eq:sum-dM2-BL}, we have
  \[
    Y \ge X \left(\frac{X}{\binom{m}{2}}-1\right) \ge \frac{1}{3} \sum_{v \in B_L} d_M(v)^2 \cdot \left(\frac{25\ell^2}{8m}-1\right) \ge \frac{\ell^2}{m} \cdot \sum_{v \in B_L} d_M(v)^2,
  \]
  since $\ell^2 \ge Cn \ge Cm$. Now, let us denote by $Y_b$ the number of (ordered) pairs of $K_{1,2}$s counted by $Y$ that do not correspond to good $C_4$s (that is, pairs of good $K_{1,2}$s with the same non-centre vertices, whose centre vertices are adjacent in $E$). By the definition~\eqref{def:AB} of $B$, this number satisfies
    \[
    Y_b  \le  \sum_{v \in B_L} d_E(v) \binom{d_M(v)}{2}  \le  \bigg( \frac{1 + \gamma}{2} \bigg) \frac{\ell^2}{m} \cdot \sum_{v \in B_L} d_M(v)^2.
  \]
  Thus, writing $Z_g$ for the number of good $C_4$s consisting of pairs of $K_{1,2}$s counted by $Y$ and combining the last three displayed equations, we obtain
  \[
    Z_g  \ge  \frac{Y - Y_b}{4}  \ge  \frac{1}{12} \cdot \frac{\ell^2}{m} \cdot \sum_{v \in B_L} d_M(v)^2  \ge  \frac{\ell^4}{8}.
  \]
  Finally, we must disregard those good $C_4$s, counted in $Z_g$, that contain a $K_{1,2}$ of mixed edges that belongs to the family $\CC$. The edges of such a $K_{1,2}$ must come from different good $K_{1,2}$s counted by $X$ and therefore (by the definition of $B_L$) there are at most $2^{20} \ell \cdot |\CC|$ such $C_4$s. It follows that the number $Z$ of good $C_4$s that contain no $K_{1,2}$s whose edges belong to $\CC$ satisfies
  \[
    Z \ge Z_g - 2^{20} \ell \cdot |\CC| \ge \left( \frac{1}{8} - 2^{24} \beta \right) \ell^4 \ge 3\beta\ell^4,
  \]
  as required.\medskip

Note that if~\eqref{eq:subcase-BL} fails to hold, then $\sum_{v \in B_H} d_M(v) \ge \big( 4/3 - 5/4 \big) \ell m = \ell m / 12$, by~\eqref{eq:sum-degM-B}. In this case we will choose a subset of $B_H$ on which the $M$-degrees are roughly constant and apply the same argument as in Subcase~2A. 
   
  \medskip
  \noindent
  \textbf{Subcase 2B.} 
  \begin{equation}
    \label{eq:subcase-BH}
    \sum_{v \in B_H} d_M(v) \ge \frac{\ell m}{12}.
  \end{equation} 
  \smallskip

For each integer $t \ge 0$, set 
  \[
    b_t = 2^{- 4t - 28} m \qquad \text{and} \qquad d_t = 2^{3t+20} \ell
  \]
  and define
  \[
    B_t = \big\{ v \in B_H \scolon d_t < d_M(v) \le d_{t+1} \big\}.
  \]
  We claim that there exists $t$ such that $|B_t| \ge b_t$. Indeed, since $B_H = \bigcup_{t \ge 0} B_t$, if there were no such $t$, then we would have
  \[
    \sum_{v \in B_H} d_M(v)  <  \sum_{t \ge 0} b_t d_{t+1}  = \sum_{t \ge 0} \frac{\ell m}{2^{t + 5}}  < \frac{\ell m}{12},
  \]
  contradicting~\eqref{eq:subcase-BH}. Fix any such $t$ and let $X$ denote the number of $K_{1,2}$s whose centre vertex lies in $B_t$, whose pair of non-centre vertices is not in $E$, and whose pair of edges does not belong to the family $\CC$. Observe that
  \[
    \begin{split}
      X &  \ge  \sum_{v \in B_t} \left(\binom{d_M(v)}{2} - \frac{\gamma}{2}d_M(v)^2\right) - |\CC|  \ge  b_t \cdot \left(\frac{1}{2} - \gamma\right) d_t^2 - |\CC| \\
      &  \ge  \left(\frac{1}{2}-\gamma\right)2^{2t+12} \ell^2 m - 12\beta\ell^3  \ge  2^{2t+10} \ell^2 m.
    \end{split}
  \]
  As before, let $Y$ denote the number of (ordered) pairs of $K_{1,2}$s that are counted by $X$ and have the same non-centre vertices. By the convexity of the function $x \mapsto x(x-1)$, we have
  \[
    Y \ge X\left(\frac{X}{\binom{m}{2}}-1\right) \ge 2^{4t+20}\ell^4,
  \]
  where we again used the assumption that $\ell^2 \ge Cm$. The number $Y_b$ of (ordered) pairs counted by $Y$ that do not correspond to good $C_4$s (that is, pairs of good $K_{1,2}$s whose centre vertices are adjacent in $E$) satisfies
  \[
    Y_b  \le  \sum_{v \in B_t} d_E(v) \binom{d_M(v)}{2}  \le  b_t \cdot \bigg( \frac{1 + \gamma}{2} \bigg) \frac{\ell^2}{m} \cdot d_{t+1}^2  \le  2^{2t+18} \ell^4.
  \]
  Thus, the number $Z_g$ of good $C_4$s counted by $Y$ satisfies
  \[
    Z_g \ge \frac{Y-Y_b}{4} \ge \big( 2^{4t+18} - 2^{2t+16} \big) \ell^4 \ge 2^{4t+17} \ell^4.
  \]
  Finally, we disregard those good $C_4$s, counted in $Z_g$, that contain a $K_{1,2}$ of mixed edges that belongs to the family $\CC$. For each element of $\CC$, there are at most $d_{t+1}$ such $C_4$s and therefore the number $Z$ of good $C_4$s that contain no $K_{1,2}$s whose edges belong to $\CC$ satisfies
  \[
    Z  \ge  Z_g - d_{t+1} \cdot |\CC|  \ge  \big( 2^{4t+17} - 12\beta \cdot 2^{3t+23} \big) \ell^4  \ge  \ell^4,
  \]
  as required. This completes the proof of the theorem.
\end{proof}

\section{The number of split graphs and the non-structured regime}
\label{sec:non-structured-regime}

In this section, we prove assertions~\ref{item:C4-free-quasirandom} and~\ref{item:C4-free-non-split} of Theorem~\ref{thm:asymC4free}. We first establish two lower bounds on the cardinality of $\FnmCf$: a stronger bound for all $m \ll n^{4/3}$ and a~weaker bound for all $m \ll n^{4/3} (\log n)^{1/3}$. Second, we carefully estimate the number of split graphs with $n$ vertices and $m$ edges for all $n$ and $m$ with $n \ll m \ll n^2$. Third, we provide a simple upper bound on the number of graphs that are not $\eps$-quasirandom. A~straightforward comparison of these bounds yields the claimed results.

\subsection{Lower bounds for $\FnmCf$}
\label{sec:lower-bounds-C4-free}

We first show that if $m \ll n^{4/3}$, then the family $\FnmCf$ forms an $e^{-o(m)}$-proportion of all graphs with $n$ vertices and $m$ edges. In particular, as we shall later verify, if $m \gg n$, then for every fixed $\eps$, graphs with no induced copy of $C_4$ outnumber the graphs that are not $\eps$-quasirandom and thus a typical member of $\FnmCf$ is $\eps$-quasirandom.

\begin{prop}
  \label{prop:lower-bound-strong}
  For every $\gamma > 0$, there exists $\delta > 0$ such that for all sufficiently large $n$ and all $m \le \delta n^{4/3}$,
  \[
    \left| \FnmCf \right| \ge e^{-\gamma m} \cdot \binom{\binom{n}{2}}{m}.
  \]
\end{prop}

\begin{proof}
  Fix a positive $\gamma$ and choose $\delta > 0$ sufficiently small so that $17(1+\delta)^4\delta^3 < \delta/2$ and $(2e/\delta)^\delta < e^\gamma$. Suppose that $m \le \delta n^{4/3}$, let $m' = \big\lfloor (1+\delta)m \big\rfloor$, and let $G$ be the uniformly chosen random graph with vertex set $\{1, \dotsc, n\}$ and precisely $m'$ edges. Let $X$ be the number of (not necessarily induced) copies of $C_4$ in $G$. As
  \[
    \Ex[X] \le n^4 \cdot (m')^4 \cdot \binom{n}{2}^{-4} \le \frac{17(m')^4}{n^4} \le 17 (1+\delta)^4\delta^3 m \le \frac{\delta m}{2},
  \]
  Markov's inequality gives $\Pr(X \le m'-m) = \Pr(X \le \delta m) \ge 1/2$. In particular, at least half of all graphs with vertex set $\{1, \dotsc, n\}$ and $m'$ edges contain a subgraph with $m$ edges and no copy of $C_4$. This implies that
  \[
    \left|\FnmCf\right|  \ge  \frac{1}{2} \cdot \binom{\binom{n}{2}}{m'} \binom{\binom{n}{2}-m}{m'-m}^{-1} =  \frac{1}{2} \cdot \binom{\binom{n}{2}}{m} \binom{m'}{m}^{-1}.
  \]
  Finally, by our assumption on $\delta$,
  \[
    \binom{m'}{m} \le \binom{(1+\delta)m}{\delta m} \le \left(\frac{e(1+\delta)}{\delta}\right)^{\delta m} \le \frac{1}{2} \cdot e^{\gamma m}.
  \]
  This completes the proof.
\end{proof}

The derivation of our second lower bound on $|\FnmCf|$ follows a similar strategy, but the simple deletion argument is replaced with the following result of Kohayakawa, Kreuter, and Steger~\cite{KoKrSt98}, stated here for the random graph $G_{n,m}$ rather than the binomial random graph $G(n,p)$. The heart of the proof of this theorem (which we shall not give here, but rather refer the reader to~\cite[Theorem~8]{KoKrSt98} or to~\cite[Appendix~A]{FeMKSa}) is a classical result of Ajtai, Koml\'os, Pintz, Spencer, and Szemer\'edi~\cite{AjKoPiSpSz82}, or rather its corollary derived by Duke, Lefmann, and R\"odl~\cite{DuLeRo95}, that gives a lower bound on the independence number of a uniform hypergraph that contains few short cycles.

\begin{thm}[{\cite{KoKrSt98}}]
  \label{thm:KoKrSt}
  There exists a constant $c$ such that if $n^{4/3} \le m \le \binom{n}{2}$, then a.a.s.\
  \[
    \ex(G_{n,m}, C_4) \ge c n^{4/3} \big( \log \big( m / n^{4/3} \big) \big)^{1/3}.
  \]
\end{thm}

\begin{prop}
  \label{prop:lower-bound-weak}
  For every $\gamma > 0$, there exists a $\delta > 0$ such that for all sufficiently large $n$ and all $m \le \delta n^{4/3} (\log n)^{1/3}$,
  \[
    \left| \FnmCf \right| \ge n^{-\gamma m} \cdot \binom{\binom{n}{2}}{m}.
  \]
\end{prop}
\begin{proof}
  Let $c$ be the constant from the statement of Theorem~\ref{thm:KoKrSt}. Given a positive $\gamma$, choose $\delta > 0$ sufficiently small so that $\delta \le c (\gamma/2)^{1/3}$, let $m' = \big\lceil n^{4/3 + \gamma/2} \big\rceil$, and observe that
  \[
    c n^{4/3} \left(\log (m'/n^{4/3})\right)^{1/3} \ge c(\gamma/2)^{1/3} n^{4/3} (\log n)^{1/3} \ge \delta n^{4/3} (\log n)^{1/3}.
  \]
  Suppose that $m \le \delta n^{4/3} (\log n)^{1/3}$. It follows from Theorem~\ref{thm:KoKrSt} that at least half of all graphs with vertex set $\{1, \dotsc, n\}$ and $m'$ edges contain a subgraph with $m$ edges and no copy of $C_4$, provided that $n$ is sufficiently large. Therefore, similarly as in the proof of Proposition~\ref{prop:lower-bound-strong},
  \[
    \left|\FnmCf\right| \ge \frac{1}{2} \cdot \binom{\binom{n}{2}}{m} \binom{m'}{m}^{-1} \ge \frac{1}{2} \cdot \left(\frac{m}{em'}\right)^m \binom{\binom{n}{2}}{m} \ge n^{-\gamma m} \cdot \binom{\binom{n}{2}}{m}.
  \]
  This completes the proof.
\end{proof}

\subsection{The number of split graphs}
\label{sec:number-split-graphs}

As we shall need to compare the family of split graphs (and graphs that are close to a split graph) to various other families of graphs, we will need to derive some estimates on its cardinality. Let $\Snm$ denote the family of split graphs with vertex set $\{1, \dotsc, n\}$ that have precisely $m$ edges. Moreover, let $\Nnm(\ell)$ denote the number of those graphs that are complete on the set $\{1, \dotsc, \ell\}$ and empty on its complement. Observe that
\[
  \Nnm(\ell) =
  \begin{cases}
    \binom{\ell(n-\ell)}{m - \binom{\ell}{2}}, & \text{if $\binom{\ell}{2} \le m \le \ell(n-\ell) + \binom{\ell}{2}$}, \\
    0, & \text{otherwise},
  \end{cases}
\]
and
\begin{equation}
  \label{eq:Snm-bounds}
  \max_\ell \Nnm(\ell) \le |\Snm| \le \sum_{\ell} \binom{n}{\ell} \Nnm(\ell).
\end{equation}
Since~\eqref{eq:Snm-bounds} is rather hard to work with due to its inexplicit form, we establish several asymptotic properties of the function $\ell \mapsto \Nnm(\ell)$, summarised in Proposition~\ref{prop:max-Nnm-bounds} below. We postpone the rather dull and technical proof of the proposition to Appendix~\ref{sec:counting-split-graphs}.

\begin{prop}
  \label{prop:max-Nnm-bounds}
  There is a positive constant $\lambda$ such that the following holds for all sufficiently large $n$. If $n \ll m \le \lambda n^2$, then the function $\ell \mapsto \Nnm(\ell)$ attains its maximum for some $\ell$ satisfying $\ellnm/2 < \ell < 2\ellnm$, where $\ellnm$ is defined by
  \[
    \ellnm = \bigg( \frac{m}{\log\big( \ellnm n/m \big)} \bigg)^{1/2}.
  \]
  Moreover, $\Nnm(\ellnm) \ge 5^m$ and if $\ell \le \ellnm/2$ or $\ell \ge 2\ellnm$, then
  \[
    \Nnm(\ell) < \exp(-m/15) \cdot \max_\ell \Nnm(\ell).
  \]
\end{prop}

\subsection{The non-structured regime}
\label{sec:non-struct-regime}

\begin{proof}[{Proof of parts~\ref{item:C4-free-quasirandom} and~\ref{item:C4-free-non-split} of Theorem~\ref{thm:asymC4free}}]
  Fix an arbitrary positive $\eps$, suppose that $m \gg n$, and let $G$ be the uniformly chosen random graph with vertex set $\{1, \dotsc, n\}$ and exactly $m$ edges. A standard averaging argument shows that if $G$ is not $\eps$-quasirandom, then it contains a subset $A$ with exactly $\eps n$ vertices and density differing from $m/\binom{n}{2}$ by more than $\eps m/\binom{n}{2}$. Consequently, Hoeffding's inequality for the hypergeometric distribution~\cite{Ho63} asserts the existence of a positive $\rho$ that depends only on $\eps$ such that
  \[
    \Pr\left(\text{$G$ is not $\eps$-quasirandom}\right) \le \binom{n}{\eps n} \cdot \exp\left(-3\rho m\right).
  \]
  It now follows from Proposition~\ref{prop:lower-bound-strong} invoked with $\gamma \leftarrow \rho$ that if $\delta$ is sufficiently small, then for all sufficiently large $n$ and all $m$ satisfying $n \ll m \le \delta n^{4/3}$,
  \begin{multline*}
    \Pr\left(\text{$G$ is not $\eps$-quasirandom} \mid G \in \FnmCf\right) \le \frac{\Pr\left(\text{$G$ is not $\eps$-quasirandom}\right)}{\Pr\left(G \in \FnmCf\right)} \\
    \le \binom{n}{\eps n} \cdot \exp(-3\rho m + \rho m) \le \exp(-\rho m).
  \end{multline*}
  In other words, graphs that are not $\eps$-quasirandom constitute only an exponentially small fraction of $\FnmCf$.

  Now, denote by $\Snm(\eps)$ the family of graphs with vertex set $\{1, \dotsc, n\}$ and $m$ edges that are $\eps$-close to a split graph. Each graph in $\Snm(\eps)$ can be obtained from some graph in $\Snm$ by removing from it some $\eps m$ edges and replacing them with arbitrarily chosen $\eps m$ edges of $K_n$. Hence, if $m \gg n$ and $n$ is sufficiently large, then
\begin{equation}\label{eq:counting:close:to:split}
    |\Snm(\eps)| \le |\Snm| \cdot \binom{m}{\eps m} \cdot \binom{\binom{n}{2}}{\eps m} \le |\Snm| \cdot \left(\frac{em}{\eps m} \cdot \frac{en^2}{2\eps m}\right)^{\eps m} \le n^{\eps m} \cdot |\Snm|.
  \end{equation}
  Moreover, it follows from~\eqref{eq:Snm-bounds} and Proposition~\ref{prop:max-Nnm-bounds} that
  \[
    \begin{split}
      |\Snm| & \le 2^n \cdot \max_\ell \Nnm(\ell) = 2^n \cdot \max_{\ell \le 2\ellnm} \Nnm(\ell) \\
      & \le 2^n \cdot \binom{2\ellnm n}{m} \le 2^n \cdot \left(\frac{2\ellnm n}{\binom{n}{2}}\right)^m \binom{\binom{n}{2}}{m}.
    \end{split}
  \]
  Suppose now that $n \ll m \le \delta n^{4/3} (\log n)^{1/3}$. As $\ellnm \ll n^{2/3}$, it follows that
  \[
    |\Snm(\eps)| \le n^{(\eps-1/3)m} \cdot \binom{\binom{n}{2}}{m}
  \]
  for all sufficiently large $n$. Therefore, by Proposition~\ref{prop:lower-bound-weak} invoked with $\gamma = 1/24$ implies that if $\delta$ is sufficiently small, then
  \[
    \Pr\left(G \in \Snm(1/4) \mid G \in \FnmCf \right) \le n^{(1/4-1/3)m} \cdot n^{\gamma m} = n^{-m/24}.
  \]
  In other words, graphs that are $1/4$-close to a split graph constitute only a super-exponentially small proportion of $\FnmCf$, as required.
\end{proof}

\section{An approximate structural theorem}
\label{sec:C4s:proof}

In this section, we shall use Theorems~\ref{thm:container} and~\ref{thm:robust-stability-refined} to construct a collection of containers for the family $\FnmCf$ whenever $n^{4/3}(\log n)^4 \le m \ll n^2$. Our aim is to do this in such a way that all but a tiny proportion of the family will be covered by containers that describe predominantly graphs that are close to a split graph. To make this notion precise, let us say that a pregraph $\cP = (M,E)$ on $n$ vertices is an \emph{$\eps$-almost split pregraph} if there exists a partition $V(K_n) = U \cup W$ such that
\[
  e(E) \le \binom{|U|}{2}, \qquad e_E(U) \ge (1-\eps)\binom{|U|}{2}, \qquad \text{and} \qquad e_M(W) \le 7\sqrt{\eps}|U|n.
\]
We will prove the following container theorem for sparse induced-$C_4$-free graphs. Recall from Section~\ref{sec:asymmetric-container-lemma} that a graph $G$ is contained in (described by) a pregraph $\cP = (M, E)$ if $E \subseteq E(G) \subseteq E \cup M$.

\begin{thm}
  \label{thm:containers:for:C4free}
  For every $\eps > 0$, there exists $\lambda > 0$ such that the following holds. For every $n \in \N$ and $n^{4/3}(\log n)^4 \le m \le \lambda n^2$, there exists a collection $\CC$ of $\eps$-almost split pregraphs on $n$ vertices with $|\CC| = e^{o(m)}$ such that all but at most $e^{-\lambda m} \cdot |\FnmCf|$ of the graphs in $\FnmCf$ are contained in some $\cP \in \CC$. 
\end{thm}

To prove Theorem~\ref{thm:containers:for:C4free}, we will apply Theorem~\ref{thm:container} recursively, starting with the trivial container, which is defined by the `complete' pregraph with $M = E(K_n)$ (and therefore $E$ empty). We continue until we obtain a family of containers, each of which admits only few good copies of $C_4$; we will be able to control this process with the use of Theorem~\ref{thm:robust-stability-refined}, which provides us with a precise structural description of such pregraphs. 
Finally, we will show that the containers that are not $\eps$-almost split pregraphs contain at most  $e^{-\lambda m} \cdot |\FnmCf|$ members of $\FnmCf$. 

More formally, we shall build a rooted tree $\T$ whose vertices are pregraphs with $n$ vertices. The root of $\T$ is the pregraph with $M = E(K_n)$ corresponding to the trivial container. The children (in $\T$) of a pregraph will correspond to refinements of it that we obtain by applying Theorem~\ref{thm:container} to one of the hypergraphs $\HH_i$ supplied by Theorem~\ref{thm:robust-stability-refined}. This way, each graph in $\FnmCf$ that is described by some pregraph $\cP$ in $\T$ will be described by one of the children of $\cP$ in $\T$. As a consequence, each graph in $\FnmCf$ will be accounted for by one of the leaves of $\T$.

In order to decide whether a pregraph $\cP = (M,E)$ should be a leaf of the tree or not (in which case we will apply Theorem~\ref{thm:container} to it), we use the following definition. 

\begin{defn}
  \label{def:leaf}
  A pregraph $\cP = (M,E)$ on $n$ vertices is a \emph{leaf pregraph} (with respect to $m$, $\eps$, and~$\delta$) if either $\cP$ is an $\eps$-almost split pregraph, or there exists $\ell \in \N$ such that 
  \begin{equation}\label{eq:leaf:ratio}
    e(E) \ge \binom{\ell}{2} \qquad \text{and} \qquad e(M) \le (1-\delta)\ell n,
  \end{equation}
  or either of the following holds:
  \begin{equation}\label{eq:leaf:Mtoosmall:or:Etoobig}
    e(E) > m \qquad \text{or} \qquad e(M) < \left( \frac{n^2 m}{2^8 \log (n^2/m)} \right)^{1/2}.
  \end{equation}
\end{defn}

Recall that, given a pregraph $\cP = (M,E)$, the $(i,4)$-uniform hypergraph $\HH_i^\cP$ comprises all pairs $(A,B)$ such that $B$ is a good copy of $C_4$ in $\cP$ and $A$ is the set of the remaining $i$ mixed edges induced by the vertex set of this copy (which induces exactly $4+i$ edges of $M$). Also, with foresight, let us set 
\begin{equation}
  \label{def:r}
  r = \frac{m}{2^{13} \log n}.
\end{equation}
We will use Theorem~\ref{thm:robust-stability-refined} to prove the following lemma. 

\begin{lemma}
  \label{lem:supersat:corollary}
  For every $\eps > 0$, there exist positive constants $\beta$, $\delta$, and $\lambda$ such that the following holds for every $n \in \N$ and $n (\log n)^2 \le m \le \lambda n^2$. Let $\cP = (M,E)$ be a pregraph on $n$ vertices that is not a leaf pregraph with respect to $m$, $\eps$, and~$\delta$. Then there exist an integer $\ell$ with $\ell^2 \ge r$ and a hypergraph $\HH \subseteq \HH_i^\cP$, for some $i \in \{0,1,2\}$, such that
  \begin{equation}
    \label{eq:supersat:corollary}
    v(\HH) \le 5\ell n, \quad  e(\HH) \ge \beta \ell^4, \quad \Delta_{(0,1)}(\HH) \le \frac{\ell^3}{n}, \quad \text{and} \quad \Delta_{(0,2)}(\HH) \le \ell
  \end{equation}
  and, if $i > 0$, then also $\Delta_{(1,0)}(\HH) \le \ell^2$.
\end{lemma}

\begin{proof}
  Let $\beta_{\ref{thm:robust-stability-refined}}$, $\delta_{\ref{thm:robust-stability-refined}}$, $\lambda_{\ref{thm:robust-stability-refined}}$, and $C_{\ref{thm:robust-stability-refined}}$ be the constants given by Theorem~\ref{thm:robust-stability-refined} applied with $\eps_{\ref{thm:robust-stability-refined}} \leftarrow \eps$, set $\beta = \beta_{\ref{thm:robust-stability-refined}}$, $C = C_{\ref{thm:robust-stability-refined}}$, $\delta = \delta_{\ref{thm:robust-stability-refined}} / 2$, and $\lambda = 2^{-8} \big( \lambda_{\ref{thm:robust-stability-refined}} / C \big)^2$. We may assume that $C \ge 1$, $\delta \le 1/4$, $\lambda \le \delta^2$, and (by our bounds on $m$) that $n \ge 1/\lambda$.

  Suppose first that there exists $\ell \ge C \sqrt{n}$ such that 
  \begin{equation}
    \label{eq:M:too:large}
    e(E) \le \binom{\ell}{2} \qquad \text{and} \qquad e(M) \ge 4\ell n
  \end{equation}
  and choose $\ell \in \N$ maximal such that $e(M) \ge 4\ell n$. We claim that $\ell^2 \ge r$. Indeed, $\cP$ is not a leaf pregraph and thus the maximality of $\ell$ and the second inequality in~\eqref{eq:leaf:Mtoosmall:or:Etoobig} give
  \[
    2r \le \frac{m}{2^{12} \log(n^2/m)} \le \frac{e(M)^2}{16n^2} \le (\ell+1)^2.
  \]
  In this case it follows immediately from Theorem~\ref{thm:robust-stability-refined} that there exists a hypergraph $\HH$ with the claimed properties. 

  Next, suppose that there exists $\ell \ge C \sqrt{n}$ and a set $U$ of size $\ell$ such that 
  \begin{equation}
    \label{eq:E:too:concentrated}
    e(E) \le \binom{\ell}{2} \qquad \text{and} \qquad e_E(U) \ge (1-\eps)\binom{\ell}{2}.
  \end{equation}
  Note that $e(M) < 4\ell n$, otherwise~\eqref{eq:M:too:large} holds and we are done as above. Since $\cP$ is not a leaf pregraph, it follows that  $\ell^2 \ge r$, as above, and $e_M(U^c) > 7\sqrt{\eps}\ell n$,  as $\cP$ is not an $\eps$-almost split pregraph. This means that $\cP$ satisfies condition~\ref{item:non-split-cond} of Theorem~\ref{thm:robust-stability-refined} and so we obtain a hypergraph $\HH$ with the claimed properties, as before.

  Finally, let $\ell \in \N$ be minimal such that $e(M) \le (1-\delta)\ell n$ and observe that $e(E) \le \binom{\ell}{2}$, since $\cP$ is not a leaf pregraph, and that 
  \[
    \ell \ge \frac{e(M)}{(1-\delta)n} \ge \left( \frac{m}{2^8 \log (n^2/m)} \right)^{1/2} \ge \max\left\{C\sqrt{n}, \sqrt{r}\right\},
  \]
  where the second inequality follows since $\cP$ is not a leaf pregraph and the third by our bounds on $m$, since $n$ is sufficiently large. It follows that $E$ is not $\eps$-close to $K_\ell$, since if it were, then there would exist a set $U$ of size $\ell$ such that $e_E(U) \ge (1-\eps)\binom{\ell}{2}$, in which case~\eqref{eq:E:too:concentrated} would hold and we would be done as before. Note also that $e(M) \ge (1 - 2\delta)\ell n$, by our choice of $\ell$ and since $\delta \ell \ge \delta \sqrt{n} \ge \delta / \sqrt{\lambda} \ge 1$.
  
  Now, observe that if~\eqref{eq:M:too:large} fails to hold, then either $e(M) \le 4Cn^{3/2}$ or 
  \[
    e(M) \le 8 n \sqrt{e(E)} \le 8n \sqrt{m} \le 8\sqrt{\lambda} n^2,
  \]
  where in the second step we used the fact that $e(E) \le m$ (which holds if $\cP$ is not a leaf pregraph) and in the third we used our upper bound on $m$. In either case, it follows that $\ell \le 2e(M) / n \le \lambda_{\ref{thm:robust-stability-refined}} n$, since $\lambda = 2^{-8} \big( \lambda_{\ref{thm:robust-stability-refined}} / C \big)^2$ and $C \ge 1$.
  Hence $\cP$ satisfies condition~\ref{item:stability-cond} of Theorem~\ref{thm:robust-stability-refined} and we again obtain the desired hypergraph $\HH$. This completes the proof of the lemma.
\end{proof}

We next combine Theorem~\ref{thm:container} and Lemma~\ref{lem:supersat:corollary} to construct a rooted tree whose leaves correspond to a family of containers for the family $\FnmCf$.

\begin{lemma}
  \label{lem:containers:for:induced:C4free:graphs}
  For every $\eps > 0$, there exist positive constants $\delta$ and $\lambda$ such that the following holds. For every $n \in \N$ and $n^{4/3}(\log n)^4 \le m \le \lambda n^2$, there exists a collection $\CC$ of $e^{o(m)}$ pregraphs on $n$ vertices such that 
  \begin{enumerate}[label={(\alph*)}]
  \item
    \label{item:containers-C4-1}
    every $\cP \in \CC$ is a leaf pregraph with respect to $m$, $\eps$, and~$\delta$ and
  \item
    \label{item:containers-C4-2}
    every graph $G \in \FnmCf$ is contained in some $\cP \in \CC$. 
  \end{enumerate}
\end{lemma}

\begin{proof}
  We will construct a rooted tree $\T$ whose vertices are pregraphs on $n$ vertices that has the following properties: 
  \begin{enumerate}[label={(\textit{\roman*})}]
  \item
    the root of $\T$ is the complete pregraph with $M = E(K_n)$;
  \item
    if $G \in \FnmCf$ is contained in a pregraph $\cP \in V(\T)$ that is not a leaf of $\T$, then $G$ is contained in some child of $\cP$ in $\T$;
  \item
    the height of $\T$ is $O(\log n)$;
  \item
    the maximum degree of $\T$ is $\exp\big( o( m / \log n) \big)$; 
  \item
    every leaf of $\T$ is a leaf pregraph with respect to $m$, $\eps$, and~$\delta$. 
  \end{enumerate}
  It will then follow immediately that the leaves of $\T$ form a collection $\CC$ as required.
  
  To define the children of a vertex $\cP \in V(\T)$, we will apply Theorem~\ref{thm:container} to the hypergraph given by Lemma~\ref{lem:supersat:corollary}. To begin, let $\beta = \beta_{\ref{lem:supersat:corollary}}$, $\delta = \delta_{\ref{lem:supersat:corollary}}$, and $\lambda = \lambda_{\ref{lem:supersat:corollary}}$ be the constants given by Lemma~\ref{lem:supersat:corollary} applied with $\eps_{\ref{lem:supersat:corollary}} \leftarrow \eps$ and set $\xi(n) = (\log\log n)^{-1}$ (here we could use any function that tends to zero sufficiently slowly as $n \to \infty$). Note that, due to the form of the statement, we may assume throughout that $n$ is sufficiently large.

  Let $\cP \in V(\T)$ and suppose that $\cP$ is not a leaf pregraph with respect to $m$, $\eps$, and~$\delta$. By Lemma~\ref{lem:supersat:corollary}, there exist $\ell \in \N$ with $\ell^2 \ge r$ and an $(i,4)$-uniform hypergraph $\HH \subseteq \HH_i^\cP$, for some $i \in \{0,1,2\}$, satisfying the assertion of the lemma. We claim that we may apply Theorem~\ref{thm:container} to the hypergraph $\HH$ with
  \[
    K = \frac{5}{\beta} \qquad \text{and} \qquad b = \xi(n) \cdot \frac{m}{(\log n)^2},
  \]
  and $r$ as defined in~\eqref{def:r}. To do so, we need to verify that~\eqref{eq:container-Del} is satisfied for every pair $(\ell_0, \ell_1) \in \{0, \ldots, i\} \times \{0, \ldots, 4\}$ with $(\ell_0, \ell_1) \neq (0,0)$. 

  \begin{claim}
    For every $(\ell_0, \ell_1) \in \{0, \ldots, i\} \times \{0, \ldots, 4\}$ with $(\ell_0, \ell_1) \neq (0,0)$, we have 
    \begin{equation}
      \label{eq:container:condition:claim}
      \Delta_{(\ell_0, \ell_1)}(\HH) \le K \cdot \frac{b^{\ell_0+\ell_1-1}}{m^{\ell_0} \cdot v(\HH)^{\ell_1}} \cdot e(\HH) \cdot \left(\frac{m}{r}\right)^{\indicator[\ell_0 > 0]}.
    \end{equation}
  \end{claim}
  
  \begin{proof}[Proof of claim]
    Observe that the right-hand side of~\eqref{eq:container:condition:claim} decreases when $\ell_0$ or $\ell_1$ increase, since $b \le r \le m$ and $v(\HH) = e(M) \ge b$, the latter holding (with room to spare) since $\cP$ is not a leaf pregraph and $m \le \lambda n^2$. Assume first that $\ell_0 \ge 2$ or $\ell_1 \ge 3$ and note that in this case $\Delta_{(\ell_0, \ell_1)}(\HH) \le 1$. It thus suffices to show that the right-hand side of~\eqref{eq:container:condition:claim} is at least $1$. Since $v(\HH) \le 5\ell n$ and $e(\HH) \ge \beta \ell^4$, see~\eqref{eq:supersat:corollary}, we have 
    \[
      K \cdot \frac{b^5}{m^2 ( 5\ell n )^4} \cdot \beta \ell^4 \cdot \frac{m}{r}  \ge  \xi(n)^6 \cdot \frac{m^3}{n^4 (\log n)^9}  \ge  1,
    \]
    since $m \ge n^{4/3} (\log n)^4$ and $n$ is sufficiently large.

    Next, recall that $\Delta_{(0,1)}(\HH) \le \ell^3/n$, by~\eqref{eq:supersat:corollary}, and observe that if $(\ell_0, \ell_1) = (0,1)$, then the right-hand side of~\eqref{eq:container:condition:claim} is at least
    \[
      K \cdot \frac{e(\HH)}{v(\HH)} \ge K \cdot \frac{\beta \ell^4}{5\ell n} \ge \frac{\ell^3}{n},
    \]
    as required. Similarly, if $i \ge 1$ then $\Delta_{(1,0)}(\HH) \le \ell^2$, by~\eqref{eq:supersat:corollary}, and if $(\ell_0, \ell_1) = (1,0)$, then the right-hand side of~\eqref{eq:container:condition:claim} is at least
    \[
      K \cdot \frac{e(\HH)}{r}  \ge K \cdot \frac{\beta \ell^4}{r} \ge \ell^2,
    \]
    since $\ell^2 \ge r$. Finally, note that $\Delta_{(1,1)}(\HH) = \Delta_{(1,2)}(\HH) \le \Delta_{(0,2)}(\HH) \le \ell$, by~\eqref{eq:supersat:corollary}. In particular, if $(\ell_0, \ell_1) \in \{(1,1), (1, 2), (0,2)\}$, then the right-hand side of~\eqref{eq:container:condition:claim} is at least
    \[
      K \cdot \frac{b^2}{m \cdot ( 5\ell n )^2} \cdot \beta \ell^4 \cdot \frac{m}{r}  \ge \xi(n)^2 \cdot \frac{m \ell^2}{n^2 (\log n)^3} \ge  \ell,
    \]
    since $m\ell \ge m \sqrt{r} \ge m^{3/2} / \log n \ge n^2 (\log n)^4$.
  \end{proof}

  Observe that $\FnmCf \cap \cP \subseteq \FF_{\le m}(\HH)$, since each $G \in \FnmCf \cap \cP$ has $m$ edges and each constraint in $\HH$ corresponds to an induced copy of $C_4$. Therefore, by Theorem~\ref{thm:container}, there exists a collection $N(\cP)$ of at most
  \[
    \binom{\binom{n}{2}}{2b} \binom{\binom{n}{2}}{4b} \le \exp\big( 12b \cdot \log n \big) = \exp\left( o\left( \frac{m}{\log n} \right) \right)
  \]
  sub-pregraphs\footnote{This means that $M(\Q) \subseteq M(\cP)$ and $E(\cP) \subseteq E(\Q) \subseteq M(\cP) \cup E(\cP)$.} $\Q$ of $\cP$ with the following properties:
  \begin{enumerate}[label={(\textit{\alph*}$'$)}]
  \item
    \label{item:children-of-P-1}
    if $\Q \in N(\cP)$, then either $M(\Q) \le (1 - c) M(\cP)$ or $E(\Q) \ge E(\cP) + cr$, and
  \item
    \label{item:children-of-P-2}
    each $G \in \FnmCf \cap \cP$ is contained in some $\Q \in N(\cP)$, 
  \end{enumerate}
  where $c = 2^{-42} K^{-1}$. We make $N(\cP)$ the set of children of $\cP$ in $\T$, observing that the degree of $\cP$ in $\T$ is $\exp\big(o(m/\log n)\big)$. By~\ref{item:children-of-P-1} and Definition~\ref{def:leaf}, the height of the tree $\T$ obtained in this way is at most
  \[
    \frac{1}{c} \log \binom{n}{2} + \frac{m}{cr} \le 2^{42} K \left( 2 \log n + 2^{13} \log n \right) = O\big( \log n \big).
  \]
  It follows that the total number of leaves of $\T$ is $e^{o(m)}$ and hence (by the definition of $\T$ and property~\ref{item:children-of-P-2}) the collection of leaves of $\T$ forms a family $\CC$ as required.
\end{proof}

To deduce Theorem~\ref{thm:containers:for:C4free}, we will show that the containers $\cP \in \CC$ that are not $\eps$-almost split pregraphs contain only an exponentially small proportion of the family $\FnmCf$

\begin{proof}[Proof of Theorem~\ref{thm:containers:for:C4free}]
  Let $\delta_{\ref{lem:containers:for:induced:C4free:graphs}}$, $\lambda_{\ref{lem:containers:for:induced:C4free:graphs}}$, and $\CC_{\ref{lem:containers:for:induced:C4free:graphs}}$ be (respectively) the constants and the family of containers given by Lemma~\ref{lem:containers:for:induced:C4free:graphs} applied with $\eps_{\ref{lem:containers:for:induced:C4free:graphs}} \leftarrow \eps$, let $\lambda_{\ref{prop:max-Nnm-bounds}}$ be the constant given by Proposition~\ref{prop:max-Nnm-bounds}, and set $\delta = \delta_{\ref{lem:containers:for:induced:C4free:graphs}}$, $\lambda = \min\{ \lambda_{\ref{prop:max-Nnm-bounds}}, \lambda_{\ref{lem:containers:for:induced:C4free:graphs}}, 2^{-8}\delta^2 \}$, and $\CC' = \CC_{\ref{lem:containers:for:induced:C4free:graphs}}$. Note that we may assume (without loss of generality) that $\delta \le 1$ and recall that $|\CC'| = e^{o(m)}$. We claim that the collection
  \[
    \CC := \big\{ \cP \in \CC' \scolon \cP \textup{ is an $\eps$-almost split pregraph} \big\}
  \]
  has the property that all but at most $e^{-\lambda m} \cdot |\FnmCf|$ of the graphs in $\FnmCf$ are contained in some $\cP \in \CC$ and therefore $\CC$ is the required family of `almost' containers.

  To prove this, we will give an upper bound on the number of graphs in $\FnmCf$ that belong to a single container $\cP \in \CC' \setminus \CC$. Recall that every $\cP \in \CC'$ is a leaf pregraph with respect to $m$, $\eps$, and~$\delta$ and therefore we may assume that $\cP = (M,E)$ satisfies either~\eqref{eq:leaf:ratio} or~\eqref{eq:leaf:Mtoosmall:or:Etoobig}.

  \medskip
  \noindent
  \textbf{Case 1.} Either $e(E) > m$ or $e(M) < 2^{-4} \sqrt{ n^2 m / \log (n^2/m) }$. 
  
  \smallskip

  We may assume that $e(E) \le m$, as otherwise $\FnmCf \cap \cP$ is empty. Therefore
  \[
    |\FnmCf \cap \cP| \le \binom{e(M)}{m-e(E)} \le \binom{e(M) + e(E)}{m} \le \binom{2^{-4} \sqrt{ n^2 m / \log (n^2/m) } + m}{m}
  \]
  and hence, since $m \le 2^{-4} \sqrt{ n^2 m / \log (n^2/m) }$ for every $m \le \lambda n^2$, we obtain
  \[
    |\FnmCf \cap \cP| \le \binom{2^{-3} \sqrt{ n^2 m / \log (n^2/m) }}{m}.
  \]
  We claim that for some well-chosen $\ell \in \N$,
  \begin{equation}
    \label{eq:NP-case-1}
    \binom{2^{-3} \sqrt{ n^2 m / \log (n^2/m) }}{m} \le 2^{-m} \cdot \Nnm(\ell) \le 2^{-m} \cdot |\FnmCf|,
  \end{equation}
  where $\Nnm(\ell)$ (cf.~Section~\ref{sec:number-split-graphs}) denotes the number of graphs with vertex set $\{1, \dotsc, n\}$ and precisely $m$ edges that are complete on the set $\{1, \dotsc, \ell\}$ and empty on its complement. To prove~\eqref{eq:NP-case-1}, note first that 
  \begin{equation}
    \label{eq:binomials}
    \binom{a}{c} \ge \bigg( \frac{a}{b} \bigg)^c \binom{b}{c} \qquad \textup{and} \qquad \binom{a}{b} \le \left(\frac{a}{b-c}\right)^c \binom{a}{b-c}
  \end{equation}
  for every $a \ge b \ge c \ge 0$ and choose $\ell \in \N$ so that 
  \[
    \sqrt{\frac{3m}{2\log(n^2/m)}}  \le  \ell  \le  \sqrt{\frac{2m}{\log(n^2/m)}},
  \]
  so, in particular, $\ell(n-\ell) \ge \sqrt{ n^2 m / \log (n^2/m) }$. It follows that
  \begin{equation}
    \label{eq:NP:binomials:calc}
    \binom{2^{-3} \sqrt{ n^2 m / \log (n^2/m) }}{m} \le 2^{-3m} \left(\frac{\ell(n-\ell)}{m - \binom{\ell}{2}}\right)^{\binom{\ell}{2}} \binom{\ell(n-\ell)}{m - \binom{\ell}{2}}
  \end{equation}
  and, since ${\ell \choose 2} \le m / \log(n^2/m)$ and $m \le \lambda n^2$, the right-hand side of~\eqref{eq:NP:binomials:calc} is at most  
  \[
    2^{-3m} \left(\frac{n^2}{m}\right)^{\binom{\ell}{2}} \binom{\ell(n-\ell)}{m - \binom{\ell}{2}} \le 2^{-3m} \cdot e^m \cdot \Nnm(\ell) \le 2^{-m} \cdot \Nnm(\ell),
  \]
  as claimed. It follows that there are at most $2^{-m} \cdot |\FnmCf |$ graphs in $\FnmCf \cap \cP$.

  \medskip
  \noindent
  \textbf{Case 2.} There exists $\ell \in \N$ such that $e(E) \ge \binom{\ell}{2}$ and $e(M) \le (1-\delta)\ell n$.

  \smallskip

  We may again assume that $e(E) \le m$, as otherwise $\FnmCf \cap \cP$ is empty. Since ${\ell \choose 2} \le m \le \lambda n^2 \le 2^{-8} \delta^2 n^2$, it follows (using~\eqref{eq:binomials}) that
  \begin{align*}
    |\FnmCf \cap \cP| &  \le  \binom{(1-\delta)\ell n}{m - e(E)} \le \binom{(1-\delta/2)\ell(n-\ell)}{m-e(E)} \\
                     &  \le  \max\left\{4^m, \left(1-\frac{\delta}{2}\right)^{m - \binom{\ell}{2}} \binom{\ell(n-\ell)}{m-\binom{\ell}{2}}\right\},
  \end{align*}
  where the bound $4^m$ corresponds to the case $(1-\delta/2)\ell(n-\ell) \le 2m$. However, since $\lambda \le \lambda_{\ref{prop:max-Nnm-bounds}}$, it follows from Proposition~\ref{prop:max-Nnm-bounds} that
  \[
    4^m < e^{-m/5} \cdot 5^m \le e^{-m/5} \cdot \Nnm(\ellnm) \le e^{-m/5} \cdot |\FnmCf |,
  \]
  where $\ellnm$ is defined by $\ellnm = \big( m / \log(\ellnm n/m) \big)^{1/2}$, see Proposition~\ref{prop:max-Nnm-bounds}. It will therefore suffice to bound the second term in the maximum above.

  To do so, we will consider the cases $\ell \le 2\ellnm$ and $\ell > 2\ellnm$ separately. If $\ell > 2\ellnm$, then it follows from Proposition~\ref{prop:max-Nnm-bounds} that
  \[
    \binom{\ell(n-\ell)}{m-\binom{\ell}{2}} = \Nnm(\ell) \le e^{-m/15} \cdot \max_\ell \Nnm(\ell) \le e^{-m/15} \cdot |\FnmCf |.
  \]
  On the other hand, if $\ell \le 2\ellnm$, then we will show that ${\ell \choose 2} \le m/2$. Indeed,
  \[
    \binom{\ell}{2}  \le  2\ellnm^2  =  \frac{2m}{\log(\ellnm n/m)}  =  \frac{4m}{\log(n^2/m) - \log\log(\ellnm n/m)}  \le  \frac{m}{2},
  \]
  since $\ellnm \le n$, $m \le \lambda n^2$, and $\lambda \le 2^{-8}$. It follows that
  \[
    \left(1-\frac{\delta}{2}\right)^{m - \binom{\ell}{2}} \binom{\ell(n-\ell)}{m-\binom{\ell}{2}} \le e^{-\delta m / 4} \cdot \Nnm(\ell) \le e^{-\delta m / 4
} \cdot |\FnmCf |.
  \]
  We have thus shown that there are at most $e^{-\delta m / 4} \cdot |\FnmCf|$ graphs in $\FnmCf \cap \cP$
  
  \medskip

  Summing over the $e^{o(m)}$ pregraphs in $\CC' \setminus \CC$, it follows that at most $e^{-\lambda m} \cdot |\FnmCf |$ of the graphs in $\FnmCf$ are contained in some $\cP \in \CC'$ that is not an $\eps$-almost split pregraph. Since, by property~\ref{item:containers-C4-2} of Lemma~\ref{lem:containers:for:induced:C4free:graphs}, the remaining graphs in $\FnmCf$ are contained in some $\cP \in \CC$, the theorem follows.
\end{proof}

We are finally ready to prove part~\ref{item:C4-free-almost-split} of Theorem~\ref{thm:asymC4free}. To deduce from Theorem~\ref{thm:containers:for:C4free} that almost all graphs in $\FnmCf$ have the claimed structure, it only remains to bound the number of such graphs that are not $\eps$-close to a split graph but are contained in a pregraph that is $\eps'$-close to a split pregraph (for some well-chosen $\eps'$).  

\begin{proof}[{Proof of part~\ref{item:C4-free-almost-split} of Theorem~\ref{thm:asymC4free}}]
Assume (without loss of generality) that $\eps > 0$ is sufficiently small, and set $\delta = \min\big\{ \lambda_{\ref{thm:containers:for:C4free}}, \eps^4 \big\}$, where $\lambda_{\ref{thm:containers:for:C4free}}$ is the constant obtained by applying Theorem~\ref{thm:containers:for:C4free} with $\eps_{\ref{lem:containers:for:induced:C4free:graphs}} \leftarrow \eps^3$. Now, given $n^{4/3} (\log n)^4 \le m \le \delta n^2$, it follows from Theorem~\ref{thm:containers:for:C4free} that there exists a collection $\CC$ of $\eps^3$-almost split pregraphs on $n$ vertices with $|\CC| = e^{o(m)}$ and such that at most $e^{-\delta m} \cdot |\FnmCf |$ graphs in $\FnmCf$ are not contained in any $\cP \in \CC$.

  We claim that, for each $\cP \in \CC$, 
  \begin{equation}
    \label{eq:main:proof:aim}
    \big| \big\{ G \in \FnmCf \cap \cP \scolon G \textup{ is not $\eps$-close to a split graph} \big\} \big| \le e^{-\eps m} \cdot |\FnmCf |.
  \end{equation}
  Let $V(K_n) = U \cup W$ be a partition witnessing the fact that $\cP = (M,E)$ is an $\eps^3$-almost split pregraph and recall that 
  \begin{equation}
    \label{eq:almost:split:def:recall}
    e(E) \le \binom{|U|}{2}, \qquad e_E(U) \ge \big( 1 - \eps^3 \big)\binom{|U|}{2}, \qquad \text{and} \qquad e_M(W) \le 7 \eps^{3/2} |U| n.
  \end{equation}
  Let $\ell$ be the largest integer such that $e_E(U) \ge \binom{\ell}{2}$ and note that $\ell \le |U| \le (1+\eps^3)(\ell+2)$. As usual, we may assume that $e(E) \le m$, since otherwise $\FnmCf \cap \cP$ is empty; note that therefore $\ell \le \sqrt{3m} \le \sqrt{3\delta} \cdot n$. We may also assume that $\ell \ge 2^{-5} \sqrt{m / \log (n^2/m)}$, since otherwise 
  \[
    e(M)  \le  \eps^3 \binom{|U|}{2} + |U|(n-|U|) + 7 \eps^{3/2} |U| n  \le  2\ell n  \le  2^{-4} \sqrt{ \frac{n^2 m}{\log (n^2/m)} },
  \]
  and in Case~1 of the proof of Theorem~\ref{thm:containers:for:C4free} we showed that if this is the case then $|\FnmCf \cap \cP| \le 2^{-m} \cdot |\FnmCf |$, as required. It follows that
  \[
    m  \le  2^{10} \ell^2 \log \frac{n^2}{m}  \le  2^{10} \ell n \cdot \frac{\ell}{n} \log \frac{3n^2}{\ell^2}  \le  2^{10} \ell n \cdot \sqrt{3\delta} \log \frac{1}{\delta} \le \eps^{3/2} \ell n,
  \]
  where the first inequality follows from $\ell \ge 2^{-5} \sqrt{m / \log (n^2/m)}$, the second since $\binom{\ell}{2} \le e(E) \le m$, the third since $\ell \le \sqrt{3\delta} \cdot n$, and the fourth since $\delta \le \eps^4$.

  Now, observe that, by~\eqref{eq:almost:split:def:recall}, if $G \in \FnmCf \cap \cP$ is not $\eps$-close to a split graph, then $G$ has at least $\eps m$ edges in the set $W$. It follows that the left-hand side of~\eqref{eq:main:proof:aim} is at most
  \[
    \sum_{s \ge \eps m} \binom{e_M(U) + e_M(U, W)}{m - s - e_E(U) - e_E(U, W)} \binom{e_M(W)}{s - e_E(W)}.
  \]
  Noting that $e_M(U) \le 2\eps^3 {\ell \choose 2}$, $e_M(U, W) \le (1 + 2\eps^3) \ell (n - \ell)$, and $e_M(W) \le 8 \eps^{3/2} \ell n$, this is in turn at most
  \begin{equation}
    \label{eq:main:proof:sum}
    \sum_{s \ge \eps m} \binom{\ell (n - \ell) + 2\eps^3 \ell n}{m - s - e_E(U) - e_E(U, W)} \binom{8 \eps^{3/2} \ell n}{s - e_E(W)}.
  \end{equation}
  To bound this sum, note first that the inequalities $m \le \eps^{3/2} \ell n$ and $\ell < \sqrt{3\delta} n < n/2$ imply that $\ell(n-\ell) \ge 2m$, and hence 
  \begin{equation}
    \label{eq:main:proof:firstterm}
    \binom{\ell (n - \ell) + 2\eps^3 \ell n}{m - s - e_E(U) - e_E(U, W)}  \le  \binom{\ell (n - \ell) + 2\eps^3 \ell n}{m - s - \binom{\ell}{2}},
  \end{equation}
  since $e_E(U) \ge \binom{\ell}{2}$. Now, using the inequalities $\binom{a+c}{b} \le \left(\frac{a+c-b}{a-b}\right)^b \binom{a}{b}$ and $\binom{a}{b-c} \le \left(\frac{b}{a-b}\right)^c \binom{a}{b}$ and the bounds $\ell(n-\ell) \ge 2m$ and $\ell < \sqrt{3\delta} \cdot n \le n/3$, we can bound the right-hand side of~\eqref{eq:main:proof:firstterm} from above by
  \[
    \left( 1 + \frac{4\eps^3 \ell n}{\ell(n-\ell)} \right)^m \left(\frac{2m}{\ell(n-\ell)}\right)^s \binom{\ell(n-\ell)}{m-\binom{\ell}{2}} \le \big( 1 + 6\eps^3 \big)^m \left(\frac{3m}{\ell n} \right)^s \Nnm(\ell).
  \]
  Observe also that
  \[
    \binom{8 \eps^{3/2} \ell n}{s - e_E(W)} \le \binom{8 \eps^{3/2} \ell n}{s} \le \left( \frac{8 e \eps^{3/2} \ell n}{s} \right)^s,
  \]
  since $s = e(G[W]) \le m \le \eps^{3/2} \ell n$. It follows that~\eqref{eq:main:proof:sum} is at most
  \[
    \sum_{s \ge \eps m} \big( 1 + 6\eps^3 \big)^m \left( \frac{8 e \eps^{3/2} \ell n}{s} \right)^s \left(\frac{3m}{\ell n} \right)^s \Nnm(\ell),
  \]
  which is easily bounded from above by
  \[
    e^{6\eps^3 m} \Nnm(\ell) \sum_{s \ge \eps m} \left( \frac{24 e \eps^{3/2} m}{s} \right)^s  \le  m \big( 24 e \eps^{1/2} \big)^{\eps m} e^{6\eps^3 m} \Nnm(\ell)   \le  e^{-\eps m} \Nnm(\ell),
  \]
  proving~\eqref{eq:main:proof:aim}. Since $|\CC| = e^{o(m)}$ and at most $e^{-\delta m} \cdot |\FnmCf |$ graphs in $\FnmCf$ are not contained in any $\cP \in \CC$, it follows that almost all graphs in $\FnmCf$ are $\eps$-close to a split graph, as required. This completes the proof of Theorem~\ref{thm:asymC4free}.
\end{proof}

It only remains to prove Corollary~\ref{cor:GnpC4free}. We will in fact use Theorem~\ref{thm:containers:for:C4free}, together with Theorem~\ref{thm:asymC4free} and Proposition~\ref{prop:lower-bound-strong}, to prove the following slightly stronger (and more technical) statement. 


\begin{cor}\label{cor:GnpC4free:technical}
  For every $\eps > 0$, there exists $\delta > 0$ such that the following holds a.a.s.\ for $G \sim \Gnpind(C_4)$:
  \begin{enumerate}[label={(\alph*)}]
  \item
    \label{item:Gnp-C4-free-bottom}
    If $n^{-1} \ll p \le \delta n^{-2/3}$, then $G$ is $\eps$-quasirandom, and
    \begin{equation}\label{eq:cor:GnpC4free:a}
      e(G) \in \big( 1 \pm \eps \big) p \binom{n}{2}.
    \end{equation}
  \item
    \label{item:Gnp-C4-free-middle}
    If $\delta n^{-2/3} \le p \le n^{-1/3} (\log n)^4$, then 
    \[
      \frac{\delta n^{4/3}}{4} \le e(G) \le n^{4/3} (\log n)^{8}.
    \]
  \item
    \label{item:Gnp-C4-free-top}
    If $n^{-1/3} (\log n)^4 \le p \ll 1$, then $G$ is $\eps$-close to a split graph and
    \[
      e(G) = \Theta\left( \frac{p^2 n^2}{\log(1/p)} \right).
    \]
  \end{enumerate}
\end{cor}

\begin{proof}
  Let $G \sim \Gnpind(C_4)$ be the (random) graph obtained by conditioning $G(n,p)$ to contain no induced $4$-cycle and let $E_m$ denote the event that $G$ has exactly $m$ edges. Observe that $\Pr(E_m) \propto \varphi(m)$, where
  \[
    \varphi(m) = |\FnmCf| \cdot \left(\frac{p}{1-p}\right)^m,
  \]
  and that the distribution of $G$ conditioned on $E_m$ is uniform on $\FnmCf$. Our main task will be to find the values of $m$ for which $\varphi(m)$ is the largest. In order to show that $\Pr\left(\bigcup_{m \in R} E_m\right) \ge 1-\alpha$ for some $R \subseteq \{0, \dotsc, \binom{n}{2}\}$ and $\alpha \ge 0$, it is enough to prove that for some $m \in R$ and all $m' \not\in R$, we have $\alpha \cdot \varphi(m) \ge n^2 \cdot \varphi(m')$. We first observe the following straightforward lower and upper bounds on $\varphi(m)$.
  
  \begin{claimGnp}
    \label{claim:Gnp-trivial}
    There exists an absolute constant $c$ such that for every $m$,
    \[
      \left(\frac{cnp}{\sqrt{m \log \left(n^2/m\right)}}\right)^m \le \varphi(m) \le \left(\frac{en^2p}{2m(1-p)}\right)^m.
    \]
  \end{claimGnp}
  \begin{proof}
    The upper bound follows by noting that
    \[
      |\FnmCf| \le \binom{\binom{n}{2}}{m} \le \left(\frac{e\binom{n}{2}}{m}\right)^m \le \left(\frac{en^2}{2m}\right)^m.
    \]
    The lower bound follows by observing that
    \begin{align*}
      |\FnmCf| & \, \ge |\Snm| \ge \Nnm\left(\sqrt{m/\log\left(n^2/m\right)}\right) \\
     & \, \ge \left(\frac{n\sqrt{m}}{2m\sqrt{\log\left(n^2/m\right)}}\right)^{(1-1/\log(n^2/m)) \cdot m} \ge \left(\frac{cn}{\sqrt{m \log \left(n^2/m\right)}}\right)^m.
    \end{align*}
  \end{proof}

  We next choose the constant $\delta = \delta(\eps) > 0$ as follows. First, note that the function $x \mapsto (e/x)^x$ is strictly increasing for $x \in (0,1]$ and strictly decreasing for $x \in [1, \infty)$, so there is a $\gamma > 0$ such that $(e/x)^x \le e-2\gamma$ whenever $x \not\in [1-\eps/2, 1+\eps/2]$. Fix such a $\gamma$ and let $\delta_{\ref{prop:lower-bound-strong}}$ be the constant given by applying Proposition~\ref{prop:lower-bound-strong} with $\gamma_{\ref{prop:lower-bound-strong}} \leftarrow \gamma/4$. Moreover, let $\delta_{\ref{thm:asymC4free}}$ be the constant given by Theorem~\ref{thm:asymC4free} applied with $\eps_{\ref{thm:asymC4free}} \leftarrow \eps$ and set $\delta := \min\{ \delta_{\ref{thm:asymC4free}}, \delta_{\ref{prop:lower-bound-strong}} \}$. The following claim is an immediate consequence of Proposition~\ref{prop:lower-bound-strong}.

  \begin{claimGnp}
    \label{claim:Gnp-deletion}
    If $1 \ll m \le \delta n^{4/3}$, then
    \[
      \varphi(m) \ge \left(\frac{(e-\gamma)n^2p}{2m(1-p)}\right)^m
    \]
    for all sufficiently large $n \in \N$.
  \end{claimGnp}

  \begin{proof}
    By Proposition~\ref{prop:lower-bound-strong} and our choice of $\delta$, we have
    \[
      \left| \FnmCf \right| \ge e^{-\gamma m / 4} \cdot \binom{\binom{n}{2}}{m} \ge e^{-\gamma m / 4} \cdot \left(\frac{(e - \gamma/4)n^2}{2m}\right)^m
    \]
    for all sufficiently large $n$, where in the second inequality we used the fact that $1 \ll m \ll n^2$. Since $e^{-\gamma / 4} \cdot (e - \gamma/4) \ge (1-\gamma/4)(e-\gamma/4) > e - \gamma$, the claimed bound follows.
  \end{proof}


  Finally, Theorem~\ref{thm:containers:for:C4free} gives the following upper bound on $\varphi(m)$.

  \begin{claimGnp}
    \label{claim:Gnp-containers}
    There exists an absolute constant $C$ such that if $n^{4/3} (\log n)^4 \le m \ll n^2$, then
    \[
      \varphi(m) \le \left(\frac{Cnp}{\sqrt{m \log \left(n^2/m\right)}}\right)^m.
    \]
  \end{claimGnp}
  \begin{proof}
    Suppose that $n^{4/3}(\log n)^4 \le m \ll n^2$. By Theorem~\ref{thm:containers:for:C4free}, almost all graphs in $\FnmCf$ are contained in one of at most $e^{o(m)}$ pregraphs that are $(1/4)$-almost split. A given pregraph $\cP = (M, E)$ contains exactly $\binom{e(M)}{m-e(E)}$ graphs with $m$ edges and if $\cP$ is $(1/4)$-almost split, then
    \[
      \frac{3}{4}\binom{u}{2} \le e(E) \le \binom{u}{2} \qquad \text{and} \qquad e(M) \le \frac{9}{2}un
    \]
    for some integer $u$. Let $\ell$ be the smallest integer such that $e(E) \le \binom{\ell}{2}$. It follows that $\binom{\ell-1}{2} \le e(E) \le \binom{\ell}{2}$ and $e(M) \le 6\ell n$. In particular,
    \[
      \binom{e(M)}{m-e(E)} \le \max_\ell \binom{6\ell n}{m-\binom{\ell}{2}} \cdot (6\ell n)^\ell \le 7^m \cdot \max_\ell \binom{\ell(n-\ell)}{m - \binom{\ell}{2}}.
    \]
    By Proposition~\ref{prop:max-Nnm-bounds}, the maximum above is attained at some $\ell$ satisfying $\ellnm/2 \le \ell \le 2\ellnm$, that is, $\ell = \Theta\left(\sqrt{m/\log\left(n^2/m\right)}\right)$. It follows that for some absolute constants $C$ and $C'$,
    \[
      |\FnmCf| \le e^{C'm} \cdot \Nnm\left(\sqrt{m/\log\left(n^2/m\right)}\right) \le \left(\frac{Cn}{2\sqrt{m \log \left(n^2/m\right)}}\right)^m,
    \]
    which implies the claimed bound on $\varphi(m)$.
  \end{proof}

  We will now use Claims~\ref{claim:Gnp-trivial},~\ref{claim:Gnp-deletion}, and~\ref{claim:Gnp-containers} to bound the ratios $\varphi(m) / \varphi(m')$ for various $m$ and $m'$. Suppose first that $n^{-1} \ll p \le \delta n^{-2/3}$, set $m = pn^2/2$, and let $m' \not\in (1\pm\eps)p\binom{n}{2}$. By Claim~\ref{claim:Gnp-deletion},
  \[
    \varphi(m) \ge \left(\frac{(e-\gamma)n^2p}{2m(1-p)}\right)^m = \left(\frac{e-\gamma}{1-p}\right)^m
  \]
  and by Claim~\ref{claim:Gnp-trivial},
  \[
    \varphi(m') \le \left(\frac{en^2p}{2m'(1-p)}\right)^{m'} = \left(\frac{em}{m'(1-p)}\right)^{\frac{m'}{m} \cdot m}.
  \]
  If $m' > 2em$, then $\varphi(m') \le 1 \le 2^{-m} \varphi(m)$, so we may assume that $m' \le 2em$. Since $m' \not\in (1\pm\eps/2)m$, then
  \[
    \left(\frac{em}{m'}\right)^{\frac{m'}{m}} \le e-2\gamma
  \]
  by our definition of $\gamma$, implying that
  \[
    \frac{\varphi(m')}{\varphi(m)} \le \left(\frac{e-2\gamma}{e-\gamma}\right)^m \cdot (1-p)^{m-m'} \le \exp\big( - \gamma m / e + 2\eps mp \big) \le \exp\big( - \gamma m / 3 \big).
  \]
  It follows that if $n^{-1} \ll p \le \delta n^{-2/3}$, then $\Pr\big(e(G) \in (1\pm\eps) p \binom{n}{2}\big) \ge 1 - e^{-\gamma m /4}$. In particular, it follows from Theorem~\ref{thm:asymC4free} that $G$ is a.a.s.\ $\eps$-quasirandom. This establishes part~\ref{item:Gnp-C4-free-bottom} of the corollary.

  Suppose now that $\delta n^{-2/3} \le p \le n^{-1/3}(\log n)^4$ and let $m = \delta n^{4/3}/2$. By Claim~\ref{claim:Gnp-deletion},
  \[
    \varphi(m) \ge \left(\frac{(e-\gamma)n^2p}{2m(1-p)}\right)^m \ge 2^m,
  \]
  where the second inequality holds as $m \le pn^2/2$. If $m' \le m/2$, then by Claim~\ref{claim:Gnp-trivial} and since $x \mapsto (a/x)^x$ is increasing for $x \le a/e$,
  \[
    \varphi(m') \le \left(\frac{en^2p}{2m'(1-p)}\right)^{m'} \le \left(\frac{en^2p}{m(1-p)}\right)^{m/2} \le \left(\frac {4 e m (1-p)}{(e-\gamma)^2n^2p}\right)^{m/2}  \cdot \varphi(m).
  \]
  Since $m \le pn^2/2$, it follows that
  \[
    \varphi(m') \le \left(\frac{2 e (1-p)}{(e-\gamma)^2}\right)^{m/2} \cdot \varphi(m) \le \left(\frac{3}{4}\right)^{m/2} \cdot \varphi(m).
  \]
  On the other hand, if $m' \ge n^{4/3}(\log n)^8$, then by Claim~\ref{claim:Gnp-containers},
  \[
    \varphi(m') \le \left(\frac{Cnp}{\sqrt{m' \log \left(n^2/m'\right)}}\right)^{m'} \le 1 \le 2^{-m} \cdot \varphi(m).
  \]
  It follows that if $\delta n^{-2/3} \le p \le n^{-1/3}(\log n)^4$, then $\Pr\big(\delta n^{4/3}/4 \le e(G) \le n^{4/3}(\log n)^8\big) \ge 1 - e^{-m/10}$, establishing part~\ref{item:Gnp-C4-free-middle} of the corollary.

  Finally, suppose that $n^{-1/3}(\log n)^4 \le p \ll 1$, let $m = \beta p^2n^2 / \log(1/p)$ for some small positive constant $\beta$, and observe that $m \ge n^{4/3}(\log n)^6$. By Claim~\ref{claim:Gnp-trivial},
  \[
    \varphi(m) \ge \left(\frac{cnp}{\sqrt{m \log \left(n^2/m\right)}}\right)^m \ge \left(\frac{cnp\sqrt{\log(1/p)}}{\sqrt{\beta} \cdot pn \cdot \sqrt{3\log (1/p)}}\right)^m  \ge e^m,
  \]
  since $\beta$ is small. If $m' \le n^{4/3}(\log n)^4$, then by Claim~\ref{claim:Gnp-trivial},
  \[
    \varphi(m') \le \left(\frac{en^2p}{2m'(1-p)}\right)^{m'} \le n^{2m'} \le \exp\big( 2 n^{4/3}(\log n)^5 \big) \le e^{m/2} \le e^{-m/2} \cdot \varphi(m).
  \]
  Let $\xi$ be a small positive constant. If $m' \ge m/\xi$, then by Claim~\ref{claim:Gnp-containers},
  \[
    \varphi(m') \le \left(\frac{Cnp}{\sqrt{m' \log \left(n^2/m'\right)}}\right)^{m'} \le 1 \le e^{-m} \cdot \varphi(m).
  \]
  Finally, suppose that $n^{4/3} (\log n)^4 \le m' \le \xi m$. By Claim~\ref{claim:Gnp-containers},
  \[
    \varphi(m') \le \left(\frac{Cnp}{\sqrt{m' \log \left(n^2/m'\right)}}\right)^{m'} \le \left(\frac{Cnp}{\sqrt{m' \log \left(n^2/m\right)}}\right)^{m'}.
  \]
  Since the function $x \mapsto (a/\sqrt{x})^x$ is increasing if $x \le a^2/e$, then the right-hand side above is increasing in $m' \le \xi m$ and thus
  \[
    \varphi(m') \le \left(\frac{Cnp}{\sqrt{\xi m \log \left(n^2/m\right)}}\right)^{\xi m} \le \left(\frac{C}{\sqrt{\beta \xi}}\right)^{\xi m} \le e^{m/2} \le e^{-m/2} \cdot \varphi(m). 
  \]
  It follows that there is an absolute constant $K$ such that if $n^{-1/3}(\log n)^4 \le p \ll 1$, then
  \[
    \Pr\left(\frac{p^2n^2}{K \log(1/p)} \le e(G) \le \frac{Kp^2n^2}{\log(1/p)}\right) \ge 1 - e^{-m/3}.
  \]
  In particular, it follows from Theorem~\ref{thm:asymC4free} that a.a.s.\ $G$ is $\eps$-close to a split graph. This completes the proof of the corollary.
\end{proof}

\section{Concluding remarks and open problems}
\label{sec:open:problems}

This paper makes a first step towards understanding the typical structure of a sparse member of a hereditary graph property. In this final section, we discuss a few of the many natural open problems suggested by our main results. We begin with the following conjecture on the typical structure of sparse induced-$C_4$-free graphs.

\begin{conj}
  \label{conj:asymC4free}
  Suppose that $n^{4/3}(\log n)^{1/3} \ll m \le \binom{n}{2} - \Omega(n^2)$ 
  and let $G$ be a uniformly chosen random graph in $\FnmCf$. Then a.a.s.\ $G$ is a split graph.
\end{conj}

Note that Conjecture~\ref{conj:asymC4free} would sharpen Theorem~\ref{thm:asymC4free} in three different ways. First, the power of $\log n$ in the lower bound on $m$ in part~\ref{item:C4-free-almost-split} of Theorem~\ref{thm:asymC4free} would be reduced from $4$ to $1/3$, which is best possible, as shown by part~\ref{item:C4-free-non-split} of Theorem~\ref{thm:asymC4free}. Second, the description of the typical members of $\FnmCf$ would be more precise -- the graphs are required to be split rather than only close to split. Finally, the upper bound on $m$ is increased from $o(n^2)$ to $\binom{n}{2} - \Omega(n^2)$; in fact, we expect Conjecture~\ref{conj:asymC4free} to remain true even when $\binom{n}{2} - m$ is much smaller than $n^2$, but then it is (arguably) more natural to consider the complements of graphs in $\FnmCf$, which are sparse induced-$2K_2$-free graphs.

We made the assumption that $m = o(n^2)$ mainly for convenience (and to simplify the proof of Theorem~\ref{thm:robust-stability}) and it seems plausible that our techniques could be extended to all $m \le \binom{n}{2} - \Omega(n^2)$, but we have not made any serious attempts to do so. A natural alternative approach of resolving Conjecture~\ref{conj:asymC4free} for $m = \Theta(n^2)$ would be to generalise the method of Pr\"omel and Steger~\cite{PrSt91}, who characterised typical members of $\Fn(C_4)$. 

A more substantial step towards resolving the conjecture would be to either determine the precise structure of a typical member of $\FnmCf$ when $m \ge n^{4/3}(\log n)^{O(1)}$ or to determine the approximate structure in the range $n^{4/3}(\log n)^{1/3} \ll m \le n^{4/3}(\log n)^4$. We remark that the methods of~\cite{BaMoSaWa16} may be helpful in achieving the former goal (determining the precise structure), though it appears that new ideas will be needed.

\subsection{Sparse induced-$H$-free graphs}

Perhaps the most natural direction for further investigation would be to describe the typical structure of sparse induced-$H$-free graphs for an arbitrary graph $H$. The first step in this direction was made recently by Kalvari and Samotij~\cite{KaSa}, who used Theorem~\ref{thm:container} to prove the following rough characterisation of a typical member of $\Fnm(H)$ for all non-bipartite graphs $H$ and all $m \ll n^2$.

\begin{thm}
  \label{thm:KaSa}
  Suppose that $H$ is a non-bipartite graph and let $G$ be uniformly chosen random graph in $\Fnm(H)$. The following holds for every $\eps > 0$:
  \begin{enumerate}[label=(\alph*)]
  \item
    If $n \ll m \ll n^{2-1/m_2(H)}$, then a.a.s.\ $G$ is $\eps$-quasirandom.
  \item
    \label{item:KaSa-2}
    If $n^{2-1/m_2(H)} \ll m \ll n^2$, then a.a.s.\ $G$ is $\eps$-close to $(\chi(H)-1)$-partite.
  \end{enumerate}
\end{thm}

We remark that the structural characterisation of typical sparse induced-$H$-free graphs provided by assertion~\ref{item:KaSa-2} of Theorem~\ref{thm:KaSa} is not as precise as that given in Theorem~\ref{thm:asymC4free}~\ref{item:C4-free-almost-split}. We expect that for many $H$, the following holds for all $m \gg n^{2-1/m_2(H)}$: the vertex set of a typical member of $\Fnm(H)$ can be partitioned into $\chi(H)-1$ sets that are `almost independent' (i.e., they induce $o(m)$ edges) and some number of `almost-cliques' (that is, sets inducing graphs of density $1-o(1)$) of size $\Theta\big(\sqrt{m/\log n}\big)$. However, it is not true, in general, that the typical member of $\Fnm(H)$ remains $o(1)$-close to $(\chi(H)-1)$-partite when $m = \Omega(n^2)$. For example, Pr\"omel and Steger~\cite{PrStBerge} proved that almost all graphs in $\Fn(C_5)$ are so-called generalised split graphs. A graph $G$ is a \emph{generalised split graph} if the vertex set of either $G$ or of its complement can be partitioned into sets $V_1$ and $V_2$ such that $V_1$ induces a union of pairwise disjoint cliques and $V_2$ induces a clique.

\subsection{General hereditary properties of graphs}

A natural generalisation of the family of induced-$H$-free graphs that has been extensively studied in the literature (see, for example,~\cite{Al92,AlBaBoMo11,BoTh97}), are so-called \emph{hereditary properties of graphs}, that is, properties of graphs that are closed under taking induced subgraphs. As we mentioned in the Introduction, the rough structure of a typical member of an arbitrary hereditary property of graphs was determined a few years ago by Alon, Balogh, Bollob\'as, and Morris~\cite{AlBaBoMo11}. It would be very interesting (and, most likely, extremely challenging) to obtain a corresponding statement for a typical sparse graph in a hereditary property.   

In order to give the reader an idea of what it might be possible to prove in this very general setting, let us take this opportunity to state a theorem for \emph{monotone properties of graphs} (that is, properties of graphs that are closed under taking  subgraphs) which follows easily from the container theorems proved in~\cite{BaMoSa15,SaTh15}, but, as far as we are aware, has not previously been stated explicitly in the literature.

Given a monotone property of graphs $\cP$, let $\FF(\cP)$ denote the family of minimal forbidden subgraphs, i.e., the family of all graphs that are not in $\cP$, but all of whose proper subgraphs are in $\cP$. Theorem~\ref{thm:gen-monotone-prop}, below, gives an approximate structural description of a typical member of $\cP$ with (essentially) any given order $n$ and size $m$, as long as $\FF(\cP)$ is finite. In order to state the theorem, we will need the following definition.

\begin{defn}
  Given a non-trivial monotone property of graphs $\cP$ such that $\FF(\cP)$ is finite, we define the sequence
  \[
    m(\cP) = \big( (a_1,r_1), \dotsc, (a_s,r_s) \big)
  \]
  as follows:
  \begin{enumerate}[label=(\roman*)]
  \item
    Set $a_0 = 0$ and $r_0 = \infty$.
  \item
    \label{item:mP-2}
    Let $i \ge 0$ and suppose that we have already defined $(a_i, r_i)$. Let
    \[
      a_{i+1}  := \min\big\{ m_2(H) \scolon \text{$H \in \FF(\cP)$, $m_2(H) > a_i$, and $\chi(H) \le r_i$}\big\},
    \]
    provided that the above set is not empty; otherwise, set $s = i$ and stop.
  \item
    If $a_{i+1}$ was defined in step~\ref{item:mP-2}, then let
    \[
      r_{i+1} := \min\big\{ \chi(H) - 1 \scolon \text{$H \in \FF(\cP)$ and $m_2(H) = a_{i+1}$} \big\}.
    \]
    If $r_{i+1} = 1$, then set $s = i+1$ and stop; otherwise, increase $i$ by one and go to~\ref{item:mP-2}.
  \end{enumerate}
\end{defn}

Given integers $n$ and $m$ and a graph property $\cP$, denote by $\Pnm$ the family of all graphs with vertex set $\{1, \dotsc, n\}$ and precisely $m$ edges that belong to $\cP$.

\begin{thm}
  \label{thm:gen-monotone-prop}
  Let $\cP$ be a non-trivial monotone property of graphs such that $\FF(\cP)$ is finite and let $G$ be a uniformly chosen random graph in $\Pnm$. Suppose that $m(\cP) = \big( (a_1,r_1),\ldots,(a_s,r_s) \big)$. The following holds for every $\eps > 0$:
  \begin{enumerate}[label=(\alph*)]
  \item
    \label{item:gen-prop-quasirandom}
    If $n \ll m \ll n^{2-1/a_1}$, then a.a.s.\ $G$ is $\eps$-quasirandom.
  \item
    \label{item:gen-prop-structure-1}
    If $n^{2-1/a_i} \ll m \ll n^{2-1/a_{i+1}}$ for some $i \in \{1, \dotsc, s-1\}$, then a.a.s.\ $G$ is $\eps$-close to $r_i$-partite.
  \item
    \label{item:gen-prop-structure-2}
    If $m \gg n^{2-1/a_s}$ and $r_s \ge 2$, then a.a.s.\ $G$ is $\eps$-close to $r_s$-partite.
  \end{enumerate}
\end{thm}

Since the proof of Theorem~\ref{thm:gen-monotone-prop} is a (nowadays) standard application of the container method, using (a robust version of) the stability theorem of Erd\H{o}s and Simonovits~\cite{Er67, Si68} (cf. the proof of~\cite[Theorem~1.7]{BaMoSa15}), we leave the details to the reader.

Finally, we remark that the assumption that $\FF(\cP)$ is finite is essential. Indeed, suppose that $\FF(\cP)$ contains all (minimal) non-bipartite graphs $H$ with $m_2(H) \ge a$ for a given $a > 1$. If $m \ge an$, then $\Pnm$ contains only bipartite graphs and thus if $\eps > 0$ is sufficiently small, then there are no graphs in $\cP$ that are $\eps$-quasirandom.

\section*{Acknowledgements}

The bulk of this research was conducted over the course of several visits of the second author to IMPA in Rio de Janeiro. We would like to thank IMPA for their hospitality and for nurturing a wonderful research environment.


\bibliographystyle{amsplain}
\bibliography{ind-C4-free}

\appendix

\section{The number of split graphs}
\label{sec:counting-split-graphs}

\begin{proof}[{Proof of Proposition~\ref{prop:max-Nnm-bounds}}]
  Suppose that $n \ll m \le \lambda n^2$ for some positive constant $\lambda$ whose value will be specified later. Fix an $\ell$ such that both $\Nnm(\ell)$ and $\Nnm(\ell+1)$ are nonzero and note that this implies that $0 \le m - \binom{\ell+1}{2} \le \ell(n-\ell)$. Therefore, $\ell < 2\sqrt{\lambda}n$ and
  \begin{equation}
    \label{eq:Nnm-ratio}
    \frac{\Nnm(\ell+1)}{\Nnm(\ell)} = \frac{\binom{(\ell+1)(n-\ell-1)}{m - \binom{\ell+1}{2}}}{\binom{\ell(n-\ell)}{m-\binom{\ell+1}{2}}} \cdot \frac{\binom{\ell(n-\ell)}{m-\binom{\ell+1}{2}}}{\binom{\ell(n-\ell)}{m - \binom{\ell}{2}}}.
  \end{equation}
  Denote the first and the second ratios in the right-hand side of~\eqref{eq:Nnm-ratio} by $a(\ell)$ and $b(\ell)$, respectively. In other words, let
  \[
    a(\ell) = \frac{\big((\ell+1)(n-\ell-1)\big)_{m - \binom{\ell+1}{2}}}{\big(\ell(n-\ell)\big)_{m - \binom{\ell+1}{2}}}
    \qquad \text{and} \qquad
    b(\ell)  = \frac{\left(m - \binom{\ell}{2}\right)_\ell}{\left(\ell(n-\ell)-m+\binom{\ell+1}{2}\right)_\ell},
  \]
  where $(\cdot)_k$ denotes the falling factorial, that is, $(a)_k = a!/(a-k)! = a(a-1)\ldots (a-k+1)$. Routine calculation shows that
  \[
    \left(1 + \frac{n-2\ell-1}{\ell(n-\ell)}\right)^{m-\binom{\ell+1}{2}} \le a(\ell) \le \left(1 + \frac{n-2\ell-1}{\ell(n-\ell) - m}\right)^{m-\binom{\ell+1}{2}}
  \]
  and that
  \[
    \left(\frac{m-\binom{\ell + 1}{2}}{\ell (n-\ell) - m + \binom{\ell+1}{2}}\right)^\ell \le b(\ell) \le \left(\frac{m-\binom{\ell}{2}}{\ell (n-\ell) - m + \binom{\ell}{2}}\right)^\ell.
  \]

  \medskip
  Let $\eps$ be a small positive constant. We claim that $\ellnm \le \eps \sqrt{m}$, provided that $\lambda$ is sufficiently small. Indeed, otherwise we would have
  \[
    \ellnm^2 = \frac{m}{\log(\ellnm n/m)} \le \frac{m}{\log(\eps n/ \sqrt{m})} \le \frac{m}{\log(\eps/\sqrt{\lambda})} < \eps^2 m,
  \]
  contradicting our assumption. Moreover, we claim that $m \le \eps \ellnm n$, provided that $\lambda$ is sufficiently small. Indeed, otherwise we would have
  \[
    \ellnm^2 < \frac{m}{\eps^2 n^2 / m} \le \frac{m}{\ellnm n / m} < \frac{m}{\log(\ellnm n / m)}.
  \]

  We claim that the function $\ell \mapsto \Nnm(\ell)$ is increasing on the interval $[2m/n, 3\ellnm/4]$. Indeed, suppose that $2m/n \le \ell \le 3\ellnm/4$. Then $n - 2\ell - 1 \ge (1-\eps) \cdot (n-\ell)$, whenever $\lambda$ is sufficiently small, and hence
  \begin{equation}
    \label{eq:a-ell-lower}
    a(\ell) \ge \left(1 + \frac{1-\eps}{\ell}\right)^{(1-\eps)m} \ge \exp\left((1-3\eps) \cdot \frac{m}{\ell}\right),
  \end{equation}
  as $\ell > 2m/n \gg 1$. On the other hand, since $\ell \le \ellnm \le \eps\sqrt{m}$, then
  \begin{equation}
    \label{eq:b-ell-lower}
    b(\ell) \ge \left(\frac{m-\ell^2}{\ell n}\right)^\ell \ge \left(\frac{(1-\eps^2)m}{\ell n}\right)^\ell.
  \end{equation}
  Finally, one easily checks that
  \begin{equation}
    \label{eq:ellnm-definition}
    \exp\left(\frac{\alpha m}{\ell}\right) \cdot \left(\frac{\beta m}{\ell n}\right)^\ell \ge 1 \qquad \Longleftrightarrow \qquad \ell \le \sqrt{\frac{\alpha m}{\log (\ell n/(\beta m))}},
  \end{equation}
  and thus $a(\ell) \cdot b(\ell) \ge 1$, provided that $\eps$ is sufficiently small. Indeed, our assumption on $\ell$ implies that
  \[
    \ell \le 3\ellnm/4 = \frac{3}{4} \cdot \sqrt{\frac{m}{\log(\ellnm n/m)}} \le \frac{3}{4} \cdot \sqrt{\frac{m}{\log(\ell n/m)}}.
  \]

  \medskip
  Second, we show that the function $\ell \mapsto \Nnm(\ell)$ is decreasing for $\ell \ge 3\ellnm/2$. Indeed, if $\ell \ge 3\ellnm/2$, then $m < \eps \ell n$ and hence, recalling that $\ell < 2\sqrt{\lambda}n \le \eps n$,
  \begin{equation}
    \label{eq:a-ell-upper}
    a(\ell) \le \left(1 + \frac{n}{\ell(n-\ell)-m}\right)^m \le \left(1 + \frac{n}{\ell n - 2\eps \ell n}\right)^m \le \exp\left((1+3\eps)\cdot \frac{m}{\ell}\right)
  \end{equation}
  and
  \begin{equation}
    \label{eq:b-ell-upper}
    b(\ell) \le \left(\frac{m}{\ell n -2m}\right)^\ell \le \left(\frac{m}{(1-2\eps)\ell n}\right)^\ell.
  \end{equation}
  Therefore, $a(\ell) \cdot b(\ell) < 1$, provided that $\eps$ is sufficiently small, see~\eqref{eq:ellnm-definition}.
  
  We conclude that if $m \le \lambda n^2$ for some sufficiently small positive $\lambda$, then the function $[2m/n, \infty) \ni \ell \mapsto \Nnm(\ell)$ attains its maximum value for some $\ell$ satisfying $3\ellnm/4 < \ell < 3\ellnm/2$. In particular, in order to complete the proof of the first assertion of the proposition, it suffices to check that $\Nnm(\ell) < \Nnm(\ellnm)$ whenever $\ell \le 2m/n$. To this end, first note that if $\ell < 2m/n$, then
  \[
    \Nnm(\ell) \le 2^{\ell(n-\ell)} \le 2^{2m},
  \]
  whereas
  \[
    \Nnm(\ellnm) \ge \binom{(1-\eps)\ellnm n}{(1-\eps)m} \ge \binom{(1-\eps)m/\eps}{(1-\eps)m} \ge 5^m,
  \]
  provided that $\eps$ is sufficiently small.
  
  Finally, we establish the remaining `large deviation' assertions of the proposition. To this end, we first note that if $m/n \le \ell \le 3\ellnm/4$, then
  \[
    \exp\left(\frac{m}{\ell}\right) \cdot \left(\frac{m}{\ell n}\right)^\ell \ge \exp\left(\frac{4m}{3\ellnm}\right) \cdot \left(\frac{m}{\ellnm n}\right)^{\ellnm} = \exp\left(\frac{m}{3\ellnm}\right).
  \]
  Consequently, if $\eps$ is sufficiently small, then by~\eqref{eq:a-ell-lower} and~\eqref{eq:b-ell-lower},
  \begin{equation}
    \label{eq:max-Nnm-ell-lower-tail}
    \frac{\Nnm(3\ellnm/4)}{\Nnm(\ellnm/2)} \ge \prod_{\ell=\ellnm/2}^{3\ellnm/4} a(\ell)b(\ell) \ge \exp\left(\frac{4}{5} \cdot \frac{m}{3\ellnm} \cdot \frac{\ellnm}{4}\right) = \exp\left(\frac{m}{15} \right).  
    \end{equation}
  Similarly, if $\ell \ge 3\ellnm / 2 \ge m/n$, then
  \[
    \exp\left(\frac{m}{\ell}\right) \cdot \left(\frac{m}{\ell n}\right)^\ell \le \exp\left(\frac{2m}{3\ellnm}\right) \cdot \left(\frac{m}{\ellnm n}\right)^{\ellnm} = \exp\left(-\frac{m}{3\ellnm}\right)
  \]
  and consequently if $\eps$ is sufficiently small, then by~\eqref{eq:a-ell-upper} and~\eqref{eq:b-ell-upper}, 
  \begin{equation}
    \label{eq:max-Nnm-ell-upper-tail}
    \frac{\Nnm(2\ellnm)}{\Nnm(3\ellnm/2)} \le \prod_{\ell=3\ellnm/2}^{2\ellnm} a(\ell)b(\ell) \le \exp\left(-\frac{3}{4} \cdot \frac{m}{3\ellnm} \cdot \frac{\ellnm}{2}\right) = \exp\left(-\frac{m}{8}\right).
  \end{equation}
  Since the function $\ell \mapsto \Nnm(\ell)$ is increasing when $2m/n \le \ell \le \ellnm/2$ and decreasing when $\ell \ge 2\ellnm$, and since
  \[
    \Nnm(\ell) \le 2^{2m} \le e^{-m/15} \cdot 5^m \le e^{-m/15} \cdot \Nnm(\ellnm)
  \]
  for all $\ell < 2m/n$, we conclude from~\eqref{eq:max-Nnm-ell-lower-tail} and~\eqref{eq:max-Nnm-ell-upper-tail} that if $\ell \le \ellnm/2$ or $\ell \ge 2\ellnm$, then
  \[
    \Nnm(\ell) \le \exp(-m/15) \cdot \max_\ell \Nnm(\ell),
  \]
  as claimed.
\end{proof}

\end{document}